\documentclass[11pt]{article}
\pdfoutput=1
\usepackage[utf8]{inputenc}
\usepackage{amsmath,amsthm}
\usepackage{amsfonts}
\usepackage{amssymb}
\usepackage{array}
\usepackage{bm}
\usepackage{multirow}
\usepackage{subfig}
\usepackage[inline,shortlabels]{enumitem}
\usepackage[normalem]{ulem}
\usepackage[font=small]{caption}
\usepackage{graphicx}
\usepackage{setspace}
\usepackage{xspace}
\usepackage{mathtools}

\usepackage[mathscr]{euscript}

\usepackage{tikz} 
\usetikzlibrary{arrows,matrix,positioning,fit,calc}

\usepackage[margin=2.9cm]{geometry}
\usepackage[colorlinks,citecolor=blue,urlcolor=blue]{hyperref}

\usepackage{algorithm}
\usepackage{algpseudocode}

\algnewcommand{\IIf}[1]{\State\algorithmicif\ #1\ \algorithmicthen}

\usepackage[
style=alphabetic,
backend=bibtex,
doi=false,
isbn=false,
url=false,
firstinits=true,
maxnames=4
]{biblatex}

\makeatletter
\def\blx@maxline{77}
\makeatother

\renewbibmacro{in:}{}           
\DeclareFieldFormat[article,inbook,incollection,inproceedings,patent,thesis,unpublished]{citetitle}{#1}
\DeclareFieldFormat[article,inbook,incollection,inproceedings,patent,thesis,unpublished]{title}{#1} 

\usepackage{arash_macros}
\usepackage{authblk}
\title{Optimal Bipartite Network Clustering}
\author{Zhixin Zhou and Arash A. Amini}

\usepackage{arash_macros}
\DeclareMathOperator{\pmf}{pmf}

\newcommand{\Bc}{\mathcal B}
\newcommand{\Lc}{\mathcal L}

\newcommand{\y}{y} 
\newcommand{\yt}{\tilde{y}} 
\newcommand{\z}{z} 
\newcommand{\zt}{\tilde{z}} 

\newcommand{\nc}{m}
\newcommand{\nr}{n}
\newcommand{\Kc}{L}
\newcommand{\Kr}{K}

\newcommand{\bb}{\bm{b}}

\newcommand{\Lamh}{\hat \Lambda}
\newcommand{\lamh}{\hat \lambda}
\newcommand{\pih}{\hat \pi}
\DeclareMathOperator{\poi}{Poi}

\newcommand{\Fc}{\mathcal F}
\newcommand{\onev}{\bm 1}
\newcommand{\yh}{\widehat y}

\newcommand{\U}{U}
\newcommand{\V}{V}

\newcommand{\Lamin}{\Lambda_{\min}}

\newcommand{\lamt}{\tilde{\lambda}}
\providecommand{\I}{I}
\newcommand\Nn{N}

\usepackage{mathrsfs}

\newcommand{\f}{\frac}
\newcommand{\E}{\mathbb E}

\newcommand{\summ}[2]{\sum_{#1 = 1}^{#2}}

\renewcommand{\P}{\mathbb P}

\newcommand{\Z}{\mathbb Z}
\newcommand{\cvecc}[2]{ [ {#1} \,;\, {#2} ]}  
\newcommand{\rvecc}[2]{ [ {#1} \,\, {#2} ]}  
\DeclareMathOperator{\mis}{Mis}
\DeclareMathOperator{\dmis}{dMis}
\newcommand{\lp}{\Big(}
\newcommand{\rp}{\Big)}

\newcommand{\zh}{\widehat\z}

\newcommand{\Are}{A_{\text{re}}}

\newcommand{\Lamt}{\tilde\Lambda}
\newcommand{\Laminf}{\|\Lambda\|_\infty}
\newcommand{\Gaminf}{\|\Gamma\|_\infty}

\newcommand{\It}{\tilde I}

\newcommand{\Imin}{\I_{\text{min}}}

\DeclareMathOperator{\rowSC}{rowSC}
\DeclareMathOperator{\colSC}{colSC}

\newcommand{\bottom}{\text{bottom}}
\renewcommand{\top}{\text{top}}

\newcommand{\beq}{\begin{eqnarray*}}
	\newcommand{\eeq}{\end{eqnarray*}}
\newcommand{\bal}{\begin{align}}
\newcommand{\eal}{\end{align}}

\newcommand{\pit}{\tilde\pi}

\newcommand{\Lammin}{\Lambda_{\min}}
\newcommand{\Gammin}{\Gamma_{\min}}

\newcommand{\vertiii}[1]{{\left\vert\kern-0.25ex\left\vert\kern-0.25ex\left\vert #1 
		\right\vert\kern-0.25ex\right\vert\kern-0.25ex\right\vert}}

\DeclareMathOperator\sbm{SBM}
\DeclareMathOperator\psbm{pSBM}
\newcommand{\Qdist}{\mathbb{Q}}
\newcommand\Perm{\Pi}

\newcommand{\match}[0]{{\textsc{Match}}\xspace}
\newcommand{\Grow}{G^{\text{row}}}
\newcommand{\Gcol}{G^{\text{col}}}

\newcommand{\Blam}[2]{\mathscr B_{#1}(#2)}
\newcommand{\Bcl}{\mathscr B_\Lambda}

\newcommand\Ac{\mathcal A}
\newcommand{\Sh}{\widehat S}
\newcommand{\Zh}{\widehat Z}
\newcommand{\Yh}{\widehat Y}

\newcommand{\lr}{{l}}
\newcommand{\pe}[1]{P_{e,#1}}
\newcommand{\sopt}{s^*}
\newcommand\Iopt{I^*}

\newcommand{\Ppmf}{\Phi}
\newcommand{\Pbpmf}{\widetilde{\Phi}}
\newcommand{\ppmf}{\varphi}
\newcommand{\pbpmf}{\widetilde{\varphi}}

\newcommand{\poillr}{\Psi}  
\newcommand{\Sc}{\mathcal{S}}
\newcommand{\pla}{{\scshape PLA}\xspace}
\newcommand{\Lop}{\mathscr L}
\newcommand{\lrc}{{\scshape LRC}\xspace}
\DeclareMathOperator{\LR}{LR}
\newcommand{\J}{J}
\newcommand{\Q}{Q}
\newcommand{\ptail}{p}

\newcommand{\Af}[1]{\mathfrak{A}_{#1}}
\newcommand{\Bf}[1]{\mathfrak{B}_{#1}}

\newcommand{\Pf}{\mathfrak{P}}
\newcommand{\Df}{\mathfrak{D}}
\newcommand{\Lamhsix}{\Lamh_{\text{step\!~\ref{step:local:Lambda:1}}}}
\newcommand{\Lamhnine}{\Lamh_{\text{step\!~\ref{step:local:Lambda:2}}}}
\newcommand{\Lamhten}{\Lamh_{\text{step\!~\ref{step:global:Lambda}}}}
\newcommand{\ztfive}{\zt_{\text{step\!~\ref{step:consist}}}}
\newcommand{\ytfive}{\yt_{\text{step\!~\ref{step:consist}}}}
\newcommand{\delsix}{\delta_1}
\newcommand{\delnine}{\delta_2}
\newcommand{\etaseven}{\eta^{\text{step\!~\ref{step:row:LR:1}}}}
\newcommand{\etaeight}{\eta^{\text{step\!~\ref{step:similar:z:1}}}}
	
\newcommand{\etaeleven}[1]{\eta_{#1}^{\text{step\!~\ref{step:row:LR:2}}}}
\newcommand{\ytseven}{\yt_{\text{step\!~\ref{step:row:LR:1}}}}
\newcommand{\gamrow}{\gamma^{\text{row}}}
\newcommand{\gamcol}{\gamma^{\text{col}}}
\newcommand{\zteight}{\zt_{\text{step\!~\ref{step:similar:z:1}}}}
\newcommand{\beleven}{b^{\text{step\!~\ref{step:row:LR:2}}}}

\newcommand{\taurow}{\tau^{\text{row}}}
\newcommand{\taucol}{\tau^{\text{col}}}
\newcommand{\delrow}{\delta^{\text{row}}}
\newcommand{\delcol}{\delta^{\text{col}}}
\newcommand{\delsixcol}{\delsix^{\text{col}}}

\newcommand{\ucol}{u^{\text{col}}}
\newcommand{\Imincol}{\Imin^{\text{col}}}

\renewcommand{\bot}{\text{bottom}}
\newcommand\Top{\text{Top}}

\newcommand\etap{\eta'}
\newcommand\Gamin{\Gamma_{\min}}
\newcommand\ptailp{\ptail'}
\newcommand{\Icol}{I^{\text{col}}}

\newcommand{\scerr}[0]{{\scshape SC-RRE}\xspace}
\newcommand{\Lamwedge}{\Lambda_{\wedge}}
\newcommand{\Sigh}{\hat{\Sigma}}
\newcommand{\kalg}{\mathscr{K}}
\newcommand{\alp}{\alpha}

\newcommand{\epsi}{\varepsilon}
\newcommand{\epsikr}{\epsi_{kr}}

\DeclareMathOperator{\id}{id}
\DeclareMathOperator{\Bin}{Bin}

\newcommand\dav{d_{\text{av}}}
\newcommand{\les}{\,\le\,}
\newcommand{\ges}{\,\ge\,}
\newcommand{\eqs}{\,=\,}
\newcommand\Pt{\widetilde{P}}
\newcommand\lambdat{\widetilde{\lambda}}

\DeclareMathOperator\lab{lab}
\DeclareMathOperator*\avg{avg}
\newcommand\Ss{\mathcal S}
\newcommand\Ts{\mathcal T}
\newcommand{\psumerr}{P_{\text{Err},+}}

\begin{document}
	
	\maketitle
	
	\begin{abstract}
		We study bipartite community detection in networks, or more generally the network  biclustering problem. We present a fast two-stage procedure based on spectral initialization followed by the application of a pseudo-likelihood  classifier twice. Under mild regularity conditions, we establish the weak consistency of the procedure (i.e., the convergence of the misclassification rate to zero) under a general bipartite stochastic block model. We show that the procedure is optimal in the sense that it achieves the optimal convergence rate that is achievable by a biclustering oracle, adaptively over the whole class, up to constants. This is further formalized by deriving a minimax lower bound over a class of biclustering problems.  The optimal rate we obtain sharpens some of the existing results and generalizes others to a wide regime of  average degree growth, from sparse networks with average degrees growing arbitrarily slowly  to fairly dense networks with average degrees of order $\sqrt{n}$. As a special case, we recover the known exact recovery threshold in the $\log n$ regime of sparsity. To obtain the consistency result, as part of the provable version of the algorithm, we introduce a sub-block partitioning scheme that is also computationally attractive, allowing for distributed implementation of the algorithm without sacrificing optimality. The provable  algorithm is derived from a general class of pseudo-likelihood biclustering algorithms that employ simple EM type updates. We show the effectiveness of this general class  by numerical simulations.
	\end{abstract}
	
	
	\section{Introduction}
	
	Network analysis has become an active area of research over the past few years, with applications and contributions from many disciplines including statistics, computer science, physics, biology and social sciences. 
	%
	A fundamental problem in network analysis is detecting and identifying communities, also known as clusters, to help better understand the underlying structure of the network. The problem has seen
	rapid advances in recent years with numerous  breakthroughs in modeling, theoretical understanding, and practical
	applications~\cite{fortunato2016community}. In particular, there has been much excitement and progress in understanding and analyzing the stochastic block model (SBM) and its variants. We refer to~\cite{abbe2017community} for a recent survey of the field. Much of this effort, especially on the theoretical side has been focused on the univariate (or symmetric) case, while the bipartite counterpart, despite numerous practical applications, has received comparatively less attention. Of course, there has been lots of activity in terms of modeling and algorithm development for bipartite clustering both in the context of networks~\cite{Zhou2007,Larremore2014,Wyse2014,Rohe2015,Razaee2017} as well as other domains such as topic modeling and text mining~\cite{Dhillon2001, Dhillon2003} and biological applications~\cite{cheng2000biclustering, Madeira2010}. But much of this work either lacks theoretical investigations 
	or has not considered the issue of statistical optimality.
	
	
	In this paper, we consider 
	the community detection, or clustering, in the  bipartite setting with a focus on deriving fundamental theoretical limits of the problem.  The main goal is to propose computationally feasible algorithms for bipartite network clustering that exhibit provable statistical optimality.
	We will focus on the bipartite version of the SBM which is a natural model for bipartite networks with clusters.  SBM is a stochastic network model where the probability of edge formation depends on the latent (unobserved) community assignment of the nodes, often referred to as node labels. The goal of the community detection problem is to recover these labels given an instance of the network. This is a non-trivial task since, for example, maximum likelihood estimation involves a search over exponentially many labels.
	
	Community detection in bipartite SBM is closely related to the biclustering problem, for which many algorithms have been developed over the years~\cite{hartigan1972direct,cheng2000biclustering, tanay2002discovering, gao2016optimal}. On the other hand, in recent years, various algorithms have been proposed for clustering in univariate SBMs, including global approaches such as spectral clustering~\cite{ rohe2011spectral, krzakala2013spectral, lei2013consistency, fishkind2013consistent, vu2014simple, massoulie2014community, yun2014accurate, chin2015stochastic, bordenave2015non, gulikers2017spectral, pensky2017spectral} and convex relaxations via  semidefinite programs (SDPs)~\cite{amini2018semidefinite, hajek2016achieving, bandeira2015random, guedon2016community, montanari2016semidefinite, ricci2016performance, agarwal2017multisection, perry2017semidefinite}, as well as local methods such as belief propagation~\cite{decelle2011asymptotic}, Bayesian MCMC~\cite{nowicki2001estimation} and variational Bayes~\cite{celisse2012consistency,bickel2013asymptotic}, greedy profile likelihood~\cite{bickel2009nonparametric,zhao2012consistency} and pseudo-likelihood maximization~\cite{amini2013pseudo}, among others. A limitation of spectral clustering approaches is that they are often not optimal on their own, and the SDPs have the drawback of not being able to fit the full generality of SBMs. 
	
	Various algorithms can further improve the clustering accuracy, and adapt to the generality of SBM. Profile likelihood maximization was proposed and analyzed in~\cite{bickel2009nonparametric}, but the underlying optimization problem is computationally infeasible and the approach only applicable to networks of limited size.  Pseudo-likelihood ideas were used in~\cite{amini2013pseudo} to derive EM type updates to maximize a surrogate to the likelihood of the SBM.
	We extend the ideas of \cite{amini2013pseudo} to the bipartite settings and greatly improve their analysis by showing that these pseudo-likelihood approaches can achieve minimax optimal rates in a wide variety of settings.
	%
	
	In the unipartite setting, there has been interesting recent advancements in understanding optimal recovery 
	in the semi-sparse regime where the (expected) average network degree $\dav$ is allowed to grow to infinity but rather slowly, as the number of nodes $n$ increases to infinity. In a series of papers~\cite{mossel2015consistency,abbe2016exact,hajek2016achieving,hajek2016achievingExt} the thresholds for optimal exact recovery, also known as strong consistency, were established in the context of simple planted partition models. 
	In~\cite{abbe2015community}, the problem of strong consistency was considered for a general SBM and the optimal threshold for strong consistency was established. In subsequent work~\cite{zhang2016minimax,gao2017achieving,gao2018community}, the results were extended to include weak consistency, i.e., requiring the fraction of misclassified nodes to go to zero, rather than drop to exactly zero (as in strong consistency), and rates of optimal convergence where established, up to a slack in the exponent. To achieve the more relaxed consistency results, \cite{gao2017achieving} limited the model to what we refer to as  strongly assortative SBM; see~\cite{amini2018semidefinite} for a definition.
	
	Our work  is inspired by the insightful analysis of~\cite{abbe2015community} and~\cite{gao2017achieving}. We extend these ideas by presenting results that are strictly sharper and more general that what has been obtained so far.  In short, we only assume that the clusters are distinguishable (in the sense of Chernoff divergence) and the network is not very dense, i.e. $\dav = O(\sqrt{n})$ where $\dav$ denotes the average expected degree, and $n$ is the number of nodes.  Our  results establish minimax optimal rates  below 
	this $\sqrt{n}$ regime and above the sparse regime $\dav = O(1)$. In particular, we obtain precise rates of (weak) consistency when $\dav$ grows arbitrarily slowly.
	%
	We require $\dav = O(\sqrt{n})$ to allow for Poisson approximations on the degrees of nodes restricted to large subsets. The regimes denser than $\sqrt n$ can obviously achieve exact recovery and hence not interesting from a theoretical standpoint. 
	%
	%
	We make more detailed comparisons with 
	existing work in Section~\ref{main:res}.
	
	\paragraph{Contributions.} Establishing these results require a fair amount of technical and algorithmic novelty.
	Here, we highlight some of these features:
	
	\begin{enumerate}[1., wide, labelwidth=!, itemsep=.5pt, topsep=2pt]
		\item Existing minimax rates of convergence for the misclassification error are known for what we refer to as the \emph{nearly assortative model} where the probability of connection is $\ge a/n$ within clusters and $\le b/n$ outside clusters.  The existing results establish an error rate that belongs to an interval: \[ \text{Error} \in \big[e^{-(1+o(1)) I}, e^{-(1-o(1)) I} \big], \quad \text{as $I \to \infty$}, \]  for some $o(1)$ 
		terms that are positive and where $I$ is related to the Bhattacharyya distance (also known as the Hellinger affinity) of Bernoulli variables with probabilities $a/n$ and $b/n$. This type of result  originally appeared in~\cite{zhang2016minimax} and propagated to many subsequent works~\cite{gao2018community, gao2017achieving, zhang2017theoretical, xu2017optimal, chien2018community}.  This rate, however, is not sharp since the slack term $e^{o(1) I}$ could be unbounded (because $o(1) > 0$ and $ I \to \infty$). Another shortcoming of these results are their limitations to the simple nearly assortative setting. Our results sharpens and generalizes this known minimax rate to
		\[ \text{Error} = e^{-I-R}, \quad \text{for some $ 0 < R \asymp \log I$}\]
		for the general class of all SBMs under a mild distinguishably assumption on the rows (and columns) of the edge probability matrix. Furthermore, the $I$ in our result takes the form of a Chernoff exponent among Poisson vectors, which is the form necessary for the general SBM. 
		\item In order to achieve these sharp rates, we introduce an efficient sub-block (or sub-graph) partitioning scheme 
		which generalizes the partitioning idea of \cite{chin2015stochastic}.
		Our  partitioning scheme allows one to break down the costly  spectral initialization, by applying it to smaller subblocks, without losing optimality. If done in parallel, spectral clustering on subblocks will be computationally cheaper than performing a spectral decomposition of the entire matrix.
		The resulting algorithm  is  naturally parallelizable, hence can be deployed in a \emph{distributed fashion} allowing it to scale to very large networks.
		
		\item Our algorithms being extensions of those in~\cite{amini2013pseudo}, are modifications of a natural EM algorithm on mixtures of Poisson vectors, hence very familiar from a  statistical perspective. 
		Although other (optimal) algorithms in the literature are more or less preforming similar operations, the link to EM algorithms and mixture modeling is quite clear in our work. We provide in Section~\ref{sec:gen:pl:alg} the general blueprint of the algorithms based on the pseudo-likelihood idea  and block compression (Algorithm~\ref{alg:pl:biclust}). We then show how a provable version can be constructed by combining with the sub-block partitioning ideas in Section~\ref{sec:provable:alg}. 
		
		\item In order to get the sharper rate, analyzing a single step of an EM type algorithm is not enough, and thus we analyze the second step as well. We will show that the first step gets us from a good (but crude) initial rate $\gamma_1$ to the fast rate $\approx \exp(-I/Q)$ where $Q$ is the number of subblocks. This rate is in the vicinity of the optimal rate and  repeating the iteration once more, with the more accurate labels 
		gets us to the minimax error rate $\exp(-I - R)$. 
		
		\item Among the technical contributions are a uniform consistency result (Lemma~\ref{lem:unif:param}) for the likelihood ratio classifier (LRC) over a subset of the parameters close to the truth, sharp approximations for the Poisson-binomial distributions (Section~\ref{sec:approx:lemmas}), and extension (and elucidation) of a novel technique of~\cite{abbe2015community} in deriving error exponents for general exponential families (cf. Section~\ref{sec:err:exponent}). The uniform consistency result for LRCs lets us tolerate some degree of dependence among the statistics from iteration to iteration (Sections~\ref{sec:provable:alg} and~\ref{sec:fixed:label:analysis}).
		
		\item The bipartite clustering setup (as opposed to the  symmetric unipartite case) allows us to introduce an oracle version of the problem which helps in understanding the nature of the optimal rates  in community detection and their relation to the classical hypothesis testing and mixuture modeling. We try to answer the curious question of why or how the Chernoff exponent of a  (binary)  hypothesis testing problem controls the misclassification rate in community detection and network clustering. The oracle also provides a lower bound on the performance of any algorithm. See Section~\ref{prop:HT:err:rates} and Proposition~\ref{prop:HT:err:rates} for details.
		
	\end{enumerate}

	The rest of the paper is organized as follows. We introduce the model and the biclustering oracle in Section~\ref{sec:network:biclustering}, and then present our main results in Section~\ref{main:res}, including an upper bound on the error rate of the algorithm and a matching minimax lower bound. The general pseduo-likelihood algorithms are presented  in Section~\ref{sec:pl:approach} and a provable version in Section~\ref{sec:provable:alg}. 
	In Section~\ref{sec:sims}, we demonstrate  the numerical performance of the methods.  The proofs of the results will appear in Sections~\ref{sec:preliminary:analysis},~\ref{sec:analysis:alg:three},~\ref{sec:proof:other:main:res},~\ref{sec:proof:minimax},~\ref{sec:proof:main:lemmas} which are organized as a supplement and appear after the references. 
	Extra comments on the results and proof techniques can be found in Section~\ref{sec:add:comments} of this supplement.

	\section{Network biclustering} 
	\label{sec:network:biclustering}
	We start by introducing the network biclustering problem based stochastic block modeling, and set up some notation.
	We then discuss how a biclustering oracle with side information can optimally recover the labels. These ideas will be the basis for our algorithms. 
	
	\subsection{Bipartite block model}\label{sec:bi:sbm}
	We will be working with a bipartite network which can be represented by a biadjacency matrix $A \in \{0,1\}^{\nr \times \nc}$, where for simplicity we assume that the nodes on the two sides are indexed by the sets $[\nr]$ and $[\nc]$. We assume that there are $K$ and $L$ communities for the two sides respectively, and the membership of the nodes to these communities are given by two vectors $\y = (\y_i) \in [K]^\nr$ and $\z = (\z_j) \in [L]^\nc$. Thus, $\y_i = k$ if node $i$ on side 1 belongs to community $k \in [K]$. We call $\y_i$ and $\z_j$ the labels of nodes $i$ and $j$ respectively. We often treat these labels as binary vectors as well, using the identification $[K] \simeq \{0,1\}^K$ via the one-hot encoding, that is $\y_i = k \iff y_{ik} = 1,\; y_{ik'} =0,\, k'\neq k$.
	
	Given the labels $\y$ and $\z$, and a \emph{connectivity} matrix $P \in [0,1]^{K \times L}$ (also known as the edge probability matrix), the general bipartite stochastic block model (biSBM) assumes that: $A_{ij}$ are Bernoulli variables, independent over $(i,j) \in [\nr] \times [\nc]$ with mean parameters,
	\begin{align}\label{eq:sbm:mean:def}
	\ex[A_{ij}] = \y_i^T P \z_j = P_{k\ell}, \quad \text{if}\; \y_i = k,\; \z_j = \ell.
	\end{align}
	We  denote this model compactly as $A \sim \sbm(\y,\z,P)$. 
	It is  helpful to consider the Poisson version of the model as well which is denoted as $A \sim \psbm(\y,\z,P)$. This is the same model as the Bernoulli SBM, with the exception that each entry $A_{ij}$ is drawn (independently) from a Poisson variate with mean given in~\eqref{eq:sbm:mean:def}. These two models behave very closely when the entries of $P$ are small enough. Throughout, we treat $\y$, $\z$ and $P$ as unknown deterministic parameters. The goal of network biclustering is to recover these three parameters given an instance of $A$.
	
	\smallskip
	In fact, as we will see, 
	the parameters $P$ themselves are not that important. What matters is the set of (Poisson) \emph{mean parameters} which are derived from $P$ and the sizes of the communities. In order to define these parameters, let $n(\z) = (n_1(\z),\dots,n_L(\z)) \in \nats^L$, be the number of nodes in each of the communities of side 2. That is, $n_\ell(\z) = \sum_{j=1}^M 1\{\z_j = \ell\} = \sum_{j=1}^M \z_{j\ell}$. We also let $\pi_\ell(\z)=n_\ell(\z)/m$ be the proportion of nodes in the $\ell$th community of side 2. Similar notations, namely $n(\y) \in \nats^K$ and $\pi(y)\in [0,1]^K$, denote the community sizes  and proportions of side 1. The \emph{row mean parameters} are defined as
	\begin{align}\label{eq:tru:row:mean:def}
	\Lambda = (\lambda_{k\ell}) = (P_{k\ell} \,n_\ell(z)) = P \diag(n(z)) \in \reals^{K \times L}
	\end{align}
	where $\diag(v)$ for a vector $v = (v_k)$ is a diagonal matrix with diagonal entries $v_k$. The column mean parameters can be defined similarly, 
	\begin{align}\label{eq:tru:col:mean:def}
	\Gamma^{T} = \big(n_k(y) P_{k\ell}\big) =  \diag(n(\y)) P \in \reals^{K \times L}.
	\end{align}
	Note the transpose in the above definition, i.e., $\Gamma \in \reals^{L \times K}$, and this convention allows us to define information measures based on rows of matrices $\Lambda$ and $\Gamma$ in a similar fashion, as will be discussed in Section~\ref{main:res}. Although the rates we derive are controlled by the Poison parameters defined above, we always assume that the true distribution is the Bernoulli SBM and any Poisson approximation will be carefully derived.
	%

	\subsection{Biclustering oracle with side information}\label{sec:oracle}
	The key idea behind the algorithms, 
	as well as our consistency arguments is the following  observation: Assume that we have prior knowledge of $P$ and the column labels $\z$, but not the row labels $\y$. For each row, we can sum the columns of $A$ according to their column memberships, i.e., we can perform the (ideal) \emph{block compression} $b^*_{i\ell} := \sum_{j} A_{ij} z_{j\ell}$. The vector $b_{i*}^* = (b_{i1}^*,\dots,b_{iL}^*)$ contains the same information for recovering the community of $i$, as the original matrix $A$---i.e., it is a sufficient statistic. Assume that we are under the $\psbm$ model (i.e., the Poisson SBM). Then, $b_{i*}^*$ has the distribution of a vector of independent Poisson variables. More precisely, 
	\begin{align}\label{eq:Qk:def}
	b_{i*}^* \sim \Qdist_k := \prod_{\ell=1}^L \poi(\lambda_{k \ell}), \quad 
	\text{if}, \quad \y_i = k,
	\end{align}
	where $\lambda_{k\ell}$ are the row mean parameters defined in~\eqref{eq:tru:row:mean:def}. Note that the distributions $\Qdist_k,\,k=1,\dots,K$ are known under our simplifying assumptions. The problem of determining the row labels thus reduces to deciding from which of these $K$ known distributions it comes from. Whether node $i$ belongs to a particular community $k$ can be decided optimally by performing a likelihood ratio (LR) test of $\Qdist_k$ against each of $\Qdist_r, \, r\neq k$.
	
	The above LR test is the heart of the algorithms discussed in Sections~\ref{sec:pl:approach} and~\ref{sec:provable:alg}. The difficulty of the biclustering problem (relative to a simple mixture modeling) is that in practice, we do not know  in advance either $y$ or $\Lambda$---hence neither the exact test statistics $(b_{i*}^*)$ nor the distributions $\{\Qdist_k\}$ are known. We thus proceed by a natural iterative procedure: Based on the initial estimates of $y$ and $z$, we obtain estimates of $(b_{i*}^*)$ and $\{\Qdist_k\}$, perform the approximate LR test to obtain better estimates of $z$, and then repeat the procedure over the columns to obtain better estimates of $y$. These new label estimates lead to better estimates of $(b_{i*}^*)$ and $\{\Qdist_k\}$, and we can repeat the process.
	
	We refer to the algorithm that has access to the true column labels $\z$ and parameters $\Lambda$, and performs the optimal LR tests, as the \emph{oracle classifier}. Note that the performance of this oracle gives a lower bound on the performance of any biclustering algorithm in our model. The performance of the oracle in turn is controlled by the error exponent of the simple hypothesis testing problems $\Qdist_k$ versus $\Qdist_{r}, r \neq k$,  as detailed in Proposition~\ref{prop:HT:err:rates}. This line of reasoning reveals the origin of the information quantities $I_{kr}$ and $\Icol_{\ell r}$---defined in ~\eqref{eq:Info:def} and~\eqref{eq:Icol:def}---that control the optimal rate of the biclustering problem. Note that the bipartite setup has the advantage of disentangling the row and column labels, so that a non-trivial oracle exists. It does not make much sense to assume known column labels  in the unipartite SBM, since by symmetry we then know the row labels as well, hence nothing left to estimate. On the other hand, due to the close relation between the bipartite and unipartite problems, the above argument also sheds light on why the error exponent of a hypothesis test is the key factor controlling optimal misclassification rates of community detection in unipartite SBM.

	\subsection{Notation on misclassification rates}\label{sec:notation}
	
	Let $\Perm_n$ the set of permutations on $[n]$. The (average) misclassification rate between two sets of (column) labels $\yh$ and $\y$  is given by
	\begin{align}\label{eq:Mis:def}
	\mis(\yh,\y ) := \min_{\sigma \,\in\, \Perm_\nr} \frac1{\nr}\sum_{i=1}^\nr 1\big\{ \sigma(\yh_i)\ne \y_i \big\}.
	\end{align}
	Letting $\sigma^*$ be a minimizer in \eqref{eq:Mis:def},  the misclassification rate over cluster $k$ is
	\begin{align}\label{eq:Mis:k:def}
	\mis_k(\yh,\y) := \f 1{n_k(\y)}\sum_{i:\y_i=k} 1\big\{\sigma^*(\yh_i)\ne \y_i\big\} = \frac{|i:\, \sigma^*(\yh_i) \neq k,\, y_i = k|}{n_k(y)},
	\end{align}
	using the cardinality notation to be discussed shortly. Note that~\eqref{eq:Mis:k:def} is not symmetric in its arguments.
	We will also use the notation $\sigma^*(\yh \to \y)$ to denote an optimal permutation in~\eqref{eq:Mis:def}. When $\mis(\yh,\y)$ is sufficiently small, this optimal permutation will be unique. It is also useful to define the \emph{direct misclassification rate} between $\yh$ and $\y$, denoted as $\dmis(\yh,\y)$,  which is obtained by setting the permutation in~\eqref{eq:Mis:def} to  the identity. In other words, $\dmis(\yh,\y)$ is the normalized Hamming distance between $\yh$ and $\y$. With $\sigma^* = \sigma^*(\yh \to \y)$, we have $\mis(\yh,\y) = \dmis(\sigma^*(\yh),\y)$. We note that 
	\begin{align}\label{eq:Mis:Mis:k}
	\mis(\yh,\y) \,=\, \sum_{k\in[\Kr]} \pi_k(\y)\, \mis_k(\yh,\y) \,\le\, \max_{k\in[\Kr]}\mis_k(\yh,\y),
	\end{align}
	as well as $\max_{k\in\Kr}\mis_k(\yh,\y) \le \mis(\yh,\y) / \min_{k'} \pi_{k'}(\y)$.
	We can similarly define the misclassification rate of an estimate $\zh$ relative to $\z$. Our goal is to derive  efficient algorithms to obtain $\yh$ and $\zh$ that have minimal misclassification rates asymptotically (as the number of nodes grow). 

	
	\paragraph{Other notation.} We write w.h.p. as an abbreviation for ``with high probability'', meaning that the event holds with probability $1-o(1)$. To avoid ambiguity, we assume all parameters, including $\nc$, are functions of $\nr$. All limits and little o notations are under  $n\to \infty$. For example, $f(n)=o(g(n))$ denotes $\lim_{n\to\infty} {f(n)}/{g(n)}=0$. We write $\Z_Q = \Z / Q\Z$ to denote a cyclic group of order $Q$.  Our convention regarding solutions of optimization problems, whenever more than one exist is to choose one uniformly at random. We use the shorthand notation $|i: y_i = k| := | \{i: y_i = k\}|$ for cardinality of sets, where $i \in [n]$ is implicit, assuming the $\y$ is a vector of length $n$. For example, if $\yh, \y \in [\Kr]^n$, we have the identity $|i:\; \yh_i\neq \y_i| = \sum_{k \in [K]} |i:\; y_i = k,\, \yh_i \neq k|$.
	%
	It is worth noting that we use \emph{community} and \emph{cluster} interchangeably in this paper, although some authors prefer to use community for the assortative clusters, and use ``cluster'' to refer to any general group of nodes. We will not follow this convention and no assortativity will be implicitly assumed. 
	
	\section{Main results}\label{main:res}
	
	Let us start with some assumptions on the mean parameters. Recall the row and column mean parameter matrices $\Lambda$ and $\Gamma$ defined in~\eqref{eq:tru:row:mean:def} and~\eqref{eq:tru:col:mean:def}. 
	Let $\Lammin$ and $\infnorm{\Lambda}$ be the minimum and maximum value of the entries of $\Lambda$, respectively, and similarly for $\Gamma$. We assume
	\begin{align}\label{assump:Lambda}
	\frac{\infnorm{\Lambda}}{\Lammin} \vee \frac{\infnorm{\Gamma}}{\Gammin} \;\le\; \omega, \tag{A1}
	\end{align}
	for some $\omega > 0$. That is, $\omega$  measures the deviation of the entries of the mean matrices from  uniform.
	We  assume that the sizes of the clusters are bounded as
	\begin{align}\label{assump:balance}
	\f 1{\beta\Kr}\le \pi_k(\y) \le \f \beta\Kr\quad\text{and}\quad\f 1{\beta\Kc}\le \pi_\ell(\z) \le \f \beta\Kc\tag{A2}
	\end{align}
	for all $k \in [\Kr]$ and $\ell \in [\Kc]$. We will assume \eqref{assump:Lambda} and \eqref{assump:balance} throughout the paper.
	%
	The following key quantity controls the misclassification rate: 
	\begin{align}\label{eq:Info:def}
	I_{kr} := I_{kr}(\Lambda) := \sup_{s\, \in\, (0,1)} \;\summ \ell\Kc(1-s)\lambda_{k\ell}+ s\lambda_{r\ell}-\lambda_{k\ell}^{1-s}\lambda_{r\ell}^s, 
	\end{align}
	for $k,r \in [\Kr]$. 
	We can think of $I(\Lambda) := (I_{kr}(\Lambda)) \in \reals_+^{\Kr \times \Kr}$, as an operator acting on pairs of rows of a matrix $\Lambda \in \reals_+^{K \times L}$, say $\lambda_{k*}$ and $\lambda_{r*}$, producing a $\Kr \times \Kr$ pairwise information matrix. We often refer to the function of $s$ being maximized in~\eqref{eq:Info:def} as $s \mapsto I_s$, with some abuse of notation, dropping the dependence on $k$ and $r$ and assuming that they are fixed. This function is strictly concave over $\reals$ whenever $\lambda_{k*} \neq \lambda_{r*}$, and we have $I_0 = I_1 = 0$.
	
	Recalling the product Poisson distributions $\{\Qdist_k\}$, $I_{kr}$ given in~\eqref{eq:Info:def} is the Chernoff exponent in testing the two hypotheses $\Qdist_k$ and $\Qdist_r$~\cite{chernoff1952measure}. The difference with the classical setting in which the Chernoff exponent appears is the regime we work in, where we are effectively testing based on a sample of size of 1
	and instead, 
	let $I_{kr} \to \infty$.
	%
	The column information matrix is defined similarly 
	\begin{align}\label{eq:Icol:def}
	\Icol_{\ell \ell'} :
	= I_{\ell \ell'}(\Gamma) 
	= \sup_{s \in (0,1)} \;\summ k\Kr (1-s)\Gamma_{\ell k}+ s\Gamma_{\ell'k}-\Gamma_{\ell k}^{1-s}\Gamma_{\ell' k}^s,
	\end{align}
	for all $\ell,\ell' \in [\Kc]$. We let $\I_{\min} := \min_{k\ne r}\I_{kr}$ and $\I_{\min}^{\text{col}} := \min_{\ell\ne \ell'}\I_{\ell\ell'}^{\text{col}}$.
	Another set of key quantities in our analysis are: 
	\begin{align}\label{eq:eps:def}
	\epsikr:= \max_{\ell\in[\Kc]} \lp \f{\lambda_{k \ell}}{\lambda_{r\ell}} \vee \f{\lambda_{r \ell}}{\lambda_{k\ell}} \rp-1, \quad
	\epsi_k:=\min_{r\in[\Kr]}\epsikr, \quad \text{ and } \quad
	\epsi:=\min_{k\in[\Kr]}\epsi_k.
	\end{align}
	%
	The relation with hypothesis testing is formalized in the following proposition:
	\begin{prop}\label{prop:HT:err:rates}
		Consider the likelihood ratio (LR) testing of  the null hypothesis $\Qdist_k$ against $\Qdist_r$, based on a sample of size $1$. Let $\Lambda = [\lambda_{k*};\lambda_{r*}] \in \reals_+^{2 \times \Kc}$. Assume that as $\Lammin \to \infty$, (a) $\liminf \epsi_{kr} > 0$, and (b) $\omega = O(1)$.
		Then, there exist constants $C$ and $C'$ 
		such that
		\begin{align}
		\pr(\text{Type I error}) +\pr(\text{Type II error})  
		\begin{cases}
		\le & C \exp\big({-I_{kr} - \f 12 \log \Lamin}\big),	\\[1.2ex]
		\ge &\;\;\exp\big({-I_{kr} - \f \Kc 2 ( \log \Lammin+C')}\big).
		\end{cases}
		\end{align}
	\end{prop}
	See Corollary~\ref{cor:poi:err:exponent} and Appendix~\ref{sec:proof:sec:main:res}  for the proof.
	Any hypothesis testing procedure can be turned into a classifier, and a bound on the error of the hypothesis test (for a sample of size 1) translates into a bound on the misclassification rate for the associated classifier. This might not be immediately obvious, and we provide a formal statement in Lemma~\ref{lem:miss:markov}. Proposition~\ref{prop:HT:err:rates} thus provides a precise bound on the misclassification rate of the \emph{LR classifier} for deciding between $\Qdist_k$ and $\Qdist_r$. 

	The significance  of the Chernoff exponent of the hypothesis test in controlling the rates is thus natural, given the full information about the $\{\Qdist_k\}$ and the test statistics. What is somewhat surprising is that almost the same bound holds when no such information is available a priori. Our main result below is a formalization of this claim. In our assumptions, we include a parameter $Q \in \nats$ that controls the number of sub-blocks when partitioning, the details of which are discussed in Section~\ref{sec:provable:alg}.
	Under the following two assumptions:
	\begin{align}
	(Q^2\log Q)\beta^2\omega^3\Kr\Kc(\Kr\vee\Kc)\log(\Kr\vee\Kc) (\Laminf \vee \Gaminf)^2
	&= O(\nr \wedge \nc), \text{ and } \tag{A3} \label{assump:sparse:network} \\
	%
	(\Q\log Q)^2 \beta^3\omega^2(\Kr\vee\Kc)^3 (\alpha \vee \alpha^{-1})(\Laminf\vee\Gaminf) &= 
	o\big(\, (\Imin\wedge\Imincol)^2 \,\big), \tag{A4}\label{assump:bounded:J:strong}
	\end{align}
	where $\alpha:= \nc/\nr$, there is an algorithm that achieves almost the same rate as the oracle:
	\begin{thm}[Main result]\label{thm:main:res}
		Consider a bipartite SBM  (Section~\ref{sec:bi:sbm}) satisfying~\eqref{assump:Lambda}--\eqref{assump:bounded:J:strong}.
		Then, as $\Imin \wedge \Imincol \to \infty$ and $\Lamin \to \infty$, the row labels $\yh$ outputted by Algorithm~\ref{alg:provable} in Section~\ref{sec:provable:alg} satisfy
		for some $\zeta=o(1)$, 
		\begin{align}\label{eq:main:mis:rate}
		\mis_k \big(\yh,\,\y\big) 
		= O\lp{\omega} \sum_{r \neq k} \lp 1+ \f{1}{\epsi_{kr}}\rp \exp\Big( {-} I_{kr} -
		\Big( \f 12 - \zeta \Big) \log \Lammin \Big) \rp
		\end{align}
		for every $k\in[\Kr]$, with high probability. 
		Similar bounds holds for $\zh$ w.r.t. $z$. 
	\end{thm}
	
	One can  replace the big~$O$ with the small~$o$ in~\eqref{eq:main:mis:rate} to obtain an equivalent result (due to the presence of $\zeta = o(1)$).
	Let us discuss the assumptions of Theorem~\ref{thm:main:res}. The only real assumptions are~\eqref{assump:sparse:network} and~\eqref{assump:bounded:J:strong}. The other two, namely~\eqref{assump:Lambda} and~\eqref{assump:balance}, are more or less 
	the definitions of $\omega$ and $\beta$. For example, \eqref{assump:balance} only imposes the mild constraint that no cluster is empty. Similarly \eqref{assump:Lambda} imposes the mild assumption that no entry of $\Lambda$ or $\Gamma$ is zero. The main constraints on $\omega$ and $\beta$ are encoded in~\eqref{assump:sparse:network} and~\eqref{assump:bounded:J:strong} in tandem with the other parameters. 
	\begin{rem}
		In the first reading, one can take $\beta,\omega,Q=O(1)$, $\nr\asymp\nc$
		and $\Laminf\asymp\infnorm\Gamma$. In this setting, \eqref{assump:sparse:network} is a very mild sparsity condition, implying that  the degrees should not grow faster than $\sqrt n$.  Condition~\eqref{assump:bounded:J:strong} guarantees that the information quantities grow fast enough so that the clusters are distinguishable. We only need~\eqref{assump:bounded:J:strong} for Algorithm~\ref{alg:provable} which uses a spectral initialization. In Section~\ref{sec:gen:init}, we present Theorem~\ref{thm:without:spec:clust} for the likelihood-updating portion of the algorithm, assuming that a good initialization is provided irrespective of the algorithm used. Theorem~\ref{thm:without:spec:clust}  only requires a weakened version of  assumption~\eqref{assump:bounded:J:strong}; see~\eqref{assump:bounded:J} in  Section~\ref{sec:gen:init}. 
	\end{rem}
	Depending on the behavior of $\eps_{kr}$, the rate obtained in Theorem~\ref{thm:main:res} can exhibit different regimes which are summarized in Corollary~\ref{cor:main:res} below.
	Consider the additional assumption:
	\begin{align}\label{assump:bounded:J:very:strong}
	\max_{k,r\,\in\,[\Kr]}\omega\lp 1+\f 1{\epsi_{kr}}\rp = O(1). \tag{A5}
	\end{align}
	\begin{cor}\label{cor:main:res}
		Under the same assumptions as Theorem~\ref{thm:main:res}, w.h.p., for all $k \in [\Kr]$,
		\begin{align}\label{eq:little:o:rate}
		\mis_k \big(\yh,\,\y\big) 
		= o\lp \sum_{r \neq k} \exp ( {-} I_{kr}) \rp.
		\end{align}
		If in addition we assume~\eqref{assump:bounded:J:very:strong}, then for some $\zeta=o(1)$, w.h.p., for  all $k \in [\Kr]$, 
		\begin{align}\label{eq:Big:O:rate}
		\mis_k \big(\yh,\,\y\big) 
		= O\lp \sum_{r \neq k} \exp\Big( {-} I_{kr} - \Big( \f 12 - \zeta \Big) \log \Lammin \Big) \rp.
		\end{align}
	\end{cor}

	\begin{rem}\label{rem:oracle:problem}
		Consider the oracle version of the biclustering problem where the connectivity matrix $P$ and the true column labels $\z$ are given. Then, the optimal row clustering reduces to the likelihood ratio tests in Proposition~\ref{prop:HT:err:rates}. That is, given the row sums within blocks as sufficient statistics, we compare the likelihoods at two different mean parameters. By Proposition~\ref{prop:HT:err:rates}, the optimal misclassification rate for the oracle problem is
		\begin{align}
		\E\big[\mis_k \big(\yh,\,\y\big) \big]
		= O\lp \sum_{r \neq k} \exp\Big( {-} I_{kr} - \f 12 \log \Lammin \Big) \rp,
		\end{align}
		where the sum over $r$ is due to the need to compare against all other clusters.
		The gap between $1/2$ and $1/2-\zeta$ is not avoidable when stating high probability results, due to the Markov inequality; see Lemma~\ref{lem:miss:markov} for the details. This error rate coincides with~\eqref{eq:Big:O:rate}, which merely loses a constant due to the unknown mean parameters and column labels. The rate is sharp up to a factor of $\exp(-\frac12(\Kc-1) \log \Lamin)$ according to the lower bound in Proposition~\ref{prop:HT:err:rates}.
	\end{rem}
	
	
	The argument in Remark~\ref{rem:oracle:problem} can be formalized as the following minimax lower bound:
	
	\begin{thm}[Minimax lower bound]\label{thm:minimax}
		Consider  the parameter space 
		\begin{align*}
		\Ss:=\Ss(\nr, \nc, \Kr, \Kc, I^*) := \big\{(\y, \z, P) \mid  \;&\y\in \{0,1\}^{\nr\times\Kr}, \z\in \{0,1\}^{\nc\times\Kc}, \, P \in [0,1]^{\Kr \times \Kc}, \\ 
		&\text{$\Lambda$ is defined based on $z$ and $P$ according to~\eqref{eq:tru:row:mean:def}},\\
		&\Imin(\Lambda)\ge I^*, \;\eps\ge \eps^*\text{ where }\eps \text{ is defined in }\eqref{eq:eps:def},\\
		&\y, \z \text{ and } \Lambda  \text{ satisfy } \eqref{assump:Lambda} \text{ to }\eqref{assump:bounded:J:strong}. \big\},
		\end{align*}
		and assume that there exists $(y,z,P)\in \Ss$ such that $\Imin(\Lambda)=I^*$.
		Further assume that $\beta, \omega > 1$, $\eps^*>0$ are constants and $\Kr \le \exp(c\, L)$ for some constant $c > 0$. 
		Then, for large $n$, the minimax risk over $\Ss$ satisfies
		\begin{align}\label{eq:minimax:lower:bound}
		\inf_{\hat y} \sup_{(y,z,P) \,\in\, \Ss} \E[\,\mis(\hat y,y)\,] \;\ge\; \exp\big({-I^* - \Kc ( \log I^*+C)}\big).
		\end{align}
	\end{thm}

	Theorem~\ref{thm:minimax} is a non-asymptotic result, i.e., we fix $\nr$ (and hence $\nc$). In this case, assumption~\eqref{assump:bounded:J:strong} in the definition of the parameter space $\Ss$ should be interpreted by fixing a vanishing sequence in advance. Note that in defining $\Ss$, only $\Lambda$ and not $\Gamma$, is required to satisfy \eqref{assump:sparse:network}--\eqref{assump:bounded:J:strong}.

	In order to better understand the rates in Corollary~\ref{cor:main:res}, let us consider some examples which also  clarify our results relative to the previous literature.

	\begin{exa}
		\label{exa:comparison:with:harry}
		Consider a simple planted partition model where
		\begin{align*}
		\nr=\nc, \quad \Kr=\Kc,  \quad  P_{kk} = \frac{a}n, \quad P_{k\ell} = \f{b}n,\; \forall k \neq \ell. 
		\end{align*}
		Then, $\lambda_{kk} \in [\f {a} {\beta\Kr}, \f {\beta a} \Kr]$ and $\lambda_{k\ell} \in [\f {b} {\beta\Kr}, \f {\beta b} \Kr]$ when $k\ne\ell$. Applying~\eqref{eq:Info:def} with $s = 1/2$, 
		\begin{align*}
		I_{kr} \ge \frac12 \sum_{\ell} (\sqrt{\lambda_{k\ell}} - \sqrt{\lambda_{r\ell}})^2
		\ge \f{(\sqrt a - \sqrt b)^2}{\beta\Kr}.
		\end{align*}
		Assume that~\eqref{assump:sparse:network} and~\eqref{assump:bounded:J:strong} hold, that is  (using $\Laminf \le \beta a/K$)
		\begin{align*}
		\beta^4\omega^3(K\log K) a^2=O(n) \quad \text{and}\quad
		\beta^6\omega^2K^4 a = o\lp (\sqrt a - \sqrt b)^4 \rp.
		\end{align*}
		and further assume that $\beta\omega^2\Kr^3=o(a\wedge b)$. Then w.h.p., we have
		\begin{align}\label{eq:pp:mis:bound}
		\mis_k \big(\yh,\,\y\big) 
		= o\lp \exp\Big({-} \f{(\sqrt a - \sqrt b)^2}{\beta\Kr}\Big) \rp.
		\end{align}
		For the details of $\eqref{eq:pp:mis:bound}$, see Section~\ref{sec:proof:harry}. In particular, if 
		\begin{align}\label{assump:exact:recovery:harry}
		\liminf_{n\to \infty} \;\frac{(\sqrt a - \sqrt b)^2}{\beta\Kr\log n} \;\ge\; 1,
		\end{align}
		we have $\mis_k \big(\yh,\,\y\big)=o(1/n)$ w.h.p., that is, we have the exact recovery of the labels by Algorithm~\ref{alg:provable}. (Whenever misclassification rate drops below $1/n$, it should be exactly zero.) Note that this result holds without any assumption of assortativity, i.e., it holds whether $a  > b$ or $b > a$.
	\end{exa}
	
	\begin{exa}
		\label{exa:comparison:with:abbe}
		Suppose that $P := \Pt (\log n)/n$ where $\Pt$ is a symmetric constant matrix, $\nr = \nc$, $\Kr = \Kc$, and $\y = \z$. Then $\Kr, \omega$ and $\epsi_{kr}$ are constants. Then, 
		\begin{align*}
		\lambda_{k\ell}= \lambdat_{k\ell} \log n, \quad \text{where}\;\;\lambdat_{k\ell} := \Pt_{k\ell}\pi_{k}(y), 
		\quad\text{and}\;\;I_{kr} = \It_{kr}\log n
		\end{align*}
		where $\It_{kr}$ is defined based on $\lambdat_{k\ell}$ and $\lambdat_{r\ell}$ as in~\eqref{eq:Info:def}.  Assuming in addition that $\pi(\y)$ is constant, both $\lambdat_{kr}$ and $\It_{kr}$ are constants. In this regime, our key assumptions~\eqref{assump:sparse:network} and \eqref{assump:bounded:J:strong} are satisfied. By Corollary~\ref{cor:main:res}, w.h.p., we have 
		\begin{align}
		\mis_k \big(\yh,\,\y\big) 
		= o\lp \exp\Big( {-} \min_{r\ne k} \tilde I_{kr} \log n \Big) \rp
		= o\lp n^{{-} \min_{r\ne k} \tilde I_{kr}} \rp.
		\end{align}
		As a consequence if $\min_{k\ne r} \tilde I_{kr}\ge 1$, then $\mis_k \big(\yh,\,\y\big)=o(1/n)$ w.h.p., that is we have exact recovery by Algorithm~\ref{alg:provable}. 
	\end{exa}
	
	In addition to the above more or less familiar setups (cf. Section~\ref{sec:comparison}), our results determine the optimal rate for a much wider range of parameter settings. As an example, consider the following setting of very slow degree growth:
	\begin{exa}
		\label{exa:slow:growth}
		Consider the setup of Example~\ref{exa:comparison:with:abbe} but with $\log n$ replaced with $\log \log n$ in the definition of $P$. In this case, the expected average-degree grows very slowly as $\log \log n$, and it is known that exact recovery is not possible in this regime. However, our results establish the following minimax optimal rate: 
		\[
		\mis(\yh,y) \asymp \frac1{(\log n)^\alpha} \frac{1}{\log \log n}
		\]
		where $\alpha = \min_{k\ne r} \tilde I_{kr}$ and this rate is achieved by Algorithm~\ref{alg:provable}. 
	\end{exa}

	
	
	
	\subsection{Comparison with existing results}\label{sec:comparison}
	Let us now compare with~\cite{gao2017achieving} and~\cite{abbe2015community} whose results are closest to our work. Both papers consider the symmetric (unipartite) SBM, but the results can be argued to hold in the bipartite setting as well. The setup of Example~\ref{exa:comparison:with:harry} is more or less what is considered in~\cite{gao2017achieving}.
	%
	They have shown that there is an algorithm with misclassification error  bounded by
	\begin{align}\label{eq:pp:mis:weaker:bound}
	\exp\Big( {-} \f{(1-o(1))(\sqrt a - \sqrt b)^2}{\beta\Kr}\Big)
	\end{align}  
	when $a>b$. We have sharpened this rate to~\eqref{eq:pp:mis:bound}
	under assumption~\eqref{assump:sparse:network} (i.e., assuming the average degree grows slower than $O(\sqrt{n})$). Bound~\eqref{eq:pp:mis:weaker:bound} implies that when
	\begin{align*}
	\liminf_{n\to \infty}\f{(\sqrt a - \sqrt b)^2}{\beta\Kr\log n}> 1,
	\end{align*}
	one has exact recovery. Our bound on the other hand, imposes the relaxed condition~\eqref{assump:exact:recovery:harry}.
	
	We note that the results in~\cite{gao2017achieving} are derived for a more general class of (assortative) models than that of Example~\ref{exa:comparison:with:harry}, namely, the class with connectivity matrix satisfying $P_{kk} \ge a/n$ and $P_{k\ell} \le b/n$ for $k \neq \ell$. The rate obtained in~\cite{gao2017achieving} uniformly over this class is dominated by that of the hardest within this class which is the model of Example~\ref{exa:comparison:with:harry}. For other members of this class, neither their rate~\eqref{eq:pp:mis:weaker:bound} or the one we gave in~\eqref{eq:pp:mis:bound} is optimal. The optimal rate in those cases is given by the general form of Theorem~\ref{thm:main:res} and is controlled by the general form of $I_{kr}$ in~\eqref{eq:Info:def}. In other words, Algorithm~\ref{alg:provable} that we present  is \emph{rate adaptive} over the class considered in~\cite{gao2017achieving}, achieving the optimal rate simultaneously for each member of the class.
	
	%
	A key in our approach is to apply the likelihood-type algorithm twice, in contrast to the single application in~\cite{gao2017achieving}.  After the second stage we obtain much better estimates of the labels and parameters relative to the initial values, allowing us to establish the sharper forms of the bounds. Another key is the result in Lemma~\ref{lem:unif:param:b}  which provides a better error rate than the classical Chernoff bound, using a very innovative technique introduced in~\cite{abbe2015community}. Moreover, we keep track of the balance parameter $\beta$ in \eqref{assump:balance} throughout, allowing it to go to infinity slowly.
	Last but not  least, assortativity is a key assumption in~\cite{gao2017achieving}, while our result does not rely on it. Besides consistency, our provable algorithm is  computationally more efficient. To obtain initial labels, we  apply the spectral clustering on very few subgraphs (8 to be exact). However, the provable version of the algorithm in~\cite{gao2017achieving} applies the spectral clustering for each single node on the rest of the graph excluding that node. If the cost of running the spectral clustering on a network of $n$ nodes is $C_n$, then our approach costs $\approx 8 C_{n/8}$ while that of~\cite{gao2017achieving} costs roughly $n C_{n-1}$.   Our algorithm thus has a significant advantage in computational complexity when $n\to\infty$. To be fair, the algorithm  in~\cite{gao2017achieving} was for the symmetric SBM, which has the extra complication of dependency in $A$ due to symmetry. Our comparison here is mostly with Corollary 3.1 in~\cite{gao2017achieving}. In addition, \cite{gao2017achieving} have a result (their Theorem~5) for when $\omega$ grows arbitrarily fast which is not covered by our result. See Section~\ref{sec:discussion:theory}  for comments on the symmetric case and dependence on $\omega$.
	
	The problem of exact recovery for a general SBM has been considered in~\cite{abbe2015community}, again for the case of a symmetric SBM, though the results are applicable to the bipartite setting (with $\y =\z$). The model and scaling considered in~\cite{abbe2015community} is the same as that of Example~\ref{exa:comparison:with:abbe}, and they show that exact recovery of all labels is possible if (and only if) $\min_{k,r: k\ne r} \tilde I_{kr}\ge 1$ which is the same result we obtain in Example~\ref{exa:comparison:with:abbe} for  Algorithm~\ref{alg:provable}. Thus, our result contains that of~\cite{abbe2015community} as a special case, namely in the $\log n$-degree regime with other parameters (such as $\Kr$ and the normalized connectivity matrix) kept constant.
	The results and algorithms of~\cite{abbe2015community} do not apply to the general model in our paper. First, they only consider the regime $P \sim \log n/n$, i.e., the degree grows as fast as $\log n$, while we allow the degree to grow in the range from ``arbitrarily slowly'' up to ``as fast as $O(\sqrt n)$''. Second, as discussed in Section~\ref{sec:edge:splitting:comments}, their edge splitting idea cannot be used to derive the results in this paper, and we introduce the block partitioning to supply us with the independent copies necessary for the analysis.

	Finally, we note that Example~\ref{exa:slow:growth} with a general nonassortative matrix $\Pt$ has no counterpart in the literature. Existing results are not capable of providing any guarantees for such setups.
	
\section{Pseudo-likelihood approach}
\label{sec:pl:approach}
In this section, after introducing the local and global mean parameters which will be used throughout the paper, we present our general pseudo-likelihood approach to biclustering.

\subsection{Local and global mean parameters}
\label{sec:local:global:param}

Let us define the following operator that takes an adjacency matrix $A$ and row and column labels $\yt$ and $\zt$, and outputs the corresponding (unbiased) estimate of its mean parameters:
\begin{align}\label{eq:Lop:def}
[\Lop(A,\yt,\zt)]_{k\ell} = \frac1{n_k(\yt)} \sum_{i=1}^\nr\sum_{j=1}^\nc A_{ij} 1\{\yt_i = k,\zt_j=\ell\}, \quad k\in [\Kr],\; \ell \in [\Kc].
\end{align}
Note that $\Lop(A,\yt,\zt)$ is a $\Kr \times \Kc$ matrix with nonnegative entries. In general, we let
\begin{align}
\Lamh &= (\lamh_{k\ell}) := \Lop(A,\yt,\zt), \label{eq:Lop:Lamh}\\
\Lambda(\yt,\zt) &= (\lambda_{k\ell}(\yt,\zt)) := \Lop(\ex[A],\yt,\zt),  \label{eq:global:mean:param:def}
\end{align}
for any row and column labels $\yt$ and $\zt$. Here $\Lamh$ is the estimate of the true row mean matrix. Its expectation is $\ex[\Lamh] =  	\Lambda(\yt,\zt) $ due to the linearity of $\Lop$.
We call $\Lambda(\yt,\zt)$, the \emph{(global) row mean parameters} associated with labels $\yt$ and $\zt$.  (We do not explicitly show the dependence of $\Lamh$ on the labels, in contrast to the mean parameters.) We have the following key identity
\begin{align}\label{eq:Lambda:identity}
\Lambda(\yt,\zt)\mid_{\yt = y,\; \zt=z} \;= \;\Lambda
\end{align}
where $\Lambda$ is the \emph{true} (global) row mean parameter matrix defined in~\eqref{eq:tru:row:mean:def}. In words,~\eqref{eq:Lambda:identity} states that the global row mean parameters associated with the true labels $\y$ and $\z$, are the true such parameters. We will also use parameters such as $\Lambda(\y,\zt)$ which are obtained based on the true row labels $\y$ and generic column labels $\zt$.

We also need local versions of all these definitions which are obtained based on submatrices of $A$. More precisely, let $A^{(q',q)}$ be a submatrix of $A$, and let $\y^{(q')}$ and $\z^{(q)}$ be the corresponding subvectors of $\z$ and $\y$ (i.e., corresponding to the same row and column index sets used to extract the submatrix). Here $q,q'$ range over $[Q] = \{1,\dots,Q\}$ creating a partition of $A$ into $Q^2$ subblocks.  We call
\begin{align}\label{eq:local:mean:def}
\Lambda^{(q',q)}(\yt,\zt) := (\lambda^{(q',q)}_{k\ell}(\yt,\zt)) := \Lop(\ex[A^{(q',q)}],\,\yt^{(q')},\,\zt^{(q)}),
\end{align}
the \emph{local row mean parameters} associated with submatrix $A^{(q',q)}$ and sublabels $\y^{(q')}$ and $\z^{(q)}$. The corresponding estimates are defined similarly (by replacing $\ex[A^{(q',q)}]$ with $A^{(q',q)}$). We will mostly work with submatrices obtained from a partition $A^{(q',q)}, \, q', q\in [Q]$ of $A$ into (nearly) equal-sized blocks---the details of which are described in Section~\ref{sec:provable:alg}. In such cases, 
\begin{align*}
\Lambda^{(q',q)}(\yt,\zt) \approx \frac{1}{Q} \Lambda(\yt,\zt), \quad \forall q',q \in[Q]
\end{align*}
assuming the each subblock in the partition has nearly similar cluster proportions: $n(z^{(q)}) \approx n(\z)$. This is the case, for example, for a random partition as we show in Section~\ref{sec:subblk:analysis}. Of special interest is when we replace both $\yt$ and $\zt$ with true labels $\y$ and $\z$. In such cases, $\Lambda^{(q',q)}$ does not depend on $q'$. More precisely, we have for any $q \in [Q]$,
\begin{align}\label{eq:true:local:mean:def}
\lambda^{(q',q)}_{k\ell}(\y,\z) =  P_{k \ell}\, n_\ell(\z^{(q)}), \quad \forall q' \in [Q],
\end{align}
where $n_\ell(z^{(q)})$ is the number of labels in class $\ell$ in $z^{(q)}$, consistent with our notation for the full label vectors.
We often write $\Lambda^{(q)} = (\lambda_{k\ell}^{(q)})$ as a shorthand for $\Lambda^{(q',q)}(\y,\z)$ which is justified by the above discussion. These will be called the \emph{true} local row mean parameters (associated with column $q$ subblock in the partition). 

\subsection{General pseudo-likelihood algorithm}\label{sec:gen:pl:alg}
Let us now describe our main algorithm based on the pseudo-likelihood (PL) idea, which is a generalization of the approach in~\cite{amini2013pseudo} to the bipartite setup. The pseudo-likelihood algorithm (\pla) is effectively an EM algorithm applied to the approximate mixture of Poissons obtained from the block compression of the adjacency matrix $A$. It relies on some initial estimates of the row and column labels to perform the first block compressions (for both rows and columns). The initialization is often done by spectral clustering and  will be discussed in Section~\ref{sec:spectral:init}. 
Subsequent block compressions are performed based on the label updates at previous steps of \pla.

Let us assume that we have obtained labels $\yt$ and $\zt$ as estimates of the true labels $\y$ and $\z$. We  focus on the steps of \pla for recovering the row labels. Let us define an operator $\Bc(A;\zt)$ that takes approximate columns labels and produces the corresponding \emph{column compression} of~$A$:
\begin{align}\label{eq:Bc:def}
\Bc(A;\zt) := \bb(\zt) := \big(b_{i\ell}(\zt)\big) \in \Z_+^{\nr \times \Kc}, \quad 
b_{i\ell}(\zt) := \sum_{j=1}^{\nc} A_{ij} 1\{\zt_j = \ell\}.
\end{align}
The distribution of $b_{i\ell}(\zt)$ is determined by the row class of $i$. It is not hard to see that 
\begin{align}
\ex[b_{i\ell}(\zt)] = \lambda_{k\ell}(\y,\zt) = \lambda_{k\ell}(\y,\zt)|_{\zt=z}, \quad \text{if $y_i = k$,}
\end{align}
where $\lambda_{k\ell}(\y,\zt)$ are the (global) row mean parameters defined in~\eqref{eq:global:mean:param:def}.

Now consider an operator $\Lc(\bb;\yt)$ that, given the column compression $\bb$ and the initial estimate of the row labels $\yt$, produces estimates of the \emph{(row) mean parameters} $\lambda_{k\ell}(\y,\zt)$:
\begin{align}\label{eq:Lc:def}
\Lc(\bb; \yt) := \Lamh := [\lamh_{k\ell}] \in \reals_+^{\Kr \times \Kc},
\quad \lamh_{k\ell} := \frac1{n_k(\yt)} \sum_{i=1}^\nr b_{i\ell} 1\{\yt_i = k\}. 
\end{align}
Note that if $\yt=\y$, we have $\E[\lamh_{k\ell}]=\lambda_{k\ell}(\y,\zt)$. The definition of the estimates in~\eqref{eq:Lc:def} are consistent with those of~\eqref{eq:Lop:Lamh} due to the following identity:
\begin{align*}
\Lc(\Bc(A;\zt); \yt)  = \Lop(A,\yt,\zt)
\end{align*}
which holds for any row labels $\yt$ and column labels $\zt$.
Let us write
\begin{align}\label{eq:pi:def}
\pi(\yt) := (\pi_k(\yt)), \quad \pi_k(\yt) := \frac1{\nr}  \sum_{i=1}^\nr 1\{\yt_i = k\}
\end{align}
for the estimate of (row) class priors based on $\yt$. We note that  operations $\Bc$ and $\Lc$ remain valid even if $\yt$ and $\zt$ are \emph{soft labels} with a minor modification. By a soft row label $\zt_j \in [0,1]^\Kc$ we mean a probability vector on $[\Kc]$: $\zt_{j\ell} \ge 0$ and $\sum_{\ell=1}^\Kc \zt_{j\ell} = 1$, which denotes a soft assignment to each row cluster. To extend~\eqref{eq:Bc:def} to soft row labels, it is enough to replace $1\{z_j = \ell\}$ with $z_{j\ell}$. Extending~\eqref{eq:Lc:def} to soft column labels $\yt$ is done similarly.


Now, given any block compression $\bb = (b_{i\ell})$ and any estimate $\Lamh$ of the (row) mean parameters and any estimate $\pit\in [0,1]^\Kr$ of the (row) class prior, consider the operator that outputs the \emph{(row) class posterior} assuming that the rows of $\bb_i$ approximately follow  $\sum_k \pit_k \prod_{\ell} \poi(\lamh_{k\ell})$:
\begin{align}\label{eq:Fc:def}
\Fc(\bb, \Lamh, \pit) := (\pih_{i k}) \in [0,1]^{\nr \times \Kr}, \quad 
\pih_{i k} 	:=
\frac{\pit_k \prod_{\ell=1}^\Kc \ppmf(b_{i\ell}, \lamh_{k\ell})}{\sum_{k'=1}^\Kr \pit_{k'} \prod_{\ell=1}^\Kc \ppmf(b_{i\ell}, \lamh_{k'\ell})}
\end{align}
where $\ppmf(x,\lambda) = \exp(x \log \lambda - \lambda)$ is the Poisson likelihood (up to constants). In practice, we only use $\pi(\yt)$ or a flat prior $\onev$ as the estimated prior $\pit$ in this step; similarly, we only use a block compression which is based on estimates of row labels, i.e., $b_{i\ell} = b_{i\ell}(\zt)$ for some $\zt \in [\nr]^\Kc$. Note that $\Fc$ outputs soft-labels
as  new estimates of $\y$. We can convert $(\pih_{ik})$ to hard labels if needed.

\begin{algorithm}[t]
	\caption{Pseudo-likelihood biclustering, meta algorithm}
	\label{CHalgorithm}
	\begin{algorithmic}[1]
		\State Initialize row and column labels $\yt$ and $\zt$.
		\While{$\yt$ and $\zt$ have not converged}
		\State $\bb \gets \Bc(A;\zt)$
		\While{ $\Lamh$ and $\pih$ not converged (optional)}
		\State $\Lamh \gets \Lc(\bb; \yt)$
		\State Option 1: $\pit\gets \onev$, or option 2: $\pit\gets \pi(\yt)$
		\State $\yt \gets \Fc(\bb, \Lamh, \pit)$ 
		\State (Optional) Convert $\yt$ to hard labels.
		\EndWhile
		\State Repeat lines 3--9 with appropriate modifications to update $\zt$ and columns parameters (by changing $A$ to $A^T$ and swapping $\zt$ and $\yt$.)
		\State (Optional) Convert $\yt$ and $\zt$ to hard labels if they are not.
		\EndWhile
	\end{algorithmic}
	\label{alg:pl:biclust}
\end{algorithm}

Algorithm~\ref{alg:pl:biclust} summarizes the general blueprint of \pla which proceeds by iterating the three operators~\eqref{eq:Bc:def},~\eqref{eq:Lc:def} and~\eqref{eq:Fc:def}. Optional conversion from soft to hard labels is performed by MAP assignment per row. 
With option 2 in step 6, the inner loop on lines 4--8 is the EM algorithm for a mixture of Poisson vectors. We can also remove the inner loop and perform  iterations 5--8 only once. In total, Algorithm~\ref{alg:pl:biclust} has (at least) 6 possible versions, depending on whether we include each of the steps 8 or 11 (for the soft to hard label conversion) and whether to implement the inner loop till convergence or only for one step. We provide empirical results for two of these versions in Section~\ref{sec:sims}. In practice, we recommend to keep soft labels throughout, and only run the inner loop for a few iterations (maybe even one if the computational cost is of significance).

\subsection{Likelihood ratio classifier}
\begin{algorithm}[t]
	\caption{Simplified pseudo-likelihood clustering }
	\begin{algorithmic}[1]
		\State {\bf Input:} Initial column labels $\zt$, and $\Lamt$ that estimates $\Lambda$.
		\State {\bf Output:} Estimate of row labels $\yh$.
		\State $\bb \gets \Bc(A;\zt)$
		\State $\Lamh \gets \Lc(\bb; \yt)$
		\State $\yh \gets \Fc(\bb, \Lamt, \onev)$
		\State Convert $\yh$ to hard labels, by computing MAP estimates.
	\end{algorithmic}
	\label{alg:simp}
\end{algorithm}

A basic simplified building block of the \pla is given in Algorithm~\ref{alg:simp}. This operation---which will play a key role in the development of the provable version of the algorithm in Section~\ref{sec:provable:alg}---can be equivalently described as a \emph{likelihood ratio classifier} (\lrc). Let us write the joint Poisson likelihood (up to a constant) as:
\begin{align}\label{eq:joint:poi:pmf}
\Ppmf(x,\lambda) = \prod_{\ell=1}^\Kc \ppmf(x_\ell, \lambda_\ell) = \prod_{\ell=1}^\Kc \exp(x_\ell \log \lambda_\ell - \lambda_\ell), \quad x \in \reals^L, \; \lambda \in \reals_+^L,
\end{align}
and the corresponding likelihood ratio as:
\begin{align}\label{eq:joint:poi:llr}
\poillr(x; \lambda \mid \lambda') = \log \frac{\Ppmf(x,\lambda)}{\Ppmf(x,\lambda')} = 
\sum_{\ell=1}^L x_\ell \log \frac{\lambda_\ell}{\lambda'_{\ell}} + \lambda'_\ell - \lambda_\ell,
\quad x \in \reals^L, \; \lambda,\lambda' \in \reals_+^L.
\end{align}
Recalling the column compression~\eqref{eq:Bc:def}, the \emph{likelihood ratio classifier}, based on initial row labels $\zt$ and an estimate  $\Lamt$ of the row mean parameter matrix, is 
\begin{align}\label{eq:LRC:def}
[\LR(A, \Lamt, \zt)]_i \;\in\; \argmax_{r \,\in\, [\Kr]} \; \log \Phi(b_{i*}(\zt), \lamt_{r*}), \quad 
i \in [\nr].
\end{align}
This operation gives us a refined estimate of the row labels (i.e., $\y$). It is not hard to see that the output of Algorithm~\ref{alg:simp} is $\yh = \LR(A, \Lamt, \zt)]$.

\section{Provable version}
\label{sec:provable:alg}


\def\commentshift{2.3in}
\begin{algorithm}[t!]
	\setstretch{1.3}
	\setcounter{algorithm}{2}
	\caption{Provable (parallelizable) version}\label{alg:provable}
	\label{Palgorithm}
	\begin{algorithmic}[1]
		\State Randomly partition the rows into 2 groups of equal size ($\nr/2$), so that
		\begin{align*}
		A=\cvecc{A_\top}{A_\bot}
		\end{align*}
		
		\label{step:top:bottom:part}
		\State  Randomly partition the rows and columns of $A_{\text{bottom}}$ into $4$ groups of equal size, so that we have $16$ sub-adjacency matrix with dimension $(n/8)\times (m/4)$, i.e.  
		\begin{align*}
		A_\bot = \begin{bmatrix}
		A^{(1,1)}      &A^{(1,2)}     &A^{(1,3)}      &A^{(1,4)} \\
		A^{(2,1)}      &A^{(2,2)}     &A^{(2,3)}      &A^{(2,4)} \\
		A^{(3,1)}      &A^{(3,2)}     &A^{(3,3)}      &A^{(3,4)} \\
		A^{(4,1)}      &A^{(4,2)}     &A^{(4,3)}      &A^{(4,4)} 
		\end{bmatrix}.
		\end{align*} 
		\label{step:4:part}
		
		\noindent In each of the following steps, perform the stated operation for every $q\in \Z_4$: 
		\medskip
		
		\State \makebox[2.75in][l]{Obtain initial row labels: }
		$\cvecc{\yt^{(q-1)}}{\yt'^{(q)}}\gets \rowSC\left(\cvecc{A^{(q-1,q)}}{A^{(q,q)}}\right), \;\forall q$.
		\label{step:init:row:lab}
		\State \makebox[2.75in][l]{Obtain initial column labels:} $\rvecc{\zt^{(q)}}{\zt'^{(q+1)}}\gets \colSC( \rvecc{A^{(q,q)}}{ A^{(q,q+1)}}), \;\forall q$.
		\label{step:init:col:lab}
		\State  \makebox[2.75in][l]{Get consistent (global) labels:} 
		$\yt \gets \match(\yt,\yt')$ and $\zt \gets \match(\zt,\zt')$. 
		\label{step:consist}
		\State\makebox[2.75in][l]{Update (local) row mean parameters:} 
		$\Lamh^{(q+2)}\gets \Lop(A^{(q,q+2)}, \yt^{(q)}, \zt^{(q+2)}), \;\forall q$.
		\label{step:local:Lambda:1}
		\State \makebox[2.75in][l]{Update row labels:}
		$\yt^{(q)}\gets \LR(A^{(q,q+2)}, \Lamh^{(q+2)},  \zt^{(q+2)}), \;\forall q$. 
		\label{step:row:LR:1}
		\State Similarly update column labels $\zt$ as  in steps 6 and 7.
		\label{step:similar:z:1}
		\State \makebox[2.75in][l]{Update (local) row mean parameters:}
		$\Lamh^{(q+3)} \gets \Lop(A^{(q,q+3)}, \yt^{(q)}, \zt^{(q+3)}),\; \forall q$.
		\label{step:local:Lambda:2}
		
		\State \makebox[2.75in][l]{Obtain (global) row mean parameters:} $\Lamh\gets\sum_{q}\Lamh^{(q)}$. 
		\label{step:global:Lambda}
		\State \makebox[2.75in][l]{} $\yh_\top\gets\LR(A_\top, \Lamh, \zt)$. 
		\label{step:row:LR:2}
		\State 
		\label{step:repeat}
		Swap $A_\top$ and $A_\bottom$, then repeat steps~\ref{step:4:part}--\ref{step:row:LR:2} to obtain $\yh_\bot$, except for step~\ref{step:consist} where $\yt$ is matched to $\yh_\top$ instead, i.e., $\yt\gets \match(\yt, \yh_\top)$. 
		\label{step:swap:top:bot}
		\State\makebox[2.75in][l]{}
		$\yh\gets\cvecc{\yh_\top}{\yh_\bot}$.
		\label{step:final:row:lab}
		\State Apply step 1 to 10 on $A^T$ to obtain $\zh$.
	\end{algorithmic}
	
\end{algorithm}

When analyzing Algorithm~\ref{alg:simp}, we need the initial labels to be independent of the adjacency matrix. Hence, we cannot apply the initialization method (e.g., the spectral clustering) and the likelihood ratio classifier (Algorithm~\ref{alg:simp}) on the same adjacency matrix $A$. 
In this section, we introduce Algorithm~\ref{alg:provable} which partitions $A$ into sub-blocks and operates iteratively on collections of these blocks to maintain the desired independence.  Algorithm~\ref{alg:provable} is the version of the pseudo-likelihood algorithm for which our main result (Theorem~\ref{thm:main:res}) holds.

Let us assume that $\nr$ and $\nc$  are divisible by $2Q = 8$. This assumption is not necessary but helps simplify  the notations. Let us write 
\begin{align*}
\yh=\rowSC(A), \quad \zh = \colSC(A)
\end{align*}
to denote labels obtained by applying the spectral clustering on rows and columns of the adjacency matrix $A$, respectively; see Section~\ref{sec:spectral:init} for details. We  have $\colSC(A) = \rowSC(A^T)$. We also recall the LR classifier defined in~\eqref{eq:LRC:def}.
For matrices (or vectors) $A$ and $B$, we use $[A;B]$ to denote column concatenation and $[A\;\,B]$ to denote row concatenation. 


The general idea behind the partitioning scheme used in Algorithm~\ref{alg:provable}, which is done by sequential sampling without replacement, is to ensure that in each step where the LR classifier is applied, the initial labels used are independent of the sub-block of the adjacency matrix under consideration. We do not require, however, that the initial labels be independent of the estimates of the mean parameters $\Lamh$,  since---as will be seen in Section~\ref{sec:fixed:label:analysis}---we have uniform consistency of the LR classifier over all $\Lamh$ close to the truth. For example, in step~\ref{step:row:LR:1}, that is, in the assignment $\yt^{(q)}\gets \LR(A^{(q,q+2)}, \Lamh^{(q+2)}, \zt^{(q+2)})$, the claim is that $\zt^{(q+2)}$---at that stage in the algorithm---is independent of $A^{(q,q+2)}$ but not necessarily of $\Lamh^{(q+2)}$. This will become clear in the following discussion where we keep track of the dependence of  various estimates through the algorithm. Note that in the description of Algorithm~\ref{alg:provable}, we are using the computer coding convention for in-place assignments, e.g., $\zt^{(q)}$ gets updated in place and refers to different objects at different points in the algorithm.

Figure~\ref{fig:A:partitions} illustrates the partitions used in steps~\ref{step:4:part}--\ref{step:local:Lambda:2} of the algorithm. The collection of the submatrices in the partition is given a name in each case. For example, $\Gcol_1$ consists of the four submatrices in Figure~\ref{fig:A:partitions}(a). Note that $\{\Gcol_1,G_2,G_3\}$ form a complete partition of the matrix into disjoint blocks. Also, $\Gcol_1$ and $\Grow_1$ involve the same elements of the matrix, i.e. they \emph{cover} the same portion of $A$. Thus, $\{\Grow_1,G_2,G_3\}$ is also a complete cover of $A$ with disjoint blocks. 
Let us write $G_1$ for the common portion of $A$ covered by $\Gcol_1$ and $\Grow_1$.

Steps~\ref{step:init:row:lab} and~\ref{step:init:col:lab} operate on blocks in $\Gcol_1$ and $\Grow_1$ respectively, producing initial row and column labels. For example, in step~\ref{step:init:row:lab}, we apply row SC on each submatrix specified in Figure~\ref{fig:A:partitions}(a) and obtain the label vectors (from the leftmost submatrix to the rightmost one):
\begin{align}\label{eq:subblk:labels:y}
\cvecc{\yt'^{(1)}}{\yt^{(4)}},\; \cvecc{\yt^{(1)}}{\yt'^{(2)}}, \; \cvecc{\yt^{(2)}}{\yt'^{(3)}}, \cvecc{\yt^{(3)}}{\yt'^{(4)}}.
\end{align}
As a result of these steps, we obtain two sets of row labels $\yt = (\yt^{(q)}:\; q\in \Z_4)$ and  $\yt' = (\yt'^{(q)}:\; q\in \Z_4)$, and similarly for the columns labels. Neither of $\yt$ or $\yt'$ is necessarily a consistent set of labels for the whole matrix, since the cluster labels for individual pieces $y^{(q)}$ and $\yt'^{(q)}$ need not match (e.g., cluster~1 in one piece could  be labeled cluster~2 in another piece.). However, if the sub-block labels~\eqref{eq:subblk:labels:y} are sufficiently close to the truth,
we can use the overlap among them to find a global set of labels that are consistent with each block of $\yt$ and $\yt'$. This is  what the \match operator in step~\ref{step:consist} does, as will be detailed in Section~\ref{sec:matching:step}. 
The resulting updated global row and column labels only depend on $G_1$ portion of $A$. 
\begin{figure}[t]
	\centering
	\begin{tabular}{cccc}
		\includegraphics[width=1.4in]{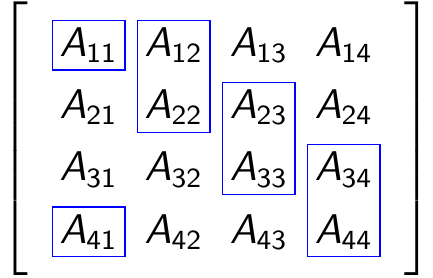} &
		\includegraphics[width=1.4in]{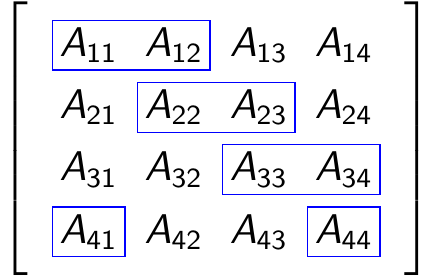} &
		\includegraphics[width=1.4in]{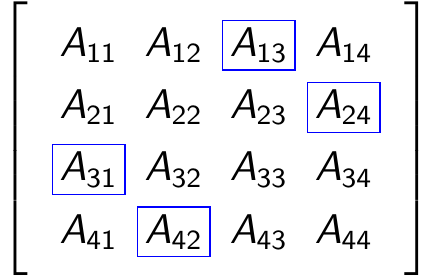}& 
		\includegraphics[width=1.4in]{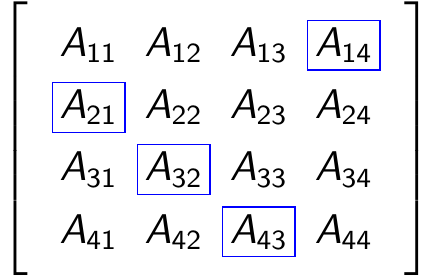} \\
		(a) $\Gcol_1$ (Step~\ref{step:init:row:lab}) &(b) $\Grow_1$ (Step~\ref{step:init:col:lab}) & (c) $G_2$ (Steps~\ref{step:local:Lambda:1},~\ref{step:row:LR:1}) &(d) $G_3$ (Step~\ref{step:local:Lambda:2})
	\end{tabular}
	\caption{The four stages of partitioning in Algorithm~\ref{alg:provable}. In each case, the collection of submatrices in the partition is given a name which is used in the text. We have used to shorthand $A_{qq'} = A^{(q,q')}$ for simplicity. Block used in obtaining initial labels~(a--b),  in obtaining the first local parameter estimates~(c), and in the first application of LR classifier~(d).} \label{fig:A:partitions}
\end{figure}
Steps~\ref{step:local:Lambda:1}--\ref{step:final:row:lab}  go through the following phases:
\begin{description}[wide, labelwidth=!, labelindent=\parindent, itemsep=4pt, topsep=2pt]
	\item[First local parameter estimates (step~\ref{step:local:Lambda:1}):] Having obtained good initial (global) row and column labels, in Step~\ref{step:local:Lambda:1}, we obtain estimates of the local mean parameters $\Lamh^{(q+2)}$ for the submatrices in $G_2$ as in Figure~\ref{fig:A:partitions}(c). Note for example, that $\Lamh^{(q+2)}$ computed in this step depends on blocks $A^{(q,q+2)}$ and on $G_1$ through $\zt^{(q+2)}$. Collectively, the estimates $\{\Lamh^{(q+2)}:\; q \in \Z_4\}$ in Step~\ref{step:local:Lambda:1} depend on $G_1 \cup G_2$ portion of $A$.
	
	\item[First LR classifier  (steps~\ref{step:row:LR:1}--\ref{step:similar:z:1}):] Using the estimates of the (local) row mean parameters, in Step~\ref{step:row:LR:1}, we apply the LR classifier, $\yt^{(q)}\gets \LR(A^{(q,q+2)}, \Lamh^{(q+2)}, \zt^{(q+2)})$ to each of the submatrices in $G_2$ (in Figure~\ref{fig:A:partitions}(c)). Here, $\Lamh^{(q+2)}$ depends on the same block $A^{(q,q+2)}$ on which we apply LR classifier, but the dependence is not problematic due the uniform consistency  of LR classifier in parameters (Lemma~\ref{lem:unif:param}). However, we note that $\zt^{(q+2)}$ is a function of $G_1$ blocks of $A$, hence independent of $A^{(q,q+2)}$ which is key in our arguments. We will similarly apply the LR classifier on the columns of $G_2$, and obtain $\zt^{(q)}$. By the end of step~\ref{step:similar:z:1}, the updated labels $\yt$ and $\zt$ will depend on blocks in $G_1 \cup G_2$; these labels will be much more accurate ($\mis \approx \exp(-I/Q)$) than the initial labels obtained by spectral clustering.
	
	\item[Second parameter estimates (steps~\ref{step:local:Lambda:2}--\ref{step:global:Lambda}):]
	Using the more accurate labels of step~\ref{step:similar:z:1}, we obtain the local mean parameters
	$\hat\Lambda^{(q+3)}$ in step~\ref{step:local:Lambda:2} for the submatrices in $G_3$ (Figure~\ref{fig:A:partitions}(d)). This step is similar to step~\ref{step:local:Lambda:1}, but due to the much more accurate labels, the parameter estimates are much more accurate as well.
	%
	Since the global mean parameter is the sum of local mean parameters, i.e. $\Lambda=\sum_{q\in[Q]}\Lambda^{(q)}$, we use $\hat\Lambda:=\sum_{q}\hat\Lambda^{(q)}$ to estimate $\Lambda$ in step~\ref{step:global:Lambda}. We recall that the true local mean parameters do not depend on the block row index; see~\eqref{eq:true:local:mean:def}.
	
	\item[Second LR classifier  (step~\ref{step:row:LR:2}):] Using the more accurate estimates of (global) row mean parameters $\hat\Lambda$ from step~\ref{step:global:Lambda} and the more accurate labels $\zt$ in step~\ref{step:similar:z:1}, in step~\ref{step:row:LR:2} we apply the LR classifier $\yh_\top\gets\LR(A_\top, \Lamh, \zt)$ on $A_\top$. We note that  $A_\top$ in this step is independent of $\zt$ (as well as $\Lamh$). This second LRC application is what brings us from very accurate labels ($\mis \approx \exp(-I/Q)$) to almost optimal ($\mis \approx \exp(-I)$); see Section~\ref{sec:analysis:alg:three}.
	
	\item[Bottom half (steps~\ref{step:swap:top:bot}--\ref{step:final:row:lab}):] The same process is repeated in step~\ref{step:swap:top:bot}, after swapping the top and bottom halves of $A$, to get the bottom portion of the row labels. 
	Matching the labels $\yt$ from the spectral clustering in step~\ref{step:consist} with $\yh_\top$ guarantees that $\yh_\top$ and $\yh_\bottom$ have consistent labels. Thus, no extra matching is required when concatenating the two in step~\ref{step:final:row:lab}.
	
\end{description}


\subsection{Matching step}\label{sec:matching:step}
Let us  describe the details of the matching step in Algorithm~\ref{alg:provable}. Although, the idea is intuitively clear, formally describing the procedure is fairly technical. In order to understand the idea,  consider the two-block labels $\yt^{(q-1,q)} := \cvecc{\yt^{(q-1)}}{\yt'^{(q)}}$, for $q=2,3$, that is,
\begin{align*}
\yt^{(1,2)} := \cvecc{\yt^{(1)}}{\yt'^{(2)}}, \quad \yt^{(2,3)} := \cvecc{\yt^{(2)}}{\yt'^{(3)}}.
\end{align*}
We will detail how these two sets of labels can be fused together to generate a set of consistent labels for the three-block true label vector $\y^{(1,2,3)} := [\y^{(1)};\y^{(2)};\y^{(3)}]$. The two (overlapping) two-blocks of the true label vector are also denoted as 
\begin{align*}
\y^{(1,2)} := \cvecc{\y^{(1)}}{\y^{(2)}}, \quad \y^{(2,3)} := \cvecc{\y^{(2)}}{\y^{(3)}}.
\end{align*}
More generally, we let $\y^{(q-1,q)} =  \cvecc{\y^{(q-1)}}{\y^{(q)}}$, similar to the notation for estimated blocks.

Recall our notation $\sigma^*(\cdot~\to~\cdot)$ for (an) optimal permutation between two sets of labels (cf. Section~\ref{sec:notation}). Finding $\sigma^*$ is a linear assignment problem, with computational complexity $O(\Kr\vee\Kc)^3)$ \cite{burkard1999linear}. 
Let us define 
\begin{align}\label{eq:perm:notations}
\sigma_{q-1,q} := \sigma^*\big(\yt^{(q-1,q)} \to \y^{(q-1,q)} \big), \quad 
\sigma_q := \sigma^*(\yt^{(q)} \to \y^{(q)}), \quad \sigma'_q := \sigma^*(\yt'^{(q)} \to \y^{(q)}).
\end{align}
Thus, for example we have
\begin{align*}
\sigma_{1,2} = \sigma^*(\yt^{(1,2)}~\to~\y^{(1,2)}), 
\quad  \sigma_2 = \sigma^*(\yt^{(2)} \to \y^{(2)}), 
\quad \sigma_3' = \sigma^*(\yt'^{(3)} \to \y^{(3)}), 
\end{align*}
and so on, as depicted in Figure~\ref{fig:matching}(a). In other words, each of these permutations is the optimal permutation from the corresponding block of the underlying estimated label to that of the truth. Let us write $\yt^{(1,2)} \approx \y^{(1,2)}$ to mean that the two sets of labels are sufficiently close (to be made precise later). 

The first claim is that $\yt^{(1,2)} \approx \y^{(1,2)}$ implies that the underlying sub-blocks have the same optimal permutation to the truth as the original two-block label, i.e.,
\begin{align*}
\yt^{(1,2)} \approx \y^{(1,2)} \implies \sigma_1 = \sigma_2' = \sigma_{1,2}
\end{align*}
and similarly $\yt^{(2,3)} \approx \y^{(2,3)} \implies \sigma_2 = \sigma_3' = \sigma_{2,3}$. The second claim is that each sub-block has ``almost'' the same misclassification error as the bigger two-block. To see this, recall the \emph{direct misclassification rate} introduced in Section~\ref{sec:notation}, i.e., misclassification rate without applying any permutation (or equivalently with the identity permutation). We have
\begin{align}\label{eq:Mis:23:eps:bound}
\dmis\big( \sigma_{2,3}(\yt^{(2,3)}),\y^{(2,3)} \big) = 	\mis\big( \yt^{(2,3)} ,\y^{(2,3)} \big) \le \eps.
\end{align}
where the  inequality is by assumption ($\eps$ being the rate achieved by the spectral clustering algorithm). A similar expression holds with $(2,3)$ replaced with $(1,2)$.
Now~\eqref{eq:Mis:23:eps:bound} implies
\begin{align}\label{eq:Mis:2:eps:bound}
\dmis\big( \sigma_{2}(\yt^{(2)}),\y^{(2)} \big)  = \dmis\big( \sigma_{2,3}(\yt^{(2)}),\y^{(2)} \big) \le 2\eps = \eps'
\end{align}
where the  equality uses $\sigma_2 = \sigma_{2,3}$. To see the inequality, let $n_{2,3}$, $n_2$ and $n_3$ be the lengths of $\y^{(2,3)}$, $\y^{(2)}$ and $\y^{(3)}$. Then, 
\begin{align*}
\dmis\big( \sigma_{2,3}(\yt^{(2,3)}),\y^{(2,3)} \big) = 
\frac{n_2}{n_{2,3}} \dmis\big( \sigma_{2,3}  (\yt^{(2)}),\y^{(2)} \big) + 
\frac{n_3}{n_{2,3}} \dmis\big( \sigma_{2,3}(\yt'^{(3)}),\y^{(3)} \big)
\end{align*}
and the result follows since we have $n_2 = n_3 = n_{2,3}/2$ by construction.
Note that $\dmis$ has the property of being easily distributed over sub-blocks as opposed to $\mis$. Similarly to~\eqref{eq:Mis:2:eps:bound}, we obtain $\dmis\big( \sigma'_3(\yt^{(3)}),\y^{(3)} \big) \le \eps'$ considering the second component of $\yt^{(2,3)}$ and $\y^{(2,3)}$. Applying the same argument to indices $(1,2)$, we conclude similarly that $\dmis\big( \sigma_1(\yt^{(1)}),\y^{(1)} \big) \le \eps'$ and $\dmis\big( \sigma'_2(\yt'^{(2)}),\y^{(2)} \big) \le \eps'$.

\begin{figure}[t]
	\centering
	\begin{tabular}{cc}
		\includegraphics[width=2.75in]{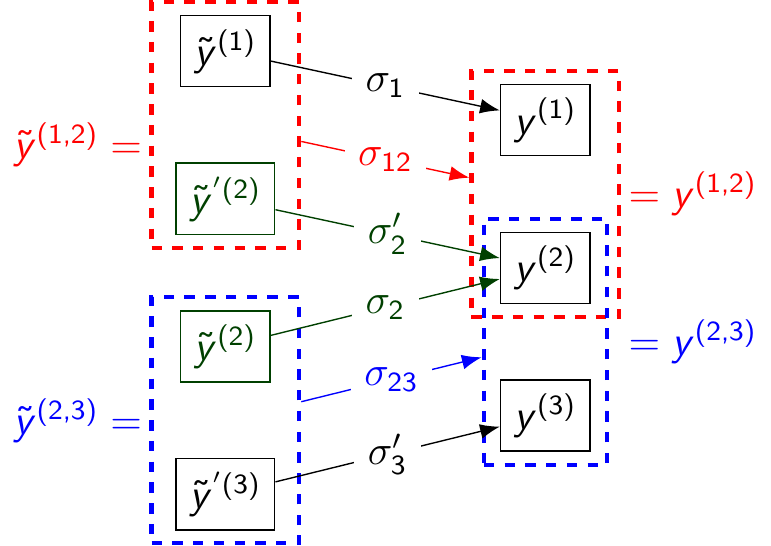} \hskip1ex &
		\raisebox{.25in}{\includegraphics[width=1.8in]{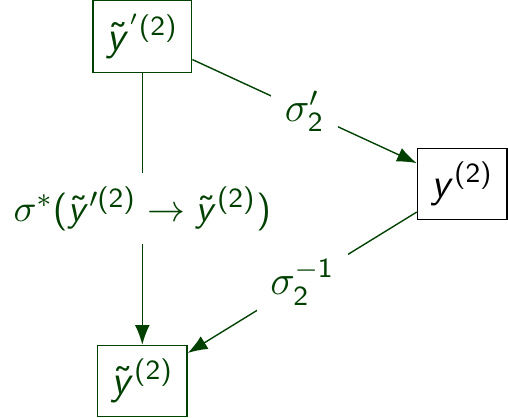}}\\
		(a) & (b)
	\end{tabular}
	\caption{Pictorial depiction of the matching step. (a) Two-block and sub-block optimal permutations to the truth. When  $\yt^{(1,2)} \approx \y^{(1,2)}$, we have~$ \sigma_1 = \sigma_2' = \sigma_{1,2}$ and similarly $\yt^{(2,3)} \approx \y^{(2,3)}$ implies $\sigma_2 = \sigma_3' = \sigma_{2,3}$. (b) Commutative diagram depicting how the missing permutation $\sigma_2^{-1} \circ \sigma'_2$ can be obtained by matching observed labels $\yt'^{(2)}$ and $\yt^{(2)}$. See Section~\ref{sec:matching:step} for details.}
	\label{fig:matching}
\end{figure}

Now consider the following three block vector undergoing transformation 
\begin{align*}
\renewcommand\arraystretch{1.2}
\begin{bmatrix}
\sigma_1(\yt^{(1)}) \\  \,\sigma_2(\yt^{(2)}) \\ \,\sigma'_3(\yt^{(3)})
\end{bmatrix} \to 
\begin{bmatrix}
\sigma_2^{-1} \circ \sigma_1(\yt^{(1)}) \\  \sigma_2^{-1} \circ \sigma_2(\yt^{(2)}) \\ \sigma_2^{-1} \circ \sigma'_3(\yt^{(3)})
\end{bmatrix} \stackrel{=}{\to}
\begin{bmatrix}
\sigma_2^{-1} \circ \sigma_1(\yt^{(1)}) \\  \yt^{(2)} \\ \yt^{(3)}
\end{bmatrix}
\stackrel{=}{\to}
\begin{bmatrix}
\sigma_2^{-1} \circ \sigma'_2(\yt^{(1)}) \\  \yt^{(2)} \\ \yt^{(3)}
\end{bmatrix}.
\end{align*}
The leftmost vector has $\dmis$ of at most $\eps'$ relative to $y^{(1,2,3)}$ by the previous arguments, and since $\mis \le \dmis$, we have the same bound  on $\mis$ rate for the leftmost vector. The first transformation keeps the same $\mis$ rate since we are applying a single permutation $\sigma_2^{-1}$ to all elements. The second transformation is in fact an equality, using $\sigma_3' = \sigma_2$ established earlier. The third transformation/equality follows similarly by  $\sigma_1 = \sigma'_2$. Thus, if we can recover $\sigma_2^{-1} \circ \sigma'_2$ from data, we can construct a consistent three-block label having $\mis \le \eps'$.

The third and final claim is that this is possible, and in fact we have
\begin{align}
\sigma_2^{-1} \circ \sigma'_2 = \sigma^*(\yt'^{(2)} \to \yt^{(2)}) \label{eq:sigma:sigma':identity:subb2}
\end{align}
that is, $\sigma_2^{-1} \circ \sigma'_2$ can be obtained (assuming $\eps'$ is sufficiently small) by optimally matching $\yt'^{(2)}$ to $\yt^{(2)}$, both of which we observe in practice. See the commutative diagram in Figure~\ref{fig:matching}(b). In order to make the above argument precise, we need to justify the first and third claims. We will discuss the details in Section~\ref{sec:matching:analysis}. The above matching process can be repeated over all the two-blocks $\yt^{(q-1,q)}$ to get a consistent set of global labels whose overall misclassification rate is no more than twice that of the original two-blocks (cf. $\eps'$ versus $\eps$).

\subsection{Results for Algorithm~\ref{alg:provable}}
\subsubsection{General initialization}\label{sec:gen:init}
Before studying the spectral initialization, let us give a general bound on the misclassification rate of Algorithm~\ref{alg:provable}, assuming sufficiently good  initial labels. Assume that the initial labels obtained in steps~\ref{step:init:row:lab} and~\ref{step:init:col:lab} of the algorithm are $\gamma_1$-good in the sense of~\eqref{assump:missclass}, with $\gamma_1$ satisfying
\begin{align}\label{assump:small:gamma:one}
\gamma_1 \;\le\; \frac{1}{384 \beta^2 \omega} \Big( \f {\Imin}{8 \Kc \Laminf} \wedge   \f {\Imincol}{16 \Kr\Gaminf} \Big) .
\end{align}
Any other initialization algorithm besides spectral clustering can be used, as long as the above guarantee on its output holds. We also  need the following weaker version of~\eqref{assump:bounded:J:strong}:
\begin{align}
\beta\omega(\Laminf\vee\Gaminf) = o\lp \Big[\f{\Imin\wedge\Imincol}{Q \log Q (\Kr\vee\Kc)} \Big]^a\rp, \text{ for some }a>0. \tag{A4$'$} \label{assump:bounded:J}
\end{align}


\begin{thm}\label{thm:without:spec:clust}
	Assume that the model parameters satisfy $\Imin \wedge \Imincol\to\infty$, $\Lammin\to\infty$, \eqref{assump:Lambda}, \eqref{assump:balance}, \eqref{assump:sparse:network} and~\eqref{assump:bounded:J}, and the initial labels satisfy~\eqref{assump:small:gamma:one}.
	%
	Then, for some $\zeta=o(1)$, $\yh$ outputted by Algorithm~\ref{alg:provable} satisfies
	\begin{align}\label{eq:rate:without:spec:clust}
	\mis_k \big(\yh,\,\y\big) 
	=  O \lp {\omega}\sum_{r \neq k} \lp 1+\f 1{\epsi_{kr}} \rp \exp \big({-} I_{kr} - \lp \f 12 -\zeta \rp \log\Lammin\big) \rp
	\end{align}
	for every $k\in[\Kr]$ with probability $1-o(1)$. 
\end{thm}

We refer to Section~\ref{main:res} for the definition of the parameters involved in the rate given in~\eqref{eq:rate:without:spec:clust}.

\begin{algorithm}[t]
	\caption{\scerr}
	\setstretch{1.2}
	\begin{algorithmic}[1]
		\medskip
		\State Apply degree-reduction regularization $A$ to obtain $\Are$, as discussed in~\cite{zhou2018spectral}.
		\State Obtain $\Are^{(\Kr\wedge\Kc)} = \Zh_1 \Sigh \Zh_2^T$, the $(\Kr\wedge\Kc)$-truncated SVD of $\Are$.
		\State Output $\kalg(\Zh_1 \Sigh)$ where $\kalg$ is an isometry-invariant $\kappa$-approximate $k$means algorithm.
	\end{algorithmic}
	\label{alg:scerr}
\end{algorithm}

\subsubsection{Spectral initialization}\label{sec:spectral:init}
Theorem \ref{thm:without:spec:clust} requires initial labels that satisfy~\eqref{assump:small:gamma:one}. 
A spectral algorithm, namely \scerr given in Algorithm~\ref{alg:scerr} can provide such initialization. (The acronym stands for reduced rank efficient spectral clustering.) The algorithm is presented and analyzed in~\cite{zhou2018spectral}. One performs a truncated SVD of rank $r := \Kr \wedge \Kc$ on a regularized version of the adjacency matrix $\Are$ to obtain $\Zh_1 \Sigh \Zh_2^T$, where $\Sigh$ is the diagonal matrix retaining the top $r$ largest singular values of $\Are$, and $\Zh_1 \in \reals^{\nr \times r}$ and $\Zh_2 \in \reals^{\nr \times r}$ are the matrices of the corresponding singular vectors. One then runs a $k$-means type algorithm on $\Zh_1 \Sigh$, rather than $\Zh_1$ which is the more common approach in spectral clustering. This allows one to state consistency results without a reference to the gap in the spectrum of $\Are$, while still retaining the attractive feature of the latter approach, namely, the computational efficiency of running $k$-means on a matrix of reduced dimension. The ``isometry-invariant'' qualification used in Algorithm~\ref{alg:scerr} means that the $k$-means algorithm should only be sensitive to the pairwise distances of the data points.  We refer to~\cite{zhou2018spectral} for a detailed discussion. In particular, one has the following bound on the misclassification rate of \scerr:
\begin{thm}[\cite{zhou2018spectral}]\label{thm:spectral:clustering}
	Let $\alpha =  \nc / \nr$ and $\Lamwedge^2 :=  \min_{t \neq s} \norm{\Lambda_{s*} - \Lambda_{t*}}^2$.
	Consider the spectral algorithm given in Algorithm~\ref{alg:scerr}, assume \eqref{assump:balance} and that for a sufficiently small $C_1 > 0$,
	\begin{align}\label{eq::asump:Lambda}
	\beta^2 \Kr\Kc(\Kr\wedge\Kc) \,\alpha \frac{\infnorm{\Lambda}}{\Lamwedge^2} \le C_1 (1+\kappa)^{-2}.
	\end{align}
	Then the algorithm outputs estimated row labels $\yt$ satisfying
	\begin{align*}
	\mis(\yt, \y) \; \le  \;C_1^{-1}\, (1+\kappa)^2
	\beta \Kc(\Kr\wedge\Kc) \, \alpha \Big(\frac{\infnorm{\Lambda}} {\Lamwedge^{2}} 	\Big).
	\end{align*}
\end{thm}
Here, $\Lambda_{s*}$ refers to the $s$th row the mean parameter matrix $\Lambda$ (cf. Section~\ref{sec:bi:sbm}).
Combining Theorems~\ref{thm:without:spec:clust} and~\ref{thm:spectral:clustering}, one obtains Theorem~\ref{thm:main:res}. Some work is required to translate the bound of the Theorem~\ref{thm:spectral:clustering} to be applicable to sub-blocks. See Section~\ref{sec:proof:main:res} for details.

\section{Simulations}\label{sec:sims}

We provide some simulation results to corroborate the theory. We generate from the SBM model of Section~\ref{sec:bi:sbm}  with the following connectivity matrix
\begin{align}\label{eq:P:sims}
P = C \frac{\big[\log(\nc \nr)\big]^\alpha}{\sqrt{\nc  \nr}} B, \quad B =
\begin{bmatrix}
1 &2 &3 &4 &5 &6 \\
2 &3 &4 &5 &6 &1 \\
3 &4 &5 &6 &1 &2 \\
4 &5 &6 &1 &2 &3 \\
\end{bmatrix}.
\end{align}
Note that $B$ does not have any clear assorative or dissortative structure.
We let $\nr = \Kr n_0$ and $\nc = \Kc n_0$, and we vary $n_0$. All clusters (both row and column) will have the same number of nodes $n_0$. By changing $\alpha$, we can study different regimes of sparsity. In particular, when $\alpha \in (0,1)$, we are in the regime where weak recovery is possible but not exact (or strong) recovery.   We consider both the misclassification rate, and the normalized mutual information (NMI) as measures of performance. NMI is a measure of accuracy which is between 0 and 1~(=perfect match). The NMI is quite sensitive to mismatch and tends to reveal discrepancies between methods more clearly.
Figure~\ref{fig:nmi:miss:plots}(a) shows the overall NMI versus $n_0$. 
Figure~\ref{fig:nmi:miss:plots}(b) illustrates the corresponding log. misclassification rates. 

\begin{figure}[t]
	\centering
	\begin{tabular}{cc}
		\includegraphics[width=2.5in]{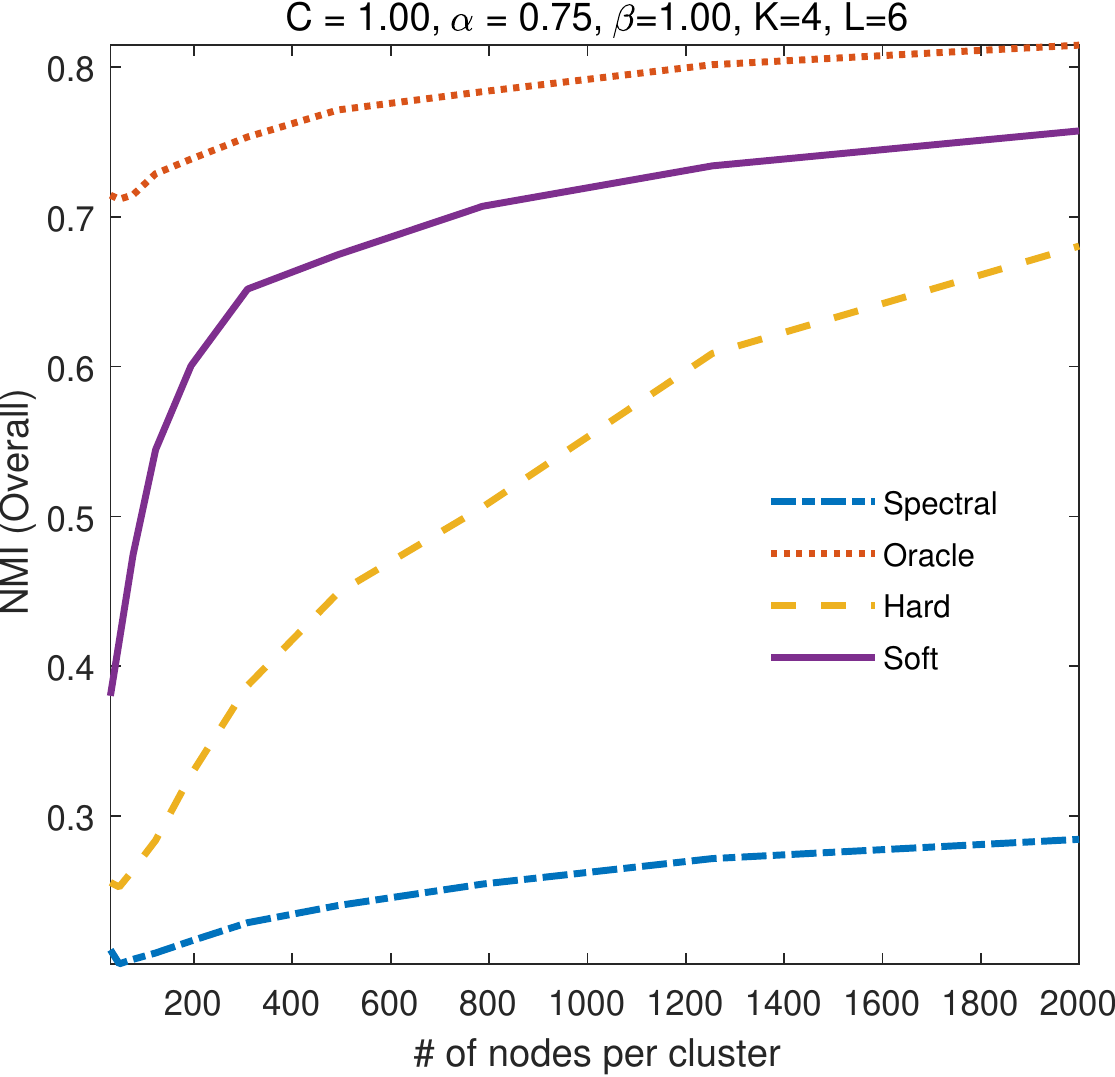} &
		\includegraphics[width=2.5in]{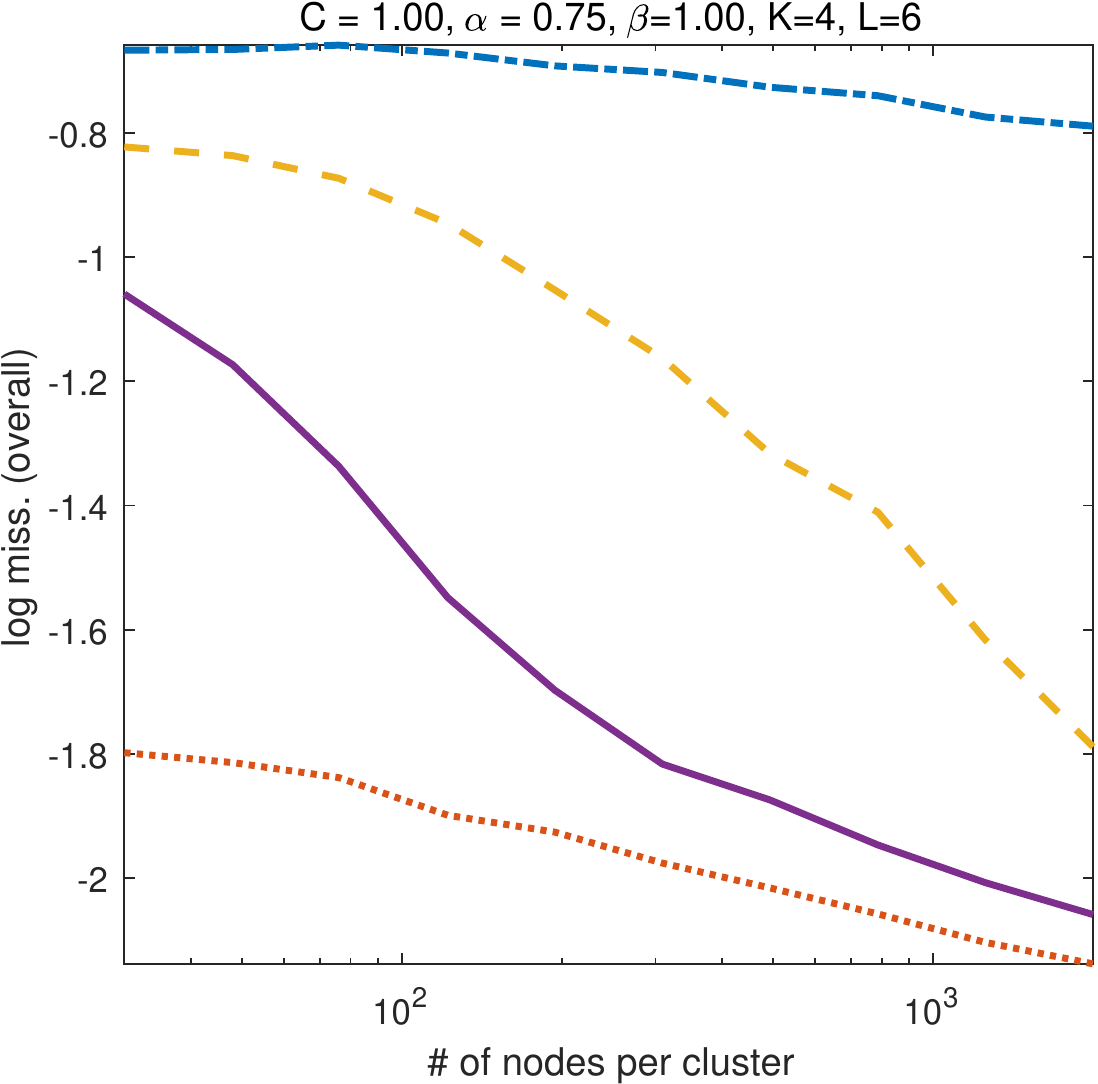} \\
		(a) & (b)
	\end{tabular}
	\caption{  Plots of (a) the (overall) NMI and (b) the corresponding log. misclassification rate, for the SBM model with connectivity matrix~\eqref{eq:P:sims}. The four algorithms considered are  the \texttt{Spectral} clustering of Algorithm~\ref{alg:scerr}, \texttt{Soft} and \texttt{Hard} versions of Algorithm~\ref{alg:pl:biclust} and the \texttt{Oracle} algorithm of Section~\ref{sec:oracle}. }
	\label{fig:nmi:miss:plots}
\end{figure}

We have considered four algorithms: 
\begin{enumerate*}[(1)]
	\item \texttt{Spectral}: the spectral clustering of Algorithm~\ref{alg:scerr}.
	\item \texttt{Soft}:  Algorithm~\ref{alg:pl:biclust} with flat prior, no inner loop and no conversion to hard labels.
	\item \texttt{Hard}: Algorithm~\ref{alg:pl:biclust} with flat prior, no inner loop and conversion to hard labels after each label computation.
	\item  \texttt{Oracle}: The oracle classifier discussed in Section~\ref{sec:oracle} and Remark~\ref{rem:oracle:problem}:  Assuming the knowledge of $\z$ and $\Lambda$, we obtain $\yh$ by the likelihood ratio classifier, and similarly obtain $\zh$, assuming the knowledge of $\y$ and $\Gamma$.
\end{enumerate*}

Figure~\ref{fig:nmi:miss:plots} shows the results for $\alpha =.75$ (regime where no exact recovery is possible) and $C=1$. Both the soft and hard versions of Algorithm~\ref{alg:pl:biclust} are initialized with the spectral clustering  and both significantly improve over it. The soft version of Algorithm~\ref{alg:pl:biclust} also  outperforms the hard version as one would expect: soft labels carry more information between iterations. It is  interesting to note that the slope for the log. misclassification rate of Algorithm~\ref{alg:pl:biclust} approaches that of the oracle (esp. clear for the soft version in Figure~\ref{fig:nmi:miss:plots}(b))  as predicted by the theory. 
Simulation results for various other settings can be found in Appendix~\ref{sec:extra:simulations}, showing qualitatively similar behavior.

\section*{Acknowledgement}
We thank Yunfeng Zhang for helpful discussions.

\printbibliography

\begin{center}
	\Large \textbf{Supplement: Proofs of the results}
\end{center}

This supplement 
contains proofs of the results and additional commentary and simulation results. It is organized as follows: In Section~\ref{sec:add:comments}, we provide additional comments on the results. The details of the matching step in Algorithm~\ref{alg:provable} are presented in Section~\ref{sec:matching:step}. A preliminary analysis is provided in Section~\ref{sec:preliminary:analysis}, presenting three key lemmas (Section~\ref{sec:fixed:label:analysis}), as well as other useful indeterminate tools. Sections~\ref{sec:analysis:alg:three},~\ref{sec:proof:other:main:res} and~\ref{sec:proof:minimax} contain the proofs of the main results of the paper. The proofs of the three key lemmas are given in Section~\ref{sec:proof:main:lemmas}. The remaining proofs are given in the appendices.

\section{Additional comments}\label{sec:add:comments}

\subsection{Comments on edge splitting}\label{sec:edge:splitting:comments}
One needs independent versions of the adjacency matrix in different stages of the algorithm. To achieve this goal, \emph{edge splitting}  was introduced in \cite{abbe2015community}. The idea is that one can regard the two (or more) graphs obtained from edge splitting to be nearly independent. To be specific, let $\P_1$ be the joint probability measure corresponding to a pair of graphs $G_1$ and $G_2$ generated independently with connectivity matrices $qP$ and $(1-q)P$. Let $\P_2$ be the joint probability measure on $G_1$ and $G_2$ obtained by edge splitting from a single SBM with connectivity matrix $P$, assigning every edge independently to either $G_1$ or $G_2$ with probabilities $q$ and $1-q$. Then, $\P_1$  and $\P_2$ have the same marginal distributions. Having a vanishing total variation between $\P_1$ and $\P_2$ is necessary for further analysis which, as was pointed out by~\cite[pp.~46-47]{abbe2015community}, is equivalent to showing 
that under $\pr_1$, $G_1$ and $G_2$ do no share any edge, with high probability.
Letting $\Pt_{\min}=\min_{k\ell}\Pt_{k\ell}$, 
\begin{align*}
\P_1(G_1\text{ and }G_2\text{ do not share edges})
\le \lp 1 - \f{(1-q)q\Pt_{\min}^2(\log n)^2}{n^2} \rp^{n^2}
\end{align*}
which is strictly bounded away from $1$ unless $(1-q)q \Pt_{\min}^2 (\log n)^2=o(1)$, that is, the connectivity matrix of either $G_1$ or $G_2$ should vanish faster than $1/n$. Our consistency result will not hold in this regime. Thus, edge splitting cannot be used to derive the results in this paper, and we introduce the block partitioning idea to supply us with the independent copies necessary for analysis.
Another technical issue about edge splitting is discussed in Remark~\ref{rem:edge:splitting}.

\subsection{Discussion}\label{sec:discussion:theory}
Our results do not directly apply to the symmetric case, due to the dependence between the upper and lower triangular parts of the adjacency matrix $A$. However, a more sophisticated two-stage block partitioning scheme can be used to derive similar bounds under mild extra assumptions. One starts with an asymmetric partition into blocks of sizes $\{q n, (1-q)n\} \times \{qn ,(1-q)n\}$, for $q = 1/Q \to 0$ very slowly. In the first stage, one applies a similar procedure as described in Algorithm~\ref{alg:provable} on the upper triangular portion of the large subblock $(1-q) n \times (1-q) n$, followed by the application of the LR classifier on the fat block $q n \times (1-q) n$ to obtain very accurate row labels of the small block $qn \times qn$.. One then repeats the process using the ``leave-one-out'' of~\cite{gao2017achieving}, but applied to small blocks $qn \times qn$ rather than individual nodes. We leave the details for a future work. 
%

It was also shown by~\cite[Theorem~5]{gao2017achieving} that their equivalent of condition~\eqref{assump:Lambda} can be removed by modifying the algorithm. In their setting, without assuming $a\asymp b$, a misclassification rate of $\exp(-(1-\eps)I)$ is achievable, where $\eps\in(0,1)$ is a variable in the new version of their algorithm. If those arguments can be extended to the general block model, it will be possible to relax the requirements on $\omega$ in \eqref{assump:sparse:network} and \eqref{assump:bounded:J:strong}. When $\Kr, \Kc = O(1)$, one can completely remove sparsity condition~\eqref{assump:sparse:network} using a much sharper Poisson-binomial approximation than what we have used in this paper. Finally, we suspect that our result could be generalized beyond SBMs to biclustering arrays where the row and column sums over clusters follow Poissonian central limit theorems. We will explore these ideas in the future.

\subsection{PL naming}
We have borrowed the name pseudo-likelihood (PL) from~\cite{amini2013pseudo} based on which the algorithms in this paper are derived. In~\cite{amini2013pseudo}, the setup is that of the symmetric SBM, and in order to treat the full likelihood as the product of independent (over nodes $i=1,\dots,n$) of the mixture of Poisson vectors, one has to ignore the dependence among the upper and lower triangular parts of the adjacency matrix, making the PL naming more inline with the traditional use of the term. In our bipartite setup, there is no such dependence to ignore, but we have kept the name PL for consistency with~\cite{amini2013pseudo} and ease of use. We interpret the ``pseudo'' nature of the likelihood as the approximation used in the block compression stage (with imperfect labels) and in replacing Poisson-binomial distribution with the Poisson.

\section{Preliminary analysis}
\label{sec:preliminary:analysis}

We start by analyzing the properties of the operators introduced in Sections~\ref{sec:local:global:param} and~\ref{sec:gen:pl:alg},  for some fixed (deterministic) initial labels $\yt$ and $\zt$. We assume that these labels satisfy: 
%
%
\begin{align}
\mis(\yt,\y)\le \f\gamma{\beta\Kr}, \quad \mis(\zt,\z)\le \f\gamma{\beta\Kc}. \tag{B3}\label{assump:missclass}
\end{align}
We  call such labels \emph{$\gamma$-good}.
Throughout, $\Lamt$ will be used to denote a generic deterministic approximation of the true row mean parameter $\Lambda$. 
%
The \emph{relative $\ell_\infty$ ball} of radius $\delta$ centered at $\Lambda$, that is,
\begin{align}\label{eq:Blam:def}
\Bcl(\delta):=\{\Lamt: \infnorm{\Lamt-\Lambda}\le \delta \infnorm{\Lambda}\},
\end{align}
will play a key role in our arguments. For sufficiently small $\delta$ and true $\Lambda$,  $\Bcl(\delta)$ will be the set of \emph{$\delta$-good} row mean parameters.



\subsection{Fixed label analysis}\label{sec:fixed:label:analysis}
We first present the analysis assuming that all the operations are performed on the entire adjacency matrix $A$. In Section~\ref{sec:subblk:analysis}, these results are extended to be applicable to sub-blocks of $A$. Recall the definitions of the mean parameters an their estimates from Section~\ref{sec:bi:sbm}. In particular, we recall that $\lambda_{k*}(\y,\zt)$ is the mean of $b_{i\ast}(\zt)$ for any node $i$ with $\y_i = k$. These mean parameters form the $k$th row of $\Lambda(\y,\zt)$. Our first main lemma illustrate that whenever the initial labels $\zt$ and $\yt$ are $\gamma$-good,
then the parameters $\Lambda(\y,\zt)$  as well as the corresponding estimates $\Lamh$ defined in~\eqref{eq:Lop:Lamh} are close to the truth, that is $\Lambda$.



\begin{lem}[Parameter consistency] \label{lem:param:consist}
	Let $C_\gamma = C_{\gamma,\beta} = \beta^2 \gamma/(1-\gamma)$, assume that $6 C_\gamma \omega \le 1$, and let $h_c(\tau):=\f 3{4c}\tau\log\left(1+\f{2c}3 \tau\right)$. Then under assumptions~\eqref{assump:Lambda}, \eqref{assump:balance} and~\eqref{assump:missclass}, we have
	\begin{enumerate}[label=(\alph*), ref={\thelem(\alph*)}]
		\item $\infnorm{\Lambda(y,\zt) - \Lambda} \le C_\gamma \infnorm{\Lambda}$, \quad $\infnorm{\Lambda(y,\zt)}\le 2\infnorm{\Lambda}.$
		\label{lem:param:consist:a}
		\item $\infnorm{\Lambda(\yt,\zt) - \Lambda(y,\zt)} \le 2\gamma \infnorm{\Lambda}$, \quad  $\infnorm{\Lambda(\yt,\zt)}\le 4\infnorm{\Lambda}.$
		\label{lem:param:consist:b}
		\item $\infnorm{\Lamh - \Lambda(\yt,\zt)} \le 4  \tau \infnorm{\Lambda}$ with probability at least $1 -  2 \ptail_{1}$ where
		\begin{align}\label{eq:ptail:Lambh}
		\ptail_{1} = \ptail_{1}(\tau;\,\nr,\Lamin,\beta) := \Kr \Kc \exp\Big({-} \frac{\nr\Lammin \,h_1(\tau)}{4\beta \Kr} \Big), \quad \forall \tau > 0,
		\end{align}
		\label{lem:param:consist:c}
		\item $\infnorm{\Lambda(y,\zt) - \Lambda} \le C_\gamma \infnorm{\Lambda}$, \quad
	\end{enumerate}
	and $\Lamh$ is as defined in~\eqref{eq:Lop:Lamh}.
	In particular, all the estimates $\Lambda(y,\zt)$, $\Lambda(\yt,\zt)$ and $\Lamh$ are within relative $\ell_\infty$ distance of at most $4(C_\gamma + \tau)$ from $\Lambda$.
\end{lem}
The lemma is proved in Section~\ref{sec:proof:param:consist}. Note that the lemma implies that
$\Lamh\in\Bcl(4(C_\gamma+\tau))$ with the stated probability.

\medskip
Our second key lemma shows that the LR classifiers in~\eqref{eq:LRC:def} are uniformly dominated, over $\Lamt \in \Bcl(\delta)$, by a single (perturbed) classifier. To state this result, recall the block compression  $\bb(\zt):=\Bc(A;\zt)$ given in~\eqref{eq:Bc:def}, and define the following:  
\begin{align}
Y_{ikr}(b_{i\ast}, \Lamt) 
&:= \poillr(b_{i*}; \lamt_{r*} \mid \lamt_{k*}) = \sum_{\ell=1}^\Kc b_{i\ell} \log \frac{\lamt_{r\ell}}{\lamt_{k\ell}} + \lamt_{k\ell} - \lamt_{r\ell}, \label{eq:Yikr:def} \\
Z_{ik}(b_{i\ast}, \Lamt) 
&:=  1\{Y_{ikr}(\Lamt) \ge 0, \; \text{for some}\; r \neq k\}. \label{eq:Zik:def} \\
S_k(\bb, \Lamt) 
&:= \frac1{n_k(y)} \sum_{i: \y_i = k} Z_{ik}(b_{i\ast},\Lamt),	\label{eq:Sk:def} 
\end{align}
where $\poillr$ is the Poisson log-likelihood ratio defined in~\eqref{eq:joint:poi:llr}. Thus,
$Y_{ikr}$ is the (pseudo) log-likelihood ratio, for $k, r \in [\Kr]$,
measuring the relative likelihood of row $i$ having label $k$.
We note that $Y_{ikr}(\Lamt) < 0, \forall r \neq k$ implies $\yh_i:=(\LR(A, \Lamt, \zt))_i  = k$. Thus, $S_{k}(b_{i\ast}, \Lamt) $ is the misclassification rate for the LR classifier over the $k$th row-class, i.e., $\mis_k(\yh,\y)$.
Let 
\begin{align}\label{eq:Jkr:Tkr:def}
\J_{kr} = \Kc \infnorm{\Lambda}/ I_{kr}. 
\end{align}
Recalling definitions of $\eps_{kr}$, $\omega$ and $\beta$ from Section~\ref{main:res}, set
\begin{align}
\etap &:= \etap(\delta;\Lambda) = 8\omega \delta \Kc \infnorm{\Lambda} =  8\omega \delta \J_{kr}\I_{kr}, \label{eq:etap:def} \\
\begin{split}
\eta_{kr} &:= \eta_{kr}(\delta; \omega, \beta, \nc, \Lambda)\\
&= 21\delta\omega\Kc\Laminf + \frac{5\beta \Kc^2 \infnorm{\Lambda}^2}\nc 
+ \log \Big[ 11\omega\lp \f{1}{\epsi_{kr}-2\omega(1+\epsi_{kr})\delta } + 1 \rp \Big] -\f 12 \log \Lammin. 
\label{eq:eta:kr:def}
\end{split}
\end{align}

We have the following key lemma:
\begin{lem}[Uniformity of LRC in mean parameters]\label{lem:unif:param} Fix any row label $\zt$ and let $\bb = \bb(\zt)$ be the corresponding column compression. Let $\Lambda' = \Lambda(y,\zt)$ be the row mean parameter associated with $\bb$.
	%
	Assume (\ref{assump:Lambda}), (\ref{assump:balance}), and  $\Lambda'\in\Bcl(\delta)$ with $3 \, \omega\delta< 1$.
	%
	%
	Then, for all $k,r \in [\Kr],\; k \neq r$, and all $i :  y_i = k$, we have the following bounds:
	%
	\begin{enumerate}[label=(\alph*), ref={\thelem(\alph*)}]
		\item \label{lem:unif:param:a} With $\etap$ defined as in~\eqref{eq:etap:def},
		\begin{align}\label{eq:expo:bound:Y}
		\P \big(\,\exists \Lamt\in\Bcl(\delta), \;Y_{ikr}(b_{i*},\Lamt) \ge 0 \,\big) \; \le \; \exp(-I_{kr} + \etap).
		\end{align} 
		\item \label{lem:unif:param:b}If in addition  $\epsi_{kr}-2\omega\delta>0$,
		then with $\eta_{kr}$ defined as in~\eqref{eq:eta:kr:def},
		\begin{align}\label{eq:expo:bound:Y:stronger}
		\P\big(\, \exists \Lamt\in\Bcl(\delta), \;Y_{ikr}(b_{i\ast}, \Lamt) \ge 0 \,\big)
		\; \le\; \exp\big({-}I_{kr} + \eta_{kr} \big).
		\end{align}
		%
		%
	\end{enumerate}
\end{lem}
The proof of Lemma~\ref{lem:unif:param:b} appears in Section~\ref{sec:proof:unif:param:b}, and that of part~(a) in Appendix~\ref{sec:proof:unif:param:a}.

\begin{rem}[Typical setting]\label{rem:typical:rates}
	In the error exponent in Lemma~\ref{lem:unif:param:b}, i.e. $-I_{kr} +\eta_{kr}$, the first three terms in~\eqref{eq:eta:kr:def} are positive and constitute the undesirable part of the bound.	
	Our goal is to keep these terms dominated at the final stage of the algorithm, i.e., make them $o(\log \Lamin)$, by making $\delta$ sufficiently small. For now, let us introduce a simple \emph{typical setting} to give some idea of the order of $\eta_{kr}$. In the first reading, one can consider the case where $\beta, \omega = O(1)$, $I_{kr} \asymp I \to \infty$ for all $k,r$ and some $I$, and assume that $\Kc  \Laminf / I = O(1)$ and~\eqref{assump:bounded:J:very:strong} holds. In this setting, $J_{kr} = O(1)$ and we have $\eta_{kr} = C( \delta + \nc^{-1} I) I - \f 12\log\Lammin$ for some constant $C$. Keeping these typical orders in mind will be helpful in understanding the statements of the subsequent results. 
	
	It is also worth noting that we always have $J_{kr} \ge \frac12$. which follows from the general bound $I_{kr} \le 2 \Kc \infnorm{\Lambda}$. Another important quantity is $C_\gamma$ in Lemma~\ref{lem:param:consist}, which in the typical setting behaves as $C_\gamma \asymp \gamma$ when $\gamma \to 0$.
\end{rem}




Combining Lemma~\ref{lem:unif:param} with the Markov inequality, we can get uniform control on the misclassification rate of the $\LR$ classifier in its parameter argument (i.e., $\Lamh$):

\begin{lem}\label{lem:miss:markov}
	Fix $k \in [K]$ and  $\zt \in [\Kc]^{\nc}$. Let $\Lamh \in \reals_+^{K \times L}$ be any random matrix and set $\yh(\zt) :=  \LR(A,\Lamh,\zt)$. Assume that~\eqref{eq:expo:bound:Y} holds. Then, for any $u \in \reals$, we have 
	\begin{align*}
	\mis_k \big(\yh(\zt), y\big) \le   \sum_{r \neq k} \exp\big( {-}I_{kr} +\eta'  + u\big), 
	\end{align*}
	with probability at least $1- e^{-u} - \pr\big(\Lamh \notin \Bcl(\delta)\big)$. The result is also true if we replace $\eta'$ by $\eta_{rk}$ when \eqref{eq:expo:bound:Y:stronger} holds. 
\end{lem}

\begin{rem}\label{rem:edge:splitting}
	
	Edge splitting (ES) was proposed in~\cite{abbe2015community} to generate nearly independent copies from a single network.
	One might ask whether combining the edge splitting idea with Lemma~\ref{lem:miss:markov} is enough to give us a result similar to Theorem~\ref{thm:main:res}. In ES, edges are randomly assigned to two graphs $G_1$ and $G_2$, with probabilities $q$ and $1-q$. The new graphs $G_1$ and $G_2$ will follow a SBM  with a reduced connectivity matrix (by a factor of $q$ and $1-q$ respectively). Hence, the corresponding parameters $\Lambda$ and $\I$ are reduced by the same factor; for example $I$ will be scaled to $q I$ for $G_1$. Let us consider the typical setting where $\beta, \Kr, \Kc, \omega, \eps_{kr}=O(1)$ and $I_{kr}\asymp I$ for all $k$, $r$ and some $I$; assume the connectivity matrix is symmetric, i.e., $\Lambda=\Gamma$ and $I=\Icol$. 
	Let $\zt$ and $\yt$ be the labels obtained by performing biclustering on $G_1$. Lemma~\ref{lem:miss:markov} in the best case scenario, with the most favorable version of $\eta_{kr}$---i.e., ignoring the first three positive terms in~\eqref{eq:eta:kr:def}---gives a misclassification rate 
	\begin{align*}
	\max\{\mis \big(\yt, y\big), \mis \big(\zt, \z\big) \} \;\le\;
	\gamma_2 := \sum_{r\ne k}\exp\Big({-}qI_{kr} -\f 1 2 \log(q\Lammin)+v\Big)
	\end{align*} 
	for some $v\to\infty$, w.h.p..
	In the second stage, given the labels $\zt$ and $\yt$, we obtain an estimate of the (row) mean parameters based on $G_2$, using the natural estimator $\Lamh_2 = \Lop(G_2,\yt,\zt)$.
	%
	We then obtain the second stage labels $\y(\zt):=\text{LR}(G_2, \hat\Lambda, \zt)$.  Let $\Lambda_2 = (1-q) \Lambda$ be the row mean parameter of $G_2$. By Lemma~\ref{lem:param:consist}, $\hat\Lambda_2 \in \mathscr B_{\Lambda_2}(\delta)$ w.h.p for some $\delta\ge\gamma_2$. By Lemma~\ref{lem:miss:markov}, and the perturbation of information (Lemma~\ref{lem:info:pert}) we have 
	\begin{align*}
	\mis \big(\yh(\zt), y\big) \le \gamma_3:= \sum_{r\ne k}\exp\lp {-}(1-q)I_{kr}+C(1-q)\delta\infnorm{\Lambda}-\f 1 2\log\Lammin+u\rp
	\end{align*}
	for some $u\to\infty$ w.h.p.. To obtain  result \eqref{eq:Big:O:rate} in Corollary~\ref{cor:main:res}, we at least hope to have 
	\begin{align*}
	qI_{kr}+C(1-q)\gamma_2\Laminf=o(\log\Lammin).
	\end{align*}
	So we need $qI_{kr}=o(\log\Lammin)$ and $(1-q)\gamma_2\Laminf=o(\log\Lammin)$.
	Assume that we have $qI_{kr}=o(\log\Lammin)$. Then, 
	\begin{align*}
	\gamma_2 = \sum_{r\ne k}\exp\lp {-} q{I_{kr}}-\f 1 2\log(q\Lammin)+v\rp=O(\Lammin^{-1/2-o(1)}/\sqrt q).
	\end{align*}
	However, this is not sufficient to show  $(1-q)\gamma_2\Laminf=o(\log\Lammin)$. Therefore, applying edge splitting and Lemma~\ref{lem:miss:markov} does not lead to the main result of this paper.
\end{rem}

\subsection{Analysis on subblocks}
\label{sec:subblk:analysis}

We now extend the analysis of Section~\ref{sec:fixed:label:analysis} to be applicable to the sub-blocks obtained by random partitioning. Some care needs to be taken since the true (row and column) mean parameters of the sub-blocks are changed by partitioning, due to the change in the distributions of the labels within each sub-block among the $\Kr \times \Kc$ classes. The deviations of the sub-block class proportions from the global version will be controlled by \emph{a slack parameter $\xi$} which will be set at the final stage of the proof (see Section~\ref{sec:choosing:params}). Throughout this section, assumptions~\eqref{assump:Lambda} and~\eqref{assump:balance} will be implicit in all the stated lemmas. We will also state the result for a general $2\Q \times \Q$ partitioning scheme, although $\Q = 4$ is enough for the analysis of Algorithm~\ref{alg:provable}.





Recall that the class priors $\pi_\ell(\z)$ for the full labels are defined in~\eqref{eq:pi:def}. We will use the same notation for sublabels $\z^{(q)}$, that is, $\pi_\ell(\z^{(q)})$ is the proportion of labels in $\z^{(q)}$ that lie in class $\ell$. Note that we have 
\begin{align}
\pi_\ell(\z) = \frac{n_\ell(\z)}{\nc}, \quad 
\pi_\ell(\z^{(q)}) = \frac{n_\ell(\z^{(q)})}{\nc/\Q}, \quad \text{hence}, 
\quad \frac{\pi_\ell(\z^{(q)})}{\pi_\ell(\z)} = \Q \frac{n_\ell(\z^{(q)})}{n_\ell(\z)}, 
\end{align}
since $\z^{(q)}$ has length $\nc/Q$. We similarly we have $\pi_k(\y^{(q)}) = n_{k}(\y^{(q)})/(\nr/(2Q))$. We will work under the assumption that the partitioning scheme satisfies:
\begin{align}
\max_{k,q}|\pi_{k}(\y^{(q)})-\pi_k(y)|\le\xi &\quad \text{ and } \quad
\max_{\ell,q}|\pi_{\ell}(\z^{(q)})-\pi_\ell(\z)|\le\xi, \tag{B4a}\label{assum:subblk:counts:a}\\
&\xi \le \min \Big( \frac{1}{2\beta \Kr}, \frac{1}{2\beta \Kc} \Big). \tag{B4b}\label{assum:subblk:counts:b}
\end{align}
When these conditions hold, we call the scheme a \emph{good partition}.
We note that these conditions combined with~\eqref{assump:balance} give,
\begin{align}\label{eq:subblock:balance}
\Big| \f{\pi_{\ell}(\z^{(q)})}{\pi_\ell(\z)} -1 \Big| 
\, \le \, \xi L \beta \le \frac12 
\implies	 \frac12 \frac{1}{\beta L} 
\le \pi_{\ell}(\z^{(q)}) 
\le \frac32 \frac{\beta}{ L} 
\end{align}
and similarly for $\y^{(q)}$. It follows that both $\z^{(q)}$ and $\y^{(q)}$ satisfy~\eqref{assump:Lambda} with $\beta$ replaced with $2\beta$.

Each count $n_k(\y^{(q)})$ follows a hypergeometric distribution with parameters $(\nr, n_k(\y), \nr/(2\Q))$, that is, the number of nodes labeled $k$, in a sample of size $\nr/(2Q)$, from a population of size $\nr$, with a total of $n_k(\y)$ nodes labeled $k$. The concentration of the hypergoemtric distribution gives the following: 
\begin{lem}\label{lem:rand:part}
	\eqref{assum:subblk:counts:a} holds for random partitioning, with probability at least $1-\ptail_2$, where
	\begin{align}\label{eq:ptail:rand:part}
	\ptail_2 = 2\Q(\Kr+\Kc) \exp\big({-}\min(\nr,\nc) \xi^2 /\Q\big).
	\end{align}
\end{lem}

The proof of this lemma and others in this section appear in Appendix~\ref{sec:proofs:subblk:analysis}.
%


\begin{lem}\label{lem:subblk:tru:lam:dev}
	Under~\eqref{assum:subblk:counts:a} and~\eqref{assum:subblk:counts:b}, the true local mean parameters $\Lambda^{(q)} = (\lambda^{(q)}_{k\ell})$ satisfy: 
	\begin{align}\label{eq:subblk:tru:lam:dev}
	\Big|\lambda_{k\ell}^{(q)} - \f{\lambda_{k\ell}}\Q\Big| 
	\le (\xi L \beta) \f{\lambda_{k\ell}}\Q
	\le  \frac12 \f{\lambda_{k\ell}}\Q, 
	\quad \forall q,k,\ell.
	\end{align}
	In particular, $\Lammin^{(q)}\ge \f1{2\Q} \Lammin$, $\infnorm{\Lambda^{(q)}}\le \f3{2\Q} \Laminf$ and $\Lambda^{(q)} \in \Blam{\Lambda/\Q}{\xi L \beta}$ for all $q \in [Q]$.
\end{lem}

\medskip
Our main lemma for the sub-blocks establishes the consistency of the local mean parameter estimates $\Lamh^{(q',q)}$ for a \emph{good} partitioning scheme. This lemma is an extension of Lemma~\ref{lem:param:consist}. We recall the operator $\Lop$ from~\eqref{eq:Lop:def}: 
\begin{lem}[Local parameter consistency]\label{lem:loc:param:consist}
	Let $C_\gamma = \beta^2 \gamma/(1-\gamma)$ and $h_c(\tau):=\f 3{4c}\tau\log\left(1+\f{2c}3 \tau\right)$ as in Lemma~\ref{lem:param:consist} and assume that $72\, C_\gamma \omega \le 1$. 
	Fix the underlying partition and fix $q,q' \in [Q]$, and labels $\zt$ and $\yt$. Let 
	$$
	\Lamh^{(q',q)} = \Lop(A^{(q',q)}, \yt^{(q')}, \zt^{(q)}).
	$$
	Assume that the partition satisfies~\eqref{assum:subblk:counts:a} and~\eqref{assum:subblk:counts:b}, and the pairs~$(\zt^{(q)},\z^{(q)})$ and~$(\yt^{(q)},\y^{(q)})$ satisfy the misclassification rate in~\eqref{assump:missclass}. Then,
	\begin{align*}
	\infnorm{\Lamh^{(q',q)} - \Lambda^{(q)}} 
	&\;\le\; \big(24\,C_\gamma + 6\tau \big) \,\infnorm{\Lambda/\Q}, \quad \text{and} \\
	\infnorm{\Lamh^{(q',q)} - \Lambda/Q} 
	&\;\le\; \big(24\,C_\gamma + 6\tau + \xi L\beta\big) \,\infnorm{\Lambda/\Q}
	\end{align*}
	with probability at least $1-2\ptail_3$, where
	\begin{align}\label{eq:ptail:local:param}
	\ptail_3 = \ptail_3(\tau;\,\nr,\Kr,\Lamin,\Q) := \Kr \Kc \exp \Big( {-}\frac{\nr \Lamin\, h_1(\tau)}{32 \Q^2 \beta \Kr}\Big).
	\end{align}
	We also have
	\begin{enumerate}[label=(\alph*)]
		\item $\infnorm{\Lambda^{(q',q)}(\y,\zt) - \Lambda^{(q)}} \le 4 C_\gamma \infnorm{\Lambda^{(q)}}$.
		\item $\infnorm{\Lambda^{(q',q)}(\yt,\zt) - \Lambda^{(q',q)}(\y,\zt)} \le 2 \gamma \infnorm{\Lambda^{(q)}}$.
		\item $\infnorm{\Lamh^{(q',q)} - \Lambda^{(q',q)}(\yt,\zt)} \le 4 \tau \infnorm{\Lambda^{(q)}}$, with probability at least $1-2p_3$.
	\end{enumerate}
\end{lem}

\begin{rem}\label{rem:Lambda:Gamma:subs}
	Similar results to those obtained above hold for the column parameters. Recall that the dual to the row mean parameters $\Lambda$ are the column mean parameters $\Gamma$. The result of Lemma~\ref{lem:subblk:tru:lam:dev} can be translated to the column version by making the following substitutions $\Lambda \to \Gamma$, $\Q \to 2\Q$ and $\Kc \leftrightarrow \Kr$. For Lemma~\ref{lem:loc:param:consist}, in addition we need to make $\nr \to 4\nc$. (The reason for this is that in~\eqref{temp:879}, in the proof, we need to replace $\nr/2\Q$ with $\nc/\Q$, and $\Lamin/2\Q$ with $\Gamin/(4\Q)$, and the combination of the aforementioned substitutions achieves this. We also note for future reference that the corresponding $\omega$ inflation by a factor of $3$ remains true for column parameters.) After these substitutions, we obtain the same constant in~\eqref{eq:ptail:local:param}, that is, $\ptail_3$ has to be replaced with 
	\begin{align}\label{eq:ptail:col:local:param}
	\ptailp_3 := \ptail_3(\tau;4\nc,\Kc,\Gamin,2\Q) = \ptail_3(\tau;\nc,\Kc,\Gamin,\Q).
	\end{align}
\end{rem}


\subsection{Perturbation of information}\label{sec:pert:info}
Recall the definition of Chernoff information from~\eqref{eq:Info:def}, and let us write $I_{kr} = I_{kr}(\Lambda)$
to explicitly show its dependence on the mean parameter matrix $\Lambda$. The following lemma, proved in Appendix~\ref{sec:proofs:subblk:analysis}, bounds the perturbations of $ I_{kr}(\Lambda)$ in $\Lambda$:
\begin{lem}\label{lem:info:pert}
	Under \eqref{assump:Lambda}, for any $\Lamt \in \Bcl(\delta)$, we have
	$|I_{kr}(\Lamt) - I_{kr}(\Lambda)| \le 2 \omega  \delta \Kc \infnorm{\Lambda}$.
	
\end{lem}

\subsection{Analysis of the matching step}\label{sec:matching:analysis}
In this section, we fill in the details of the argument sketched in Section~\ref{sec:matching:step}. 
Specifically, we need to give sufficient conditions so that the first and the third claims of Section~\ref{sec:matching:step} hold. 
We will use the following two lemmas. Recall the notation $\sigma^*(\yt \to \y)$ introduced in Section~\ref{sec:notation} to denote the optimal permutation from the set of labels $\yt$ to another set $\y$.

\begin{lem}\label{lem:optim:id:perm}
	Let $\yt,\y \in [\Kr]^n$, and assume that $\dmis(\yt,y) < \frac12 \min_k \pi_k(\y)$. Then, 
	\begin{itemize}
		\item[(a)] $\sigma^*(\yt \to \y) = \id$, the identity permutation, and this optimal permutation is unique, and
		\item[(b)] $\pi_k(\yt) > \frac12 \pi_k(\y)$ for all $k$.
	\end{itemize}
\end{lem}

Note that Lemma~\ref{lem:optim:id:perm} implies that if $\dmis(\sigma(\yt),y) < \frac12 \min_k \pi_k(\y)$ for some permutation $\sigma$, then
$\sigma^*(\yt \to \y) = \sigma$. 
\begin{lem}\label{lem:three:perms}
	Consider three sets of labels $\y,\yt,\yt' \in [\Kr]^\nr$, and assume that 
	\begin{align*}
	\max\{\,  \mis(\yt,\y), \,\mis(\yt',\y)\, \} < \frac14 \min_k \pi_k (\yt).
	\end{align*}
	Let $\sigma = \sigma^*(\yt \to \y)$ and $\sigma' =  \sigma^*(\yt' \to \y)$. Then,
	$\sigma^{-1} \circ \sigma' = \sigma^*( \yt' \to \yt ).$
\end{lem}

The first claim of Section~\ref{sec:matching:step} follows from Lemma~\ref{lem:optim:id:perm}, under the further assumption:
\begin{align}\label{eq:mis:q-1:q:bound}
\mis(\yt^{(q-1,q)},\y^{(q-1,q)}) < \frac1{32\beta \Kr}, \quad q \in [\Q].
\end{align}
Using the permutation notations~\eqref{eq:perm:notations} of Section~\ref{sec:matching:step}, we have:
\begin{cor}\label{cor:subblock:perm:equality}
	Under assumptions~\eqref{assump:balance},~\eqref{assum:subblk:counts:a},~\eqref{assum:subblk:counts:b} and~\eqref{eq:mis:q-1:q:bound},  $\sigma_{q-1,q} = \sigma_{q-1}$ for all $q \in [Q]$.
\end{cor}

The third and final claim of  Section~\ref{sec:matching:step} follows from Lemmas~\ref{lem:optim:id:perm} and~\ref{lem:three:perms}, by applying them to the sub-block labels $\y^{(2)}, \yt^{(2)}, \yt'^{(2)}$:
\begin{cor}\label{cor:missing:perm}
	Under assumptions~\eqref{assump:balance},~\eqref{assum:subblk:counts:a},~\eqref{assum:subblk:counts:b} and~\eqref{eq:mis:q-1:q:bound},  $\sigma_q^{-1} \circ \sigma'_q = \sigma^*(\yt'^{(q)} \to \yt^{(q)})$ for all $q \in [\Q]$.
\end{cor}

The proofs of the results of this section are deferred to Appendix~\ref{sec:proofs:subblk:analysis}.


\section{Proof of Theorem~\ref{thm:without:spec:clust}}\label{sec:analysis:alg:three}

We start with the high-level analysis of Algorithm~\ref{alg:provable} in Section~\ref{sec:main:analysis}. This analysis is parametrized by many parameters such as $\xi$ , $\tau_1$, $\taucol_1$, $\tau_2$, etc. This allows us to give the high-level idea of the mechanics of the proof without making the arguments obscured by the expressions ultimately chosen for these parameters. In Section~\ref{sec:choosing:params}, we make specific choices about these parameters and finish the proof of Theorem~\ref{thm:without:spec:clust}.

\subsection{Parametrized analysis of Algorithm~\ref{alg:provable}}\label{sec:main:analysis}

We now have all the pieces for analyzing Algorithm~\ref{alg:provable}.
Let $\ytfive$ and $\ztfive$ be the labels from step~\ref{step:consist} of of Algorithm~\ref{alg:provable}. As before, in all the lemmas stated, \eqref{assump:Lambda} and \eqref{assump:balance} will be implicitly assumed.
Consider the following event:
\begin{align*}
\Af{\gamma} := \Big\{ \text{$\ytfive^{(q)}$ and $\ztfive^{(q)}$ satisfy~\eqref{assump:missclass} with parameter $\gamma$, for all $q \in [Q]$} \Big\}.
\end{align*}
We implicitly assume that clusters in $\ztfive$ and $\ytfive$ are relabeled according to optimal permutation relative to the truth. In other words, $\ztfive$ and $\ytfive$ in the above event are not the raw output of the algorithm, but the relabeled versions (which we do not have access to in practice, but  are well-defined and can be used in the proof.) When $\gamma$ is sufficiently small, this implies that community $k$ in $\ztfive$ is the same as community $k$ in $\zt$, for all $k \in [\Q]$. 

Let $\Pi$ be the random partition used in Algorithm~\ref{alg:provable}, and let $\Pf$ be the event that $\Pi$ satisfies condition~\eqref{assum:subblk:counts:a}. By Lemma~\ref{lem:rand:part}, we have $\pr(\Pf) \ge 1-\ptail_2$ where $\ptail_2$ is given in~\eqref{eq:ptail:rand:part}. For the most part, we will work on events of the form $\Af{\gamma_1} \cap \Pf$. Let us also establish some terminology.  By the probability  ``on an event $\Pf$'', we mean the probability under the restricted measure $\pr_{\Pf} := \pr(\cdot \cap \Pf)$. For example, if $\Df = \{\text{property X  holds}\}$,  we will say that ``property X fails'' on $\Pf$ with probability at most $q$ if  $\pr (\Df^c \cap \Pf) \le q$. In this case, if $\Pf$ holds with high probability, say  $\ge 1-\ptail_2$, and $q$ is small, then $\Df$ holds with high probability as well: $\pr(\Df) \ge 1 - q - \ptail_2$.

\medskip
Let $\Lamhsix^{(q)} = \Lop(A^{(q-2,q)},\yt^{(q-2)},\zt^{(q)}), \;q\in \Z /\Q \Z$, be the first local parameter estimates obtained in step~\ref{step:local:Lambda:1} of Algorithm~\ref{alg:provable} (it is easier to work with the shifted index), and let
\begin{align}\label{eq:delta:gamma:def}
\delta_1 :=  24\,C_{\gamma_1} + 6\tau_1 + \xi \Kc \beta.
\end{align}
A better name for $\delta_1$, and $\tau_1$ would be $\delrow_1$, and similarly $\taurow_1$ contrasting with $\delcol_1$ and $\taucol_1$ defined later in~\eqref{eq:delta:one:col}. However, for simplicity, we drop the ``row'' qualifier here. Recall that $\xi$ is a parameter controlling the tail probability related to the random partition, while $\tau_1$ will be controlling the tail probability  $p_3(\tau_1)$ related to the local parameter estimates in Lemma~\ref{lem:loc:param:consist}. These parameters will be optimized at the end of the argument (see Section~\ref{sec:choosing:params}).
%

\begin{lem}[First local parameters]\label{lem:first:loc:estim}
	Assume~\eqref{assum:subblk:counts:b} and $72\, C_{\gamma_1} \omega \le 1$, and let $\delta_1$ be as defined in~\eqref{eq:delta:gamma:def}.
	Then, on event $\Af{\gamma_1} \cap \Pf $,  
	\begin{align*}
	\Lamhsix^{(q)} \; \in \; \Blam{\Lambda^{(q)}}{\delsix}, \quad \forall q\in \Z_\Q,
	\end{align*}
	fails with probability at most $2\Q\,\ptail_3$, where $\ptail_3 = p_3(\tau_1)$ as given in~\eqref{eq:ptail:local:param}.
\end{lem}
\begin{proof}
	Conditioning on blocks $G_1$ (cf. Section~\ref{sec:provable:alg}) of the (bottom) adjacency matrix $A_\bot$---denoted as $A_\bot^{(G_1)}$---the distribution of blocks $A^{(q-2,q)},  q \in \Z_\Q$ used in defining $\Lamhsix^{(q)}$ is not changed. Under this conditioning, both initial labels $\ytfive$ and $\ztfive$ are deterministic, hence the results of Section~\ref{sec:subblk:analysis} apply. We will apply Lemma~\ref{lem:loc:param:consist} to $\Lamhsix^{(q)}$. Let us verify the conditions of the lemma. On $\Af{\gamma_1}$, for all $q \in [\Q]$, the sublabel pairs  $(\ztfive^{(q)},\z^{(q)})$ and~$(\ytfive^{(q)},\y^{(q)})$ satisfy~\eqref{assump:missclass}. On $\Pf$, condition~\eqref{assum:subblk:counts:a} holds for the random partition and~\eqref{assum:subblk:counts:b}  holds by assumption.
	Recall that the random partition is independent of all else, hence conditioning on it does not change the distribution of blocks $A^{(q-2,q)},  q \in \Z_\Q$
	either. We may then apply Lemma~\ref{lem:loc:param:consist} to conclude that for every $q \in \Z_\Q$, conditioned on the partition $\Pi$ and $A^{(G_1)}_\bottom$, the event $\{\Lamhsix^{(q)} \notin  \Blam{\Lambda^{(q)}}{\delsix}\} \cap \Af{\gamma_1} \cap \Pf$ holds with probability $ \le 2\ptail_3$. Let us write 
	$$\Df = \big\{\Lamhsix^{(q)} \in  \Blam{\Lambda^{(q)}}{\delsix},\; \forall q \in \Z_\Q \big\} $$
	which is the desired event in this lemma. Using the union bound, and removing the conditioning, we have $\pr(\Df^c \cap \Af{\gamma_1} \cap \Pf) \le 2\Q\ptail_3$, unconditionally. The proof is complete.
\end{proof}

Next we consider the first LR classifier application. Let $\ytseven$ be the row label estimates in step~\ref{step:row:LR:1}. That is, we have
\begin{align*}
\ytseven^{(q-2)} = \LR\big( A^{(q-2,q)}, \;\Lamhsix^{(q)},\; \ztfive^{(q)} \big)
\end{align*}
for which we have the following bound on  misclassification rate:
\begin{lem}[First LR classifier]\label{lem:first:LR}
	Under the assumptions of Lemma~\ref{lem:first:loc:estim}, further assume that 
	$ 9  \omega \delsix  < 1$. Let $\etaseven := 2\etap(\delta_1;\Lambda/\Q)$ 	where $\etap(\cdot)$ is defined in~\eqref{eq:eta:kr:def}.
	Then, on event $\Af{\gamma_1} \cap \Pf$, 
	\begin{align}\label{eq:gam2:def}
	\mis_k \big(\ytseven^{(q)} ,\,\y^{(q)}\big) 
	\;\le\;  \sum_{r \neq k} \exp\Big( {-}\frac{I_{kr}}\Q +\etaseven  + u\Big) =: \gamrow_{2k}, \quad \forall q \in \Z_\Q, 
	\end{align}
	fails with probability at most $ \Q(e^{-u} +2\Q\,\ptail_3)$ where  $p_3 = p_3(\tau_1)$ as given in~\eqref{eq:ptail:local:param}.
\end{lem}
\begin{proof}
	Fix $q \in \Z_\Q$ and consider $\yt^{(q-2)}$.
	As in the proof Lemma~\ref{lem:first:loc:estim}, we condition on blocks in $G_1$ so that $ \ztfive^{(q)}$ can be assumed deterministic. We will apply Lemma~\ref{lem:unif:param:a} to the subblock $A^{(q-2,q)}$. As discussed earlier, the corresponding $\omega$ is inflated to $3 \omega$, hence we need $3(3\omega) \delsix< 1$ which we have assumed. We also note that  $\Lambda^{(q-2,q)}(\y,\zt)$ and $\Lambda^{(q)}$ play the role of $\Lambda(y,\zt)$ and $\Lambda$ in Lemma~\ref{lem:unif:param:b}, and we have the needed condition $\Lambda^{(q-2,q)}(\y,\zt) \in \Blam{\Lambda^{(q)}}{\delsix}$ from Lemma~\ref{lem:loc:param:consist}. Let $b_{i\ast}^{(q-2,q)}$ be the row block compression of $A^{(q-2,q)}$ based on $\ztfive^{(q)}$. Then, Lemma~\ref{lem:unif:param:a} gives
	\begin{align}\label{eq:temp:6787}
	\P \Big( \Big\{ \exists\, \Lamt \in  \Blam{\Lambda^{(q)}}{\delsix} , \;
	Y_{ikr}\big(b_{i\ast}^{(q-2,q)}, \Lamt\big) \ge 0 \Big\} \cap \Af{\gamma_1} \cap \Pf  \,\Big|\, A^{(G_1)}_\bottom, \Pi \Big)
	\; \le \; \exp\big({-}I_{kr}^{(q)} + \eta^{(q)} \big)
	\end{align}
	for all rows $i$ (in row block $q-2$) with $\y_i = k$.
	Here $\Lamin^{(q)}$ is the minimum element of $\Lambda^{(q)}$, and 
	\begin{alignat*}{3}
	& I_{kr}^{(q)} 
	&&\;:=\; I_{kr}(\Lambda^{(q)}) 
	&&\;\ge\;  \frac{I_{kr}}\Q - 2\omega \delsix \Kc \infnorm{\Lambda/\Q} \\
	& \eta^{(q)}  
	&&\;:=\; \etap(\delsix; \; \Lambda^{(q)})   
	&&\;\le\; (1+\delta_1)\, \etap(\delsix; \; \Lambda/\Q) 
	\end{alignat*}
	where $\etap(\delta_1; \Lambda^{(q)}) = 8 \omega \delta_1 \Kc \infnorm{\Lambda^{(q)}}$ as defined in~\eqref{eq:eta:kr:def}. The first inequality uses Lemma~\ref{lem:info:pert} and the second is obtained using the definition of $\etap(\cdot)$ combined with $\Lambda^{(q)} \in \Blam{\Lambda/\Q}{\delsix}$ (Lemma~\ref{lem:subblk:tru:lam:dev}) which implies $\infnorm{\Lambda^{(q)}} \le(1+\delta_1) \infnorm{\Lambda/\Q}$. By taking expectation in~\eqref{eq:temp:6787}, the same bound holds unconditionally. 
	
	%
	
	By Lemma~\ref{lem:first:loc:estim}, on event $\Af{\gamma_1} \cap \Pf$, we have $\Lamhsix^{(q)} \notin \Blam{\Lambda^{(q)}}{\delsix}$ with probability at most $2\Q\,\ptail_3$. Then, applying Lemma~\ref{lem:miss:markov}, we conclude that
	\begin{align*}
	\mis_k \big(\ytseven^{(q-2)} ,\,\y^{(q-2)}\big) 
	\;\le\;	 \sum_{r \neq k} \exp\big( {-}I_{kr}^{(q)} +\eta^{(q)} + u\big) 
	\end{align*}
	fails on $\Af{\gamma_1} \cap \Pf$ with probability $\le e^{-u} +2\Q\,\ptail_3$, for each $q \in [Q]$.  Note that
	\begin{align*}
	-I_{kr}^{(q)} + \eta^{(q)} \;\le\; -\frac{I_{kr}}\Q  + (1+\delta_1 + 9^{-1}) \,\eta'(\delsix; \; \Lambda/\Q).
	\end{align*}
	Since $9 \omega \delta_1 < 1$ implies $\delta_1 < 9^{-1}$ (recall $\omega \ge 1$), we have $1+\delta_1 + 9^{-1} < 2$.
	Combining with the previous bound and applying the union bound over $q$ gives the result.
\end{proof}

Note that we have called the rate in~\eqref{eq:gam2:def} $\gamrow_2$ for the (column) misclassification rate based on the row information. This rate is faster than initial rate $\gamma_1$. Repeating the procedure in steps~\ref{step:local:Lambda:1} and~\ref{step:row:LR:1} for the column labels---as prescribed in step~\ref{step:similar:z:1} in Algorithm~\ref{alg:provable}--we obtain a similar rate for the misclassification rate of $\zteight^{(q)}$ relative to $\z^{(q)}$ which we call $\gamcol_2$. In deriving $\gamcol_2$, we have to make the substitutions in Remark~\ref{rem:Lambda:Gamma:subs}, and particular, $\Lambda \to \Gamma$ where $\Gamma$ is the column mean parameters defined in Section~\ref{sec:bi:sbm}. (A minor exception is when counting the number of blocks which will still be $\Q$ rather than $2\Q$.)  Recall the definition of the column information matrix $(\Icol_{\ell r})$ from~\eqref{eq:Icol:def}.
Letting
\begin{align}\label{eq:delta:one:col}
\delsixcol :=  24\,C_{\gamma_1} + 6\taucol_1 + \xi \Kr \beta,
\end{align}
we obtain the following counterpart of Lemma~\ref{lem:first:LR}:
\begin{cor}[First LR classifier, column version]\label{cor:first:LR:col}
	Under the assumptions of Lemma~\ref{lem:first:loc:estim}, further assume that 
	$ 9  \omega \delsixcol  < 1$. Let $\etaeight := 2\etap(\delsixcol;\Gamma/(2\Q))$ 	where $\etap(\cdot)$ is defined in~\eqref{eq:eta:kr:def}.
	Then, on event $\Af{\gamma_1} \cap \Pf$, 
	\begin{align}\label{eq:gamcol2:def}
	\mis_\ell \big(\zteight^{(q)} ,\,\z^{(q)}\big) 
	\;\le\;  \sum_{r \neq \ell} \exp\Big( {-}\frac{\Icol_{\ell r}}{2\Q} +\etaeight  + \ucol\Big) =: \gamcol_{2\ell}, \quad \forall q \in \Z_\Q, 
	\end{align}
	fails with probability at most $ \Q(e^{-\ucol} +2\Q\,\ptailp_3)$, where $p'_3 = p'_3(\taucol_1)$ as given in~\eqref{eq:ptail:col:local:param}.
\end{cor}


Let $\gamcol_2 := \max_{k\in[\Kr]} \gamcol_{2k}$, $\gamrow_2 := \max_{\ell\in[\Kc]} \gamrow_{2\ell}$ and
\begin{align}\label{eq:gamma:two:def}
\gamma_2 := \max\{ \beta\Kr\gamrow_2, \beta\Kc\gamcol_2 \}.
\end{align}
By~\eqref{eq:Mis:Mis:k}, we have that~\eqref{eq:gam2:def} and~\eqref{eq:gamcol2:def} imply
\begin{align}
\mis \big(\ytseven^{(q)} ,\,\y^{(q)}\big) \le \gamcol_2\le \f{\gamma_2}{\beta\Kr}, \quad  \mis \big(\zteight^{(q)} ,\,\z^{(q)}\big) \le \gamrow_2\le \f{\gamma_2}{\beta\Kc}
\end{align}
Thus, if we consider the following event:
\begin{align*}
\Bf{\gamma} := \Big\{ \text{$\ytseven^{(q)}$ and $\zteight^{(q)}$ satisfy~\eqref{assump:missclass} with parameter $\gamma$, for all $q \in [Q]$} \Big\},
\end{align*}
after Step~\ref{step:similar:z:1}, we can work on $\Bf{\gamma_2} \cap\, \Pf$ which holds with high probability: Combining Lemma~\ref{lem:first:LR} and Corollary~\ref{cor:first:LR:col}, by union bound, $\pr (\Bf{\gamma_2}^c \cap \Af{\gamma_1} \cap \Pf) \le \Q(2 e^{-u} + 2\Q (\ptail_3 + \ptail'_3))$, hence
\begin{align}\label{eq:Bf:Pf:tail}
\begin{split}
\P(\Bf{\gamma_2} \cap \Pf) 
&\;\ge\; \P \big(\Bf{\gamma_2} \cap \Af{\gamma_1}  \cap \Pf \big)  \\
&\;=\; \P\big(\Af{\gamma_1}  \cap \Pf\big) -  \P\big(\Bf{\gamma_2}^c \cap \Af{\gamma_1}  \cap \Pf \big)   \\
&\;\ge\; 1 - \P(\Af{\gamma_1}^c) - \P(\Pf^c) - \Q( 2e^{-u} + \Q  (\ptail_3 + \ptail'_3)).
\end{split}
\end{align}

\medskip
Let $\Lamhnine^{(q)} =  \Lop(A^{(q-3,q)}, \yt^{(q-3)}, \zt^{(q)}), \;q\in \Z_\Q$, be the second local parameter estimates obtained in step~\ref{step:local:Lambda:2} of Algorithm~\ref{alg:provable}.  Let 
\begin{align}\label{eq:delnine:def}
\delnine :=  
24 C_{\gamma_2} + 6\tau_2.
\end{align}
\begin{lem}[Second local parameters]\label{lem:second:loc:estim}
	Assume~\eqref{assum:subblk:counts:b} and $72\, C_{\gamma_2} \omega \le 1$, and let $\delta_2$ be as defined in~\eqref{eq:delnine:def}. Then, on event $\Bf{\gamma_2} \cap \Pf$, 
	\begin{align*}
	\Lamhnine^{(q)} \; \in \; \Blam{\Lambda^{(q)}}{\delnine} , 
	\quad 
	\forall q\in \Z_\Q 
	\end{align*}
	fails with probability at most $2\Q\,\ptail_3$, where  $\ptail_3$ is  given in~\eqref{eq:ptail:local:param}.
\end{lem}
\begin{proof}
	Conditioning on blocks $G_1 \cup G_2$ (cf. Section~\ref{sec:provable:alg}) of the adjacency matrix $A$, the distribution of blocks $A^{(q-3,q)}$ used in defining $\Lamhnine^{(q)}$ is not changed. Under this conditioning, both initial labels $\ytseven$ and $\zteight$ are deterministic, hence the results of Section~\ref{sec:subblk:analysis} apply. On $\Bf{\gamma_2}$, for all $q \in \Z_\Q$, the sublabel pairs  $(\zteight^{(q)},\z^{(q)})$ and~$(\ytseven^{(q)},\y^{(q)})$  satisfy~\eqref{assump:missclass}. The rest of the proof follows that of Lemma~\ref{lem:first:loc:estim}.
	%
	%
\end{proof}

The key is that $\delnine$ is much smaller than $\delsix$,  due to $\gamma_2 \ll \gamma_1$ (typically), i.e., the second parameter estimates are much more accurate. Let $\Lamhten = \sum_{q} 	\Lamhnine^{(q)}$ be the estimate of the global mean parameters obtained in step~\ref{step:global:Lambda} of Algorithm~\ref{alg:provable}. According to Lemma~\ref{lem:second:loc:estim}, on~$\Bf{\gamma_2} \cap \Pf$,
\begin{align}\label{eq:Lamhten:bound}
\infnorm{	\Lamhnine^{(q)} - \Lambda/\Q} \le \delnine \infnorm{\Lambda/Q},\; \forall q 
\quad \text{hence}, \quad 
\infnorm{	\Lamhten - \Lambda} \le \delnine \infnorm{\Lambda}
\end{align}
fails with probability $\le 2\Q\,\ptail_3$, where we have used triangle inequality. That is, on $\Bf{\gamma_2} \cap \Pf$, we have $\Lamhten \in \Blam{\Lambda}{\delnine}$ with high probability.

\begin{rem}
	Note that  we could have used $\Q	\Lamhnine^{(q)} $ (for any $q \in \Z_\Q$)  as our estimate $\Lamhten$, leading to the same bound as in~\eqref{eq:Lamhten:bound}. The results would be the same, though in practice, we expect the version given in the Algorithm~\ref{alg:provable} to perform better. We also note that on $\Bf{\gamma_2}$, the sublabels $(\zteight^{(q)},\; q \in \Z_\Q)$ automatically define a consistent global label vector $\zteight$, and similarly for row labels $\ytseven$. 
\end{rem}

\begin{lem}[Second LR classifier]\label{lem:second:LR}
	Under the assumptions of Lemma~\ref{lem:second:loc:estim},  further assume that $\delnine$ defined in~\eqref{eq:delnine:def} satisfies $ 3 \omega \delnine  < 1$ and
	$6 C_{\gamma_2} \omega \le 1$.
	Let $\etaeleven{kr} :=    \eta_{kr} \big( \delnine;\; \omega, \beta, \nc, \Lambda\big)$.  Then, on event $\Bf{\gamma_2} \cap \Pf$, 
	\begin{align}\label{eq:gam3:def}
	\mis_k \big(\yh_\top ,\,\y_\top\big) 
	\;\le\;  \sum_{r \neq k} \exp\Big( {-} I_{kr}  +\etaeleven{rk}  + v\Big) =: \gamma_3,
	\end{align}
	fails with probability at most $e^{-v} +2\Q\,\ptail_3$ where  $p_3 = p_3(\tau_2)$ as given in~\eqref{eq:ptail:local:param}.  The same result holds for $\etaeleven{kr} =   \eta' ( \delnine; \Lambda)$. 
\end{lem}

\begin{proof}
	As in the proof Lemma~\ref{lem:second:loc:estim}, we condition on blocks in $G_1 \cup G_2$ so that $ \zteight$ can be assumed deterministic. We will apply Lemma~\ref{lem:unif:param:b} to $A_\top$. Let $\Top \subset [\nr]$ denote the row indices of $A_\top$. Since all the columns are present in $A_\top$, we can directly apply Lemma~\ref{lem:unif:param:b} (in contrast to the argument in Lemma~\ref{lem:first:LR}), that is, the relevant row mean parameters are $\Lambda(\y,\zt)$ and $\Lambda$---the same as those for the whole matrix $A$.  The needed condition $\Lambda(\y,\zt) \in \Blam{\Lambda}{\delnine}$ is supplied by  Lemma~\ref{lem:param:consist}. Let $\beleven_{i\ast} = b_{i\ast}(\zteight)$ be the block compression in step~\ref{step:row:LR:2} of the algorithm. Then, Lemma~\ref{lem:unif:param:b} gives (after conditioning on $A_\bottom^{(G_1 \cup G_2)}$ and then removing the conditioning as in~\eqref{eq:temp:6787})
	\begin{align}
	\P \Big( \Big\{ \exists\, \Lamt \in  \Blam{\Lambda}{\delsix} , \;
	Y_{ikr}\big(\beleven_{i\ast}, \Lamt\big)  \ge 0   \Big\} \cap \Bf{\gamma_2}\Big)
	\; \le \; \exp\Big({-I_{kr} + \etaeleven{kr}}  \Big). 
	\end{align}
	for any $ i \in \Top$ with $\y_i = k$. 	By~\eqref{eq:Lamhten:bound}, on~$\Bf{\gamma_2} \cap \Pf$, we have $\Lamhten \notin \Blam{\Lambda}{\delnine}$ with probability at most $2\Q\,\ptail_3$. Then, applying Lemma~\ref{lem:miss:markov}, we conclude~\eqref{eq:gam3:def} as desired. The last statement of the theorem follows if we apply Lemma~\ref{lem:unif:param:a} in place of Lemma~\ref{lem:unif:param:b} throughout.
\end{proof}
The same exact bound holds for $\yh_\bot$ in step~\ref{step:swap:top:bot}, with the same probability. Hence, by union bound, the same bound on misclassification rate holds for the final row labels $\yh$ in step~\ref{step:final:row:lab}, with probability inflated by a factor of $2$; that is, $\mis_k(\yh,\y) \le \gamma_3$ fails   on~$\Bf{\gamma_2} \cap \Pf$, with probability at most $2(e^{-v} +2\Q\,\ptail_3)$. 

To summarize, under the conditions of the lemmas, we have
\begin{align}\label{eq:gam3:prob:bound}
\begin{split}
\pr \big( \mis_k(\yh,\y) > \gamma_3 \big) 
&\le   \pr \big( \big\{ \mis_k(\yh,\y) > \gamma_3 \big\} \cap \Bf{\gamma_2} \cap \Pf \big) + \pr\big((\Bf{\gamma_2} \cap \Pf)^c\big)\\
&\le  2\big(e^{-v} +2\Q\,\ptail_3(\tau_2) \big) +   \pr\big((\Bf{\gamma_2} \cap \Pf)^c\big)\\
&\le  2\big(e^{-v} +2\Q\,\ptail_3(\tau_2) \big) + \P(\Af{\gamma_1}^c) + \P(\Pf^c) + \Q\big( 2e^{-u \wedge \ucol} + \Q (\ptail_3(\tau_1) + \ptail'_3(\taucol_1)\big) 
\end{split}
\end{align}
where $\gamma_3$ is the rate given in~\eqref{eq:gam3:def} and the second inequality uses~\eqref{eq:Bf:Pf:tail}. 

\medskip

\subsection{Choosing the parameters}\label{sec:choosing:params}
It remains to choose the parameters, $\tau_1$, $\tau_2$, $\xi$, etc. to simultaneously achieve the desired rate for $\gamma_3$ and ensure that the probability in~\eqref{eq:gam3:prob:bound} is $o(1)$.

\begin{proof}[Proof of Theorem~\ref{thm:without:spec:clust}]
	\textbf{First row LR classifier.} 
	Let us write $\taurow_1 =\tau_1$ for clarity.
	Under our assumptions, we will have $\gamma_2 \le \gamma_1 \le 1/2$ so that $C_{\gamma_i} \le 2\beta^2 \gamma_i$ for $i=1,2$, recalling the definition of $C_\gamma =  \beta^2 \gamma/(1-\gamma)$. 
	In Lemma \ref{lem:first:LR}, we defined (recall~\eqref{eq:delta:gamma:def})
	\begin{align}\label{eq:def:etaseven}
	\etaseven 
	= 8\delta_1\omega\Kc\infnorm{\Lambda/Q} 
	\;\le\; \big(384\beta^2\gamma_1+48\tau_1^{\text{row}}+8\beta\Kc\xi\big)\, \omega\Kc\infnorm{\Lambda/Q}.
	\end{align}
	By~\eqref{assump:small:gamma:one}, 
	$384 \beta^2\gamma_1\omega\Kc\infnorm{\Lambda/Q} \le \Imin/(8Q)$. Take 
	\begin{align}\label{eq:taurow:xi:def}
	\taurow_1 = \frac{\Imin}{384\omega\Kc\Laminf}, 
	\quad \xi = \frac {\Imin\wedge\Imin^{\text{col}}}{64\beta\omega(\Kr \vee \Kc)^2(\Laminf \vee \Gaminf)},
	\quad u = \frac{\Imin}{8Q},
	\end{align}
	where $u$ is the parameter in~\eqref{eq:gam2:def}. Then from~\eqref{eq:def:etaseven} we have
	\begin{align*}
	\etaseven 
	\le \f\Imin{8Q} + \frac{\Imin}{8Q} + \frac{\Imin}{8Q} = \f{3\Imin}{8Q}, \quad \etaseven+u\le \f{\Imin}{2Q}.
	\end{align*}
	%
	Hence Lemma \ref{lem:first:LR} implies that on event $ \Pf$, 
	\begin{align}\label{eq:gam2:final}
	\mis_k \big(\ytseven^{(q)} ,\,\y^{(q)}\big) 
	\;\le\; \gamrow_{2k} 
	:= \sum_{r \neq k} \exp\Big( {-}\frac{I_{kr}}\Q + \f{\Imin}{2Q}\Big) 
	\le \Kr\exp\lp {-}\f{\Imin}{2Q} \rp, \quad \forall q \in \Z_\Q
	\end{align}
	fails with probability at most $ \Q(e^{-u} +2\Q\,\ptail_3(\tau_1^{\text{row}}))$. By~\eqref{assump:bounded:J}, $\Q \log \Q = o(\Imin)$, hence $Q e^{-u} = o(1)$. By~\eqref{assump:sparse:network}, 
	\begin{align}\label{first:row:LR:tau}
	\begin{split}
	Q^2\ptail_3(\tau_1^{\text{row}})
	&=Q^2\Kr \Kc \exp \Big( {-}\frac{\nr \Lamin\, h_1(\tau_1^{\text{row}})}{32 \Q^2 \beta \Kr}\Big) \\
	&\le Q^2\Kr \Kc \exp \lp {-}\f {\nr\Imin^2  }{256(384^2) \Q^2 \beta \Kr \Kc^2\omega^3\Laminf} \rp = o(1)
	\end{split}
	\end{align}
	where we have used the definition~\eqref{eq:ptail:local:param} of $p_3$,  $h_1(\tau)\ge \tau^2/8$ for $\tau\le 1$ and $\Laminf/\Lamin \le \omega$. 
	Moreover, \eqref{assump:bounded:J} implies $(\Imin\wedge\Imin^{\text{col}}) / (\Kr \vee \Kc) \to \infty$, hence eventually $(\Imin\wedge\Imin^{\text{col}}) / (\Kr \vee \Kc)  \ge 1$ which gives
	\begin{align}\label{eq:P:Pf:final:bound}
	\begin{split}
	\P(\Pf^c)
	=\ptail_2(\xi) 
	&= 2\Q(\Kr+\Kc) \exp\big({-}(\nr\wedge\nc) \xi^2 /\Q\big) \\
	&\le 2\Q(\Kr+\Kc)\exp\lp {-}\f{\nr\wedge\nc}{64^2 Q\beta^2\omega^2(\Kr \vee \Kc)^2(\Laminf \vee \Gaminf)^2} \rp = o(1)
	\end{split}
	\end{align}
	where the last implication follows from~\eqref{assump:sparse:network}.
	
	\paragraph{First column LR classifier.}
	We can apply a similar argument to $\zteight$. Let 
	\begin{align*}
	\taucol_1 = \f{\Imin^{\text{col}}}{768\omega\Kr\Gaminf}, \quad \ucol = \frac{\Imincol}{16Q},
	\end{align*}
	with $\xi$ defined as in~\eqref{eq:taurow:xi:def}. By~\eqref{assump:small:gamma:one},  and a similar argument, we obtain  $\etaeight + \ucol \le \Imincol/(4\Q)$.
	%
	%
	By Corollary \ref{cor:first:LR:col},  on event $ \Pf$, 
	\begin{align}\label{eq:gamcol2:final}
	\mis_\ell \big(\zteight^{(q)} ,\,\z^{(q)}\big) 
	\;\le\;  \gamcol_{2\ell} \;\le\; \sum_{r \neq \ell} \exp\Big( {-}\frac{\Icol_{\ell r}}{2\Q} +\f{\Imincol}{4\Q} \Big) 
	\le \Kc \exp \lp {-}\f{\Imincol}{4\Q} \rp , \quad \forall q \in \Z_\Q, 
	\end{align}
	fails with probability at most $ \Q(e^{-\ucol} +2\Q\,\ptailp_3)$, where $\Q^2 \ptailp_3=\Q^2 \ptailp_3(\tau_1^{\text{col}}) = o(1)$ by \eqref{assump:sparse:network} and $\Q e^{-\ucol} = o(1)$ by~\eqref{assump:bounded:J}, similar to how we argued for the row labels.
	
	\paragraph{Second row LR classifier.}
	Recalling $\gamma_2$ from~\eqref{eq:gamma:two:def} and combining with~\eqref{eq:gam2:final} and~\eqref{eq:gamcol2:final}, 
	\begin{align}\label{eq:gam2:obound}
	\gamma_2 \,\le\, \max\lp \beta\Kr^2 \exp\Big( {-}\frac{\Imin}{2\Q}\rp,\, \beta\Kc^2 \exp\Big( {-}\frac{\Imincol}{4\Q} \rp\rp = o\lp \f{\beta(\Kr\vee\Kc)^2}{(\Imin\wedge \Imincol)^b} \rp
	\end{align}
	for any $b > 0$, as $\Imin\wedge \Imincol\to\infty$.
	By Lemma~\ref{lem:second:LR} and~\eqref{eq:eta:kr:def},
	\begin{align}
	\etaeleven{kr} &:=    \eta_{kr} \big( \delnine;\; \omega, \beta, \nc, \Lambda\big) \notag \\
	&=21\delta_2\,\omega\Kc\Laminf + \frac{5\beta \Kc^2 \infnorm{\Lambda}^2}\nc + \log \lp 11\omega\lp \f{1}{\epsi_{kr}-2\omega(1+\epsi_{kr})\,\delta_2 } + 1 \rp\rp -\f 12 \log \Lammin \notag \\
	&=: T_1 + T_2 + T_3 + T_4, \label{eq:eta:three:pieces}
	\end{align}
	where we have called the four summands above $T_1,\dots,T_4$ in the order they appear.
	We have $\delta_2=24 C_{\gamma_2} + 6\tau_2 \le 48\beta^2\gamma_2 + 6\tau_2$ by \eqref{eq:delnine:def} and the assumption $\gamma_2\le \f 12$. Then, 
	\begin{align*}
	T_1\;\le\; 21(48) \beta^2 \omega  \Kc \Laminf\,\gamma_2 \,+ \,21(6)  \omega \Kc \Laminf \,\tau_2 \;=:\; T_{11}+T_{12}.
	\end{align*}
	%
	For any $b>0$, by~\eqref{eq:gam2:obound}
	\begin{align}\label{eq:T:one:one}
	T_{11} = O\big(  \beta^2 \omega  \Kc \Laminf \gamma_2 \big)
	= o\lp \f{\beta^3\omega(\Kr\vee\Kc)^3\Laminf}{[(\Imin\wedge\Imincol)/Q]^b} \rp.
	\end{align}
	Recall that we have $\beta\omega(\Kr\vee\Kc)\Laminf=o([(\Imin\wedge\Imincol)/Q]^a)$ for some $a>0$ by~\eqref{assump:bounded:J}. Taking $b=3a$ in \eqref{eq:T:one:one}, we obtain $T_{11}=o(1)$.
	Letting 
	$\tau_2=(\omega\Kc\Laminf)^{-1}$, 
	we have
	$T_{12} = O(1)$, hence, $T_1=O(1)$. Recalling the probability bound in Lemma~\ref{lem:second:LR}, we have by \eqref{assump:sparse:network}
	\begin{align}\label{second:row:LR:tau}
	\begin{split}
	Q\ptail_3(\tau_2)
	&=Q\Kr \Kc \exp \Big( {-}\frac{\nr \Lamin\, h_1(\tau_2)}{32 \Q^2 \beta \Kr}\Big) \\
	&\le Q\Kr \Kc \exp \lp {-}\f {\nr }{256 \Q^2 \beta \Kr \Kc^2\omega^3\Laminf} \rp = o(1)
	\end{split}
	\end{align}
	where we have used $h_1(\tau)\ge \tau^2/8$ for $\tau\le 1$ and $\Laminf/\Lamin \le \omega$.
	%
	Using~\eqref{assump:sparse:network} again, $T_2={5\beta \Kc^2 \infnorm{\Lambda}^2}/\nc=O(1)$. 
	
	Now let us consider the third piece $T_3$ in~\eqref{eq:eta:three:pieces}. 
	Recall that $J_{kr} = \Kc \Laminf / I_{kr}$.
	By Lemma~\ref{lem:control:theta:diff} in Section~\ref{sec:Poi:error:exponent}, $\epsi_{kr}\ge 2 \big(  J_{kr}^{-1} \wedge 1 \big)$. In bounding $T_1$, we have shown  $\delta_2\omega\Kc\Laminf=O(1)$, hence $2\omega\delta_2=O((\Kc\Laminf)^{-1})$. Since $I_{kr}\to\infty$ and $J_{kr} \ge  1/2$ (see Remark~\ref{rem:typical:rates}), $(\Kc\Laminf)^{-1}=o\big( J_{kr}^{-1}\wedge 2\big)$. Therefore, $2\omega\delta_2=o(\epsi_{kr}\wedge 1)$. As a result, $2\omega(1+\epsi_{kr})\delta_2=o(\epsi_{kr})$, hence
	\begin{align}\label{eq:exp:T2}
	e^{T_3} := 11\omega\lp 1+ \f{1}{\epsi_{kr}-2\omega(1+\epsi_{kr})\delta_2 } \rp
	=O\lp \omega\lp 1+\f 1{\epsi_{kr}}\rp \rp.
	\end{align}
	Finally, we let $v=\sqrt{\log\Lammin}$. Since $\Lammin\to\infty$, $e^{-v}=o(1)$. Applying Lemma~\ref{lem:second:LR}, combined with $T_1+T_2=O(1)$, and~\eqref{eq:exp:T2}, then for  $\zeta = 1/\sqrt{\log \Lamin} = o(1)$,
	\begin{align*}
	\mis_k \big(\yh_{\text{top}},\,\y_{\text{top}}\big) 
	&= O\lp {\omega}\sum_{r \neq k} \lp 1+ \f{1}{\epsi_{kr}}\rp \exp\Big( {-} I_{kr} -\f 12 \log \Lammin +\sqrt{\log\Lammin} \Big) \rp\\
	&= O\Big({\omega}\sum_{r \neq k}\lp 1+ \f{1}{\epsi_{kr}}\rp \exp\Big( {-} I_{kr} - \lp \f 12 - \zeta\rp \log \Lammin \Big)\Big)
	\end{align*}
	fails w.p.  $ \le 2(e^{-v} +2\Q\,\ptail_3(\tau_2))  + \P(\Pf^c) + \Q( 2e^{-u \wedge \ucol} + \Q (\ptail_3(\tau_1^{\text{row}}) + \ptailp_3(\tau_1^{\text{col}}))) =o(1)$. When we swap ${A_\top}$ and ${A_\bot}$ and repeat the algorithm, the same misclassification rate holds. The proof of Theorem~\ref{thm:without:spec:clust} is complete.
	%
\end{proof}

\section{Proof of Other Main Results}\label{sec:proof:other:main:res}

\subsection{Proof of Theorem~\ref{thm:main:res}}\label{sec:proof:main:res}
We proceed by stating a few lemmas.  The proofs are deferred to Appendix~\ref{sec:proofs:main:analysis}.

\begin{lem}\label{lem:l2dist:I:inequality}
	$\sum_{\ell\in[\Kc]} (\lambda_{r\ell}-\lambda_{k\ell})^2\; \ge \;2 \Lammin I_{kr}$. As a consequence, $\Lamwedge^2 \ge 2\Lammin\Imin$.
\end{lem}

Combining Lemma~\ref{lem:l2dist:I:inequality} with Theorem~\ref{thm:spectral:clustering}, and noting that $\infnorm{\Lambda}/ \Lamwedge^2 \le \omega / (2 \Imin)$ as a consequence of the lemma, we obtain the following guarantee for spectral clustering in terms of the information matrix $(I_{kr})$:

\begin{cor}\label{cor:spectral:clustering:I}
	Consider the spectral algorithm \ given in Algorithm~\ref{alg:scerr}, assume that for a sufficiently small $C_1 > 0$,
	\begin{align}\label{eq::asump:I}
	\f{\beta^2 \omega\Kr\Kc(\Kr\wedge\Kc) \,\alpha}{2\Imin} \le C_1 (1+\kappa)^{-2}.
	\end{align}
	Then the algorithm outputs estimated row labels $\yt$ satisfying w.h.p.
	\begin{align*}
	\mis(\yt, \y) \;\le \f{(1+\kappa)^2\omega\beta \Kc(\Kr\wedge\Kc)\alpha }{2C_1\Imin}.
	\end{align*}
\end{cor}

We next modify Corollary~\ref{cor:spectral:clustering:I} to be applicable on sub-blocks:

\begin{lem}[Spectral clustering on subblocks]\label{lem:spec:clust:subblk}
	Suppose \eqref{assump:sparse:network} holds, and we assume for a sufficiently small $C_1 > 0$,
	\begin{align}\label{assump:spec:clust:subblk}
	\f{6 Q\beta^2 \omega^2\Kr\Kc(\Kr\wedge\Kc) \,\alpha}{\Imin} \le C_1 (1+\kappa)^{-2}.
	\end{align}
	Using Algorithm \ref{alg:scerr} in Step~\ref{step:init:row:lab} of Algorithm \ref{alg:provable}, w.h.p., the misclassification rate of $\yt^{(q)}$ satisfies
	\begin{align*}
	\mis(\yt^{(q)},y^{(q)})\le \f{3Q(1+\kappa)^2\omega^2\beta \Kc(\Kr\wedge\Kc)\alpha }{C_1\Imin} \quad \forall q \in [Q].
	\end{align*}
\end{lem}

A similar result holds for misclassification rate of the spectral clustering for column labels, with appropriate modifications. 

\begin{proof}[ of Theorem~\ref{thm:main:res}]
	Assumption~\eqref{assump:bounded:J:strong} implies \eqref{assump:spec:clust:subblk}, eventually as $\Imin \to \infty$. Letting $\gamrow_1$ and $\gamcol_1$ be bounds on the misclassification rates of the spectral clustering algorithms in steps~\ref{step:init:row:lab} and~\ref{step:init:col:lab}, we can take, by Lemma~\ref{lem:spec:clust:subblk} (and its column counterpart), w.h.p.
	\begin{align*}
	\gamrow_1 = O \Big( \f{\Q\omega\beta \Kc(\Kr\wedge\Kc)\alpha }{ \Imin} \Big), \quad 
	\gamcol_1 = O \Big( \f{\Q\omega\beta \Kr(\Kr\wedge\Kc)\alpha^{-1} }{ \Imincol} \Big).
	\end{align*}
	That is, by the end of step~\ref{step:init:col:lab}, w.h.p., $\mis(\yt^{(q)},y^{(q)}) \le \gamrow_1$ and $\mis(\zt^{(q)},z^{(q)}) \le \gamcol_1$ for all $q \in [Q]$.
	Since the matching step increases the misclassification rate by at most a factor of 2,  
	the same bounds hold for the overall initial labels at step~\ref{step:local:Lambda:1}. Taking $\gamma_1 = \gamrow_1 \vee \gamcol_1$, we observe that in order to satisfy condition~\eqref{assump:small:gamma:one} of Theorem~\ref{thm:without:spec:clust}, it is enough to have
	\begin{align*}
	\f{\Q\omega\beta (\Kr\vee \Kc)^2 (\alpha \vee \alpha^{-1}) }{ \Imin \wedge \Imincol} = 
	o \Big( \frac{1}{\beta^2 \omega} \frac{\Imin}{\Kc \Laminf} \wedge \frac{\Imincol}{\Kr \Gaminf} \Big)  
	\end{align*}
	which holds if we require the stronger condition
	\begin{align*}
	\f{\Q\omega\beta (\Kr\vee \Kc)^2 (\alpha \vee \alpha^{-1}) }{ \Imin \wedge \Imincol} = 
	o \Big( \frac1{\beta^2 \omega} \frac{ \Imin \wedge \Imincol}{ (\Kr \vee \Kc) (\Laminf \vee \Gaminf)} \Big). 
	\end{align*}
	But this latter condition is satisfied by assumption~\eqref{assump:bounded:J:strong}. Thus, the assumptions of Theorem~\ref{thm:without:spec:clust} hold with high probability, and so is its result. The proof is complete.
\end{proof}

\subsection{Proof of Corollary \ref{cor:main:res}}
\begin{proof}[Proof of Corollary \ref{cor:main:res}]
	From the proof of Theorem~\ref{thm:main:res}, we have that
	\begin{align}\label{eq:main:result:with:v}
	\mis_k \big(\yh,\,\y\big) 
	= O\lp \sum_{r \neq k}\omega \lp 1+ \f{1}{\epsi_{kr}}\rp \exp\Big( {-} I_{kr} - \f 12 \log \Lammin +v \Big) \rp
	\end{align} 
	fails with probability at most $2e^{-v}+o(1)$. First, we show that
	\begin{align}\label{eq:bound:C:alpha}
	\chi_r := \omega \lp 1+ \f{1}{\epsi_{kr}} \rp \Lammin^{-1/2}=o(1), \quad \text{uniformly in $r$.}
	\end{align}
	By Lemma~\ref{lem:control:theta:diff} in Section~\ref{sec:Poi:error:exponent}, $\epsi_{kr}\ge 2 \big(  J_{kr}^{-1} \wedge 1 \big)$. Hence,
	\begin{align*}
	1+ \frac1{\eps_{kr}}  \,\le\, 			
	1+ \frac12 (J_{kr} \vee 1 ) \,\le\, \frac32 (J_{kr} \vee 1 ) \,\le\, 3 J_{kr} 
	\end{align*}
	using $2 J_{kr} \ge 1$ (see Remark~\ref{rem:typical:rates}). Thus, to show~\eqref{eq:bound:C:alpha}, it is enough to show $\omega J_{kr} /\sqrt{\Lamin} = o(1)$.
	Using $\omega^{-1} \Laminf \le \Lamin$, we have 
	\begin{align*}
	\omega J_{kr} \, \Lamin^{-1/2} =	\frac{\Laminf}{I_{kr}} \Kc \,\omega \Lammin^{-1/2} \le \frac{\Kc \omega^{3/2} \Laminf^{1/2}}{I_{kr}} \le \frac{\Kc \omega^{3/2} \Laminf^{1/2}}{\Imin} = o(1)
	\end{align*}
	where the last equality is by $\omega^3\Kc^2\Laminf=o(\Imin^2)$ which is implied by~\eqref{assump:bounded:J:strong}. Thus, we have~\eqref{eq:bound:C:alpha}, i.e., $\chi := \max_r \chi_r = o(1) $, as desired.
	Now, let $2 v = - \log \chi$. It follows that $e^{-v} = \sqrt{\chi} = o(1)$, and we have
	\begin{align*}
	\mis_k \big(\yh,\,\y\big)  
	=  O\lp \chi  \sum_{r \neq k} \exp\big( {-} I_{kr}  +v \big) \rp
	= O\lp  \sqrt{\chi} \sum_{r \neq k} \exp\big( {-} I_{kr} \big) \rp = 
	o\lp   \sum_{r \neq k} \exp\big( {-} I_{kr}\big)  \rp
	\end{align*}
	completing the proof.
	%
	%
\end{proof}

\subsection{Proof of Example \ref{exa:comparison:with:harry}}\label{sec:proof:harry}
Without loss of generality assume $a>b$ so that $\epsi_{kr}=a/b-1$. Also, $\Lammin\ge b/(\beta\Kr)$. By \eqref{eq:main:result:with:v}, which holds in the general case, we have that
\begin{align*}
\mis_k \big(\yh,\,\y\big) 
&= O\lp \sum_{r \neq k}\omega \lp\f{b}{a}\rp \exp\Big( {-} I_{kr} - \f 12 \log \lp \f{b}{\beta\Kr} \rp +v \Big) \rp\\
&= O\lp \sqrt\beta\omega \Kr^{3/2}b^{-1/2}\exp\lp-\f{(\sqrt a -\sqrt b)^2}{
	\beta\Kr}+v \rp\rp
\end{align*} 
fails with probability at most $2e^{-v}+o(1)$. Assumption $\beta\omega^2\Kr^3=o(b)$ implies 
\begin{align*}
\chi:=\sqrt\beta\omega \Kr^{3/2}b^{-1/2}=o(1).
\end{align*}
Letting $2v=-\log\chi$, the rest of the proof follows similar to that of Corollary \ref{cor:main:res}.

\section{Proof of Theorem~\ref{thm:minimax}} \label{sec:proof:minimax}
The distribution of $A$ depends on $(y,z,P)$, so that the expectation in~\eqref{eq:minimax:lower:bound}  is, in fact, $\ex = \ex_{(y,z,P)}$.  Throughout the proof, we restrict the parameter space  to a fixed $z^*$ and $P^*$ such that the corresponding row mean matrix,  $\Lambda^*$, satisfies $\Imin(\Lambda^*) = I^*$. 
From now on,  instead of writing $\pr_{(y, z^*,P^*)}$ we simply write $\pr_y$ for the distribution on $A$, and similarly for the expectations.

Let us write the optimal sum of errors in the test of $\poi(\Lambda_{k*}^*)$ against $\poi(\Lambda^*_{r*})$ as follows:
\begin{align*}
\psumerr^* &:=  \inf \big[ \pr(\text{Type I error}) +\pr(\text{Type II error})\big] \\
&\ge   \exp\big({-I^* - \f \Kc 2 ( \log \Lammin^*+C')}\big),
\end{align*}
where the inequality is by combining Proposition~\ref{prop:HT:err:rates} and Lemma~\ref{lem:bin:err:lower} (Appendix~\ref{sec:minimax:lemmas}), and noting that the assumption of that lemma is satisfied due to \eqref{assump:sparse:network}.
Note that  $\log\Lammin^*$  is not determined in $\Ss$; however, we can always replace it by $2\log I^*$ by \eqref{assump:bounded:J:strong}.  

The reduced parameter space is the set of all $(y,z^*,P^*)$ such that $y$ belongs to
\begin{align*}
\Ts := \big\{ y \in \{0,1\}^{n \times K}: \; y \; \text{ satisfy } \eqref{assump:balance}. \big\}
\end{align*}
Let $ \alpha := \big(8\beta \Kr \vee  \f{\beta}{\beta-1}\big)$ and 
choose $S \subset [n]$ and $\yt\in\Ts$  such that 
\begin{align*}
\big|\big\{i\in S^c: \yt_i=k\big\}\big| = n_0 :=
\big\lceil \f n{\Kr} \big(1-\frac1\alpha\big) \big\rceil, \quad  \forall k \in [\Kr].
\end{align*}
We then define the further restricted parameter space
\begin{align*}
\Ts' := \big\{ y = (y_S,\yt_{S^c}) :\; y_S \in \{k,r\}^{|S|}\big\}.
\end{align*}
Since $1-\frac1\alpha \ge \frac1\beta$, each label in $\Ts'$ has at least $n/\beta \Kr$ labels in each community, hence $\Ts' \subset \Ts$. The direct misclassification rate, $\dmis$ (cf.~Section~\ref{sec:notation}),   between any two labels in $\Ts'$ is
at most
\begin{align*}
\eps := \frac{|S|}{n} = 1 - \frac{|S^c|}{n} = 1- \frac{\Kr n_0}{n} \in \Big[ \frac1{2\alpha}, \frac1\alpha \Big]
\end{align*}
using $ x \le \lceil x \rceil \le x+1$, and $\Kr/n \le 1/(2\alpha)$, which holds for large $n$, for the lower bound.

In particular, the  $\dmis$ between any two elements in $\Ts'$ is at most $1/(8 \beta \Kr)$ (i.e., at most $n/(8 \beta \Kr)$ labels are different). It follows from Lemma~\ref{lem:optim:id:perm} that the optimal matching permutation between any two label vectors in $\Ts'$ is identity, hence their misclassification rate is the same as their $\dmis$.

%

We next argue that $\yh$ can be restricted to $\Ts'$ as well: 
First suppose that under the optimal permutation $\pi$, $\yh$ has at most $\eps n$ different labels on indices $S^c$ compared to $\yt$, that is, $|\{i\in S^c: \pi(\yh_i)\ne \yt_i\}| \le \eps n$. It follows that the misclassification rate between $\yh$ and any $y \in \Ts'$ is at most $2 \eps \le 1/(4 \beta \Kr)$. Thus, by Lemma~\ref{lem:optim:id:perm}, $\pi$ is the optimal  permutation for matching $\yh$ to any $y\in \Ts'$. Redefining $\yh_i := \pi^{-1}(\yt_i)$ for $i\in S^c$ then gives a uniformly better strategy over $\Ts'$. 
The new $\yh$ equals $\yt$ on $S^c$ up to a permutation, so we can restrict $\yh$ to $\Ts'$. On the other hand, if $|\{i\in S^c: \pi(\yh_i) \ne \yt_i\}| > \eps n$, then setting $\yh$ to be any fixed vector from $\Ts'$ (or randomly choosing from $\Ts'$) gives a better strategy.


\medskip
The minimax risk is lower-bounded by
\begin{align}\label{eq:minimax:lb:1}
\begin{split}
\inf_{\yh } \sup_{y \, \in \, \Ts'} \ex_y [\,\mis(\yh,y) \,]
&= \inf_{\yh \,\in\, \Ts'} \sup_{y \, \in \, \Ts'} \ex_y [\,\dmis(\yh,y)\, ] \\
&= \inf_{\yh \,\in\, \Ts'} \sup_{y \, \in \, \Ts'} \frac1n \sum_{i \in S} \pr_y( \yh_i \neq y_i) 
\\
&\ge \eps \cdot \inf_{\yh \,\in\, \Ts'} \avg_{y \, \in \, \Ts'} \frac1{|S|} \sum_{i \,\in\, S} \pr_y\big( \yh_i \neq y_i\big)  \\
&= \eps \cdot \inf_{\yh \,\in\, \Ts'} \avg_{i \,\in\, S}\; \avg_{y \, \in \, \Ts'} \;\pr_y\big( \yh_i \neq y_i\big) 
\end{split}
\end{align}
where we have used $n = |S|/\eps$ and $\max \ge \avg$. Let us now focus on a single term in the sum over $S$, say $i=1$. For simplicity, let $S \setminus 1 = S \setminus \{1\}$.  Let $\Ts'_u = \{(u,y_{S \setminus 1}, \yt_{S^c}):\;y_{S \setminus 1} \in \{k,r\}^{|S|-1}\}$. Then, $\Ts'$ is the disjoint union of $\Ts'_k$ and $\Ts'_r$ and we have
\begin{align*}
\avg_{y \, \in \, \Ts'} \pr_y( \yh_1 \neq y_1) 
&= \frac{1}{|\Ts'|}  \sum_{y \, \in\, \Ts'} \pr_y( \yh_1 \neq y_1) \\
&= \frac{1}{|\Ts'|}  \sum_{y_{S \setminus 1}} 
\Big[\, \pr_{(k,\,y_{S \setminus 1}, \,\yt_{S^c})} \big( \yh_1 \neq k \big) + 
\pr_{(r,\,y_{S \setminus 1}, \,\yt_{S^c})}\big( \yh_1 \neq r \big) \, \Big] \\
&\ge \frac{1}{|\Ts'|} \sum_{y_{S \setminus 1}} \psumerr^* = \frac12 \psumerr^* 
\end{align*}
where the second equality is by decomposing the sum as $\sum_{y_{S \setminus 1}} \sum_{y_1}$, and the last equality by noting that the sum over $y_{S \setminus 1}$ is over $\{k,r\}^{|S|-1}$ whose cardinality is $|\Ts'|/2$. The same lower bound holds for all other $i \neq 1$ in~\eqref{eq:minimax:lb:1} by symmetry. Hence, we conclude
\begin{align*}
\inf_{\yh } \sup_{y \, \in \, \Ts'} \ex_y [\mis(\yh,y) ] \ge \frac{\eps}{2}\psumerr^* \ge \frac1{4 \alpha}  \psumerr^* .
\end{align*}
Recalling the definition of $\alpha$ and using the assumptions that $\beta > 1$ is constant and $K\le \exp(c L)$ gives the desired result.

\section{Proofs of the main lemmas}
\label{sec:proof:main:lemmas}

In this section, we give the proof of the three main lemmas of Section~\ref{sec:fixed:label:analysis}. We we first give the proofs of Lemma~\ref{lem:param:consist} and~\ref{lem:miss:markov} in Section~\ref{sec:proof:param:consist} and~\ref{sec:proof:miss:markov}. The proof of Lemma~\ref{lem:unif:param:b} is more technical and occupies the remainder of this section, including auxiliary results on the error exponents and Poisson-binomial approximations, in Sections~\ref{sec:err:exponent} and~\ref{sec:approx:lemmas}.

Throughout, we will use the following concentration inequality~\cite[p.~118]{gine2015mathematical}: 

\begin{prop}[Prokhorov]\label{prop:prokh:concent}
	Let $S = \sum_{i} X_i$ for independent centered variables $\{X_i\}$, each bounded by $c < \infty$ in absolute value a.s. and let $v \ge \sum_i \ex X_i^2$, then
	\begin{align}\label{eq:poi:concent}
	\pr\big(S \ge vt \big) \le \exp[ {- v h_c(t)} ], \quad t \ge 0, \quad \text{where}\;
	h_c(t) := \frac3{4c} t \log \big(1+\frac{2c}{3} t\big).
	\end{align}
	Same bound holds for $\pr(S < -vt)$.
\end{prop}
Note that  $h_c(t) \asymp t^2$ as $t \to 0$ and $h_c(t) \asymp t \log t$ as $t \to \infty$.

\subsection{Proof of Lemma \ref{lem:param:consist}}\label{sec:proof:param:consist}

Let us define the confusion matrix as  $R(\zt,\z) \in [0,1]^{\Kc\times \Kc}$ with entries 
\begin{align}\label{cm}
R_{k\ell}(\zt,\z) = \frac1{\nc}\sum_{j=1}^{\nc} 1\{\zt_j = k, \z_j = \ell\}=\f{|j:\zt_j = k, \z_j = \ell|}{\nc}.
\end{align}
We can similarly define $R_{k\ell}(\z,\zt)$.  It is easy verify that $R(\zt,\z)=R(\z,\zt)^T$. By definition~\eqref{eq:global:mean:param:def} of the (global) row mean parameters, 
\begin{align}\label{eq:temp:lambda:3890}
\lambda_{k\ell'}(\y,\zt)=\summ j \nc \summ \ell \Kc P_{k\ell} 1\{\z_j=\ell, \zt_j=\ell'\}=mP_{k\ast}R_{\ast\ell'}(\z,\zt).
\end{align}
To see~\eqref{eq:temp:lambda:3890}, note that since we are using true labels $\y$ in the first argument of $\lambda_{k\ell'}(\y,\zt)$, the averaging $\frac{1}{n_k(\y)} \sum_{i} 1\{y_i = k\} \big(\cdots\big)$ over $i$, in the definition, is vacuous. That is, for any $i$ with $y_i = k$, we have $\lambda_{k\ell'}(\y,\zt) = \sum_{j} \ex[A_{ij}] 1\{\zt_{j} = \ell'\}$. We then further break this sum according to column labels $z_j = \ell$ to get~\eqref{eq:temp:lambda:3890}. 

Recall that $n(z)$ is the vector of sizes of clusters in $\z$ and $\pi(\z) = n(z)/\nc$ is the corresponding proportions.  To simplify, let
\begin{align*}
\Nn(\z) := \diag(n(\z)), \quad \Pi(\z) := \diag(\pi(\z)).
\end{align*}
We have $mI_\Kc=\Nn(z) \Pi(z)^{-1}$
where $I_\Kc$ is the $L\times L$ identity matrix, hence
\begin{align*}
mP_{k\ast}R_{\ast\ell'}(\z,\zt)
=P_{k\ast}\, \Nn(\z)\, \Pi(\z)^{-1}R_{\ast\ell'}(\z,\zt)
=\lambda_{k\ast}(\y,\z)\,\Pi(\z)^{-1}R_{\ast\ell'}(\z,\zt)
\end{align*}
using~\eqref{eq:tru:row:mean:def}.
Let use define
\begin{align*}
U(\z,\zt) := \Pi(\z)^{-1} R(\z,\zt).
\end{align*}
Since $\pi(z)$ contains  the row sums of $R_{\ast\ell'}(\z,\zt)$, $U(\z,\zt)$ is the row-normalized confusion matrix, i.e. $U=(R_{k\ell}/R_{k+})$. We have 
\begin{align}\label{eq:lam:U:righthand}
\lambda_{k\ell'}(\y,\zt)=\lambda_{k\ast}(\y,\z) \, U_{\ast\ell'}(\z,\zt),
\end{align}
and its matrix version $\Lambda(\y,\zt)=\Lambda(\y,\z) \, U(\z,\zt)$. We can similarly define $\U(\yt,\y)=\Pi(\yt)^{-1}R(\yt,\y)$. Recalling 
definition~\eqref{eq:global:mean:param:def}, and some algebra gives
\begin{align*}
\lambda_{k'\ell'}(\yt,\zt)
&=\f 1{n_{k'}(\yt)} \summ i\nr \sum_{k\in[\Kr]} \lambda_{k\ell'}(\y,\zt)1\{\y_i=k,\yt_i=k'\}\\
&=\f 1{n_{k'}(\yt)}\sum_{k\in[\Kr]}\lambda_{k\ell'}(\y,\zt)\, |i:\y_i=k,\yt_i=k'|,
\end{align*}
where to get the first equality one further breaks the sums over $\sum_k 1\{\y_i = k\}$ and use the expression for $\lambda_{k\ell'}(\y,\zt)$ in the comments after~\eqref{eq:temp:lambda:3890}.
Using the definition of the confusion matrix in \eqref{cm}, adapted to row labels, and the definition of $U$, we have
\begin{align}\label{lul}
\lambda_{k'\ell'}(\yt,\zt)
=\f 1{\pi_{k'}(\yt)} R_{k'\ast}(\yt,\y)\, \lambda_{\ast\ell'}(\y,\zt)
=U_{k'\ast}(\yt,\y)\, \lambda_{\ast\ell'}(\y,\zt),
\end{align}
or compactly $\Lambda(\yt,\zt) = \U(\yt,\y) \Lambda(\y,\zt)$.
We also define a  column-normalized confusion matrix,
\begin{align*}
V(\z,\zt):=R(\z,\zt)\Pi(\zt)^{-1}.
\end{align*}

\begin{lem}\label{lem:error:radius}
	(\ref{assump:balance}) and (\ref{assump:missclass}) imply
	\begin{align}
	\max_k \; [1-\U_{kk}(\yt,y)] \; &\le \; \gamma \tag{B3.1}\label{assump:missclass.1}, \\
	\max_\ell \; [1-\U_{\ell \ell}(z,\zt)] \; &\le \; \gamma \tag{B3.2}\label{assump:missclass.2}, \quad \text{and} \\
	\max_\ell \; [1-\V_{\ell \ell}(z,\zt)] \; &\le \; \gamma \tag{B3.3}\label{assump:missclass.3}.
	\end{align}
\end{lem}
\begin{proof} Without loss of generality, assume that the optimal permutation matching $\yt$ to $\y$ is identity, and similarly for $\zt$  to $\z$. By definition, $1-U_{kk}(\yt,\y)$ is the misclassification rate withing the $k$th community of $\yt$, hence 
	\begin{align*}
	1-U_{kk}(\yt,\y)
	= \f{|i:\yt_i=k,\,\y_i\ne k|/\nr}{|i:\yt_i=k|/\nr}
	= \f{|i:\yt_i=k,\,\y_i\ne k|/\nr}{|i:\yt_i=k,\,\y_i\ne k|/\nr+|i:\yt_i=\y_i= k|/\nr}.
	\end{align*}
	Recall that we can write (see Section~\ref{sec:notation})
	\begin{align}\label{eq:mis:temp:647}
	\mis(\yt,\y) = \frac1n |i:\yt_i \neq \y_i| = \frac1n \sum_k |i: \yt_i=k,\,\y_i \neq k| = \frac1n \sum_k |i: \yt_i\neq k,\,\y_i = k|.
	\end{align}
	Then, (\ref{assump:missclass}) and the second equality in~\eqref{eq:mis:temp:647} implies $ |i:\yt_i=k, \y_i\ne k| /\nr \le \gamma/({\beta\Kr})$, while the third equality in~\eqref{eq:mis:temp:647} gives $|i:\yt_i=\y_i =k| / \nr \ge \pi_k(\y)-\gamma/(\beta\Kr)$. Letting $f(x) = x/(x+1)$, 
	\begin{align*}
	1-U_{kk}(\yt,\y)
	= f \Big(\f{|i:\yt_i=k,\,\y_i\ne k|/\nr}{|i:\yt_i=\y_i= k|/\nr}\Big)
	\le \f{\gamma/(\beta\Kr)}{\gamma/(\beta\Kr)+ \pi_k(\y)-\gamma/(\beta\Kr)}=\f\gamma{\pi_k(\y)\beta\Kr}\le\gamma
	\end{align*}
	where the first inequality is by monotonicity of $f$, and the last by~\eqref{assump:balance}. This proves~\eqref{assump:missclass.1}.
	
	\smallskip
	Similarly, $1-U_{\ell \ell}(\z, \zt)$ is the misclassification rate within the $\ell$th community of $\z$, i.e., $\mis_\ell(\zt,z)$, hence
	\begin{align*}
	1-U_{\ell \ell}(\z,\zt)=\f{|j:\z_j
		= \ell, \zt_j\ne \ell|/\nc}{\pi_\ell(\z)}
	= \f\gamma{\pi_\ell(\z)\beta\Kc}\le\gamma,
	\end{align*}
	proving~\eqref{assump:missclass.2}. The same bound holds for $1-U_{\ell \ell}(\zt,\z)$ by an argument similar to that used for $U_{kk}(\yt,\y)$.
	To prove~\eqref{assump:missclass.3}, we observe 
	\begin{align*}
	U(\zt, \z)=\Pi(\zt)^{-1}R(\zt,\z)=\Pi(\zt)^{-1}R(\z,\zt)^T=[R(\z,\zt)\Pi(\zt)^{-1}]^T=V(\z,\zt)^T
	\end{align*}
	hence $1-\V_{\ell \ell}(\z,\zt)=1-U_{\ell \ell}(\zt, \z)\le\gamma$. All statements are true for any $k \in [\Kr]$ and $\ell \in [\Kc]$.
\end{proof}

\subsubsection{Proof of Lemma~\ref{lem:param:consist:a}}
For the lower bound, by~\eqref{eq:lam:U:righthand} and~\eqref{assump:missclass.2},
\begin{align*}
\lambda_{k\ell'}(\y,\zt)
=\lambda_{k\ast}(\y,\z)U_{\ast\ell'}(\z,\zt)
&\,\ge\, \lambda_{k\ell'}(\y,\z) U_{\ell'\ell'}(\z,\zt) \\
&\,\ge\, (1-\gamma)\lambda_{k\ell'}(\y,\z) \ge \lambda_{k\ell'}(\y,\z) - C_\gamma \Laminf
\end{align*}
where the last inequality is by $\gamma \le C_\gamma$ and $ \lambda_{k\ell'}(\y,\z) \le \Laminf$.
%
For the upper bound, we write 
\begin{align*}
\lambda_{k\ast}(\y,\z)\U_{\ast \ell'}(\z,\zt)=\lambda_{k\ell'}(\y,\z)\U_{\ell'\ell'}(\z,\zt)+\sum_{\ell\ne\ell'}\lambda_{k\ell}(\y,\z)\U_{\ell\ell'}(\z,\zt).
\end{align*}
The first term obviously satisfies
$\lambda_{k\ell'}(\y,\z)\,\U_{\ell'\ell'}(\z,\zt)\le \lambda_{k\ell'}(\y,\z)$, hence
\begin{align}\label{lud:2}
\lambda_{k\ell'}(\y,\zt) -  \lambda_{k\ell'}(\y,\z) \le \sum_{\ell\ne\ell'}\lambda_{k\ell}(\y,\z)\U_{\ell\ell'}(\z,\zt).
\end{align}
By (\ref{assump:missclass.3}), for every $\ell' \in[\Kc]$, 
\begin{align}\label{eq:pit:pi:ineq}
\pi_{\ell'}(\z) \,\ge\, \f 1\nc |j:\z_j=\zt_j=\ell'| 
\,= \, \pi_{\ell'}(\zt)\V_{\ell'\ell'}(\z,\zt)
\,\ge\, (1-\gamma)\pi_{\ell'}(\zt).
\end{align}
By (\ref{assump:balance}), for every $\ell'$ and $\ell$, we have $\pi_{\ell'}(\z) \le \beta^2\pi_{\ell}(\z)$, hence
\begin{align*}
U_{\ell \ell'}(\z,\zt)=\f 1{\pi_\ell(\z)} R_{\ell \ell'}(\z,\zt)=\f{\pi_{\ell'}(\zt)}{\pi_\ell(\z)}\V_{\ell \ell'}(\z,\zt)\le \f{\beta^2}{1-\gamma}\V_{\ell \ell'}(\z,\zt).
\end{align*}
Combining with~\eqref{lud:2}
\begin{align}\label{lzztup}
\lambda_{k\ell'}(\y,\zt) -  \lambda_{k\ell'}(\y,\z) 
\,\le\,  \f{\beta^2}{1-\gamma}\sum_{\ell\ne\ell'}\lambda_{k\ell}(\y,\z)\V_{\ell\ell'}(\z,\zt)
\,\le\, \f{\beta^2\gamma}{1-\gamma}\infnorm{\lambda_{k\ast}(\y,\z)}
\end{align}
where the last inequality is by~\eqref{assump:missclass.3} and that $V$ is column normalized. This proves the upper bound, and completes the proof 
of $\infnorm{\Lambda(\y,\zt)-\Lambda}\le C_\gamma\Laminf$. Since we assume $C_\gamma\le 1$, it follows that $\infnorm{\Lambda(\y,\zt)}\le 2\Laminf$.

\subsubsection{Proof of Lemma~\ref{lem:param:consist:b}}
Recalling~\eqref{lul}, we have
\begin{align*}
\lambda_{k'\ell'}(\yt,\zt)
\eqs U_{k'\ast}(\yt,\y)\lambda_{\ast\ell'}(\y,\zt)
\eqs U_{k'k'}(\yt,\y)\lambda_{k'\ell'}(\y,\zt)+ \sum_{k\ne k'} U_{k'k}(\yt,\y)\lambda_{k\ell'}(\y,\zt).
\end{align*}
By~\eqref{assump:missclass.1}, the first term is bounded as
\begin{align*}
(1-\gamma)\lambda_{k'\ell'}(\y,\zt) 
\les U_{k'k'}(\yt,\y)\,\lambda_{k'\ell'}(\y,\zt)
\les \lambda_{k'\ell'}(\y,\zt)
\end{align*}
and the second term as
\begin{align*}
0 \les \sum_{k\ne k'}U_{k'k}(\yt,\y)\lambda_{k\ell'}(\y,\zt)
\les \gamma\infnorm{\lambda_{\ast\ell'}(\y,\zt)}
\end{align*}
recalling that $U$ is row normalized hence $\sum_{k \neq k'} U_{k'k} = 1-U_{k'k'} \le \gamma$, by~\eqref{assump:missclass.1}.  
Combining the two bounds, we have 
\begin{align*}
\lambda_{k'\ell'}(\yt,\zt) - \lambda_{k'\ell'}(\y,\zt) 
&\;\in\; \big[ {-}\gamma\lambda_{k'\ell'}(\y,\zt),\, 0 \big] 
+ \big[ 0, \gamma\infnorm{\lambda_{\ast\ell'}(\y,\zt)}\big] \\
&\; \subseteq  \infnorm{\lambda_{\ast\ell'}(\y,\zt)} \big[ -\gamma , \gamma\big]
\end{align*}
showing that $\infnorm{\Lambda(\yt,\zt)-\Lambda(\y,\zt)}\le \gamma\infnorm{\Lambda(\y,\zt)}$.
Combining with  $\infnorm{\Lambda(\y,\zt)}\le 2\Laminf$ from part~(a) of the lemma, we have the first assertion of part~(b). The second assertion follows from $\gamma \le 1/2$ and part~(a) by triangle inequality.
(Note that assumption $6 C_\gamma \omega\le 1$ in fact implies $\gamma \le 1/6$ since $\beta,\omega \ge 1$ and $\gamma \le C_\gamma$.) 

\subsubsection{Proof of Lemma~\ref{lem:param:consist:c}}
Recalling definitions of $\lamh_{k \ell}$ and $\lambda_{k\ell}(\yt,\zt)$ from~\eqref{eq:Lop:Lamh} and~\eqref{eq:global:mean:param:def}, we have
\begin{align*}
n_k(\yt) \big[\lamh_{k\ell} - \lambda_{k\ell}(\yt,\zt) \big] 
&= \sum_{i=1}^\nr \sum_{j=1}^\nc \big( A_{ij} - \ex[A_{ij}] \big) \, 1\{\yt_i = k, \zt_j = \ell\}
\end{align*}
which is of the form $S = \sum_{ij} X_{ij}$ with independent centered terms $X_{ij} = A_{ij} - \ex[A_{ij}]$ with $|X_i| \le 1$ and $ \sum_{ij}\ex X_{ij}^2 = \sum_{ij}\var(A_{ij}) \le \sum_{ij}\ex A_{ij} = n_k(\yt) \lambda_{k\ell}(\yt,\zt)$. Note that the sums in these expressions run over $\{(i,j):\yt_i = k,\, \zt_j = \ell\}$.  Applying the two-sided version of Proposition~\ref{prop:prokh:concent}, with $v = n_k(\yt) \lambda_{k\ell}(\yt,\zt)$, $t = \tau$ and $c = 1$, we have 
\begin{align*}
\pr\big( \big| \lamh_{k\ell} - \lambda_{k\ell}(\yt,\zt) \big| > \lambda_{k\ell}(\yt,\zt) \, \tau \big)
&=\pr\big(\,  n_k(\yt) \big|\lamh_{k\ell} - \lambda_{k\ell}(\yt,\zt) \big| > n_k(\yt) \lambda_{k\ell}(\yt,\zt) \, \tau \,\big)\\
&\le 2\exp\big(\, {-} n_k(\yt) \lambda_{k\ell}(\yt,\zt) h_1(\tau)\, \big).
\end{align*}
Applying union bound over $(k,\ell) \in [\Kr] \times [\Kc]$, and using part~(b) of this lemma, we have $\infnorm{\Lamh - \Lambda(\yt,\zt)} \le \tau \infnorm{\Lambda(\yt,\zt)} \le 4 \tau \infnorm{\Lambda}$ with probability at least
\begin{align*}
1 - 2 \Kr \Kc  \big(\, {-} \min_{k} n_k(\yt)\, \min_{k,\ell} \lambda_{k\ell}(\yt,\zt) \,h_1(\tau) \, \big).
\end{align*}
We have $n_k(\yt) \ge \nr \pi_k(y) (1-\gamma) \ge \nr (\beta \Kr)^{-1}/2$ using~\eqref{assump:missclass.1},~\eqref{assump:balance} and  $\gamma \le 1/2$; see~\eqref{eq:pit:pi:ineq}. Similarly, since $\infnorm{\Lambda(\yt,\zt) - \Lambda} \le 3 C_\gamma \infnorm{\Lambda}$, we have
\begin{align*}
\min_{k,\ell} \lambda_{k\ell}(\yt,\zt)  \ges \Lamin - 3 C_\gamma \infnorm{\Lambda} \ges
\Lamin(1 - 3 C_\gamma \omega) \ges \Lamin/2.
\end{align*}

\subsection{Proof of Lemma~\ref{lem:miss:markov}}\label{sec:proof:miss:markov}
Let $b_{i*} = b_{i*}(\zt)$.
Recall~\eqref{eq:Yikr:def},~\eqref{eq:Zik:def} and~\eqref{eq:Sk:def}, and let 
\begin{align*}
\Sh_{k} = S_k(\bb,\Lamh), \quad  \Zh_{ik} = Z_{ik}(b_{i*},\Lamh),
\quad \Yh_{ikr} = Z_{ik}(b_{i*},\Lamh).
\end{align*}
For any event $\Ac$ and random variable $X$, let us write $\ex[X;\Ac] := \ex[X 1_{\Ac}]$. Consider the following event: $\Ac := \{\Lamh \in \Bcl(\delta)\}$. Pick some $i \in [N]$ with $y_i = k$. Then,
\begin{align*}
\ex \big[\Sh_k; \Ac\big] 
= \ex\big[\Zh_{ik}; \Ac \big] 
&= \pr\Big(\bigcup_{r \neq k} \big\{\Yh_{ikr} \ge 0 \big\} \cap \Ac \Big) \\
&\le \sum_{r \neq k}\pr\big(Y_{ikr}(b_{i*}(\zt),\Lamh) \ge 0, \; \Lamh \in \Bcl(\delta) \big) \\
&\le \sum_{r \neq k}\pr\big(\exists \Lamt \in \Bcl(\delta),\; 
Y_{ikr}(b_{i*}(\zt),\Lamt) \ge 0 \big) \\
&\le\sum_{r \neq k} \exp\big({-}I_{kr} +\eta'\big) =: p_k
\end{align*}
where the last inequality follows from Lemma~\ref{lem:unif:param} with $\eta_{kr}$ defined there. Using Markov inequality 
\begin{align*}
\pr \big( \Sh_k \ge t p_k\big) 
&\le \pr \big( \{\Sh_k \ge t p_k\} \cap \Ac \big) + \pr (\Ac^c) \\
&\le \frac {\ex[\Sh_k ; \Ac] } {tp_k} +\pr (\Ac^c) \le \frac1t + \pr (\Ac^c).
\end{align*}
for any $t > 0$. The version of Markov inequality used follows from (pointwise) inequality: $1_{\{X \ge u\}} 1_\Ac \le (X 1_{\Ac})/u$. Taking $t = e^u$ complete the proof.



\subsection{Error exponents}\label{sec:err:exponent}
We start by obtaining a bound on the error exponent (i.e., the negative logarithm of the probability of error) for  binary hypothesis testing in an exponential family. This result is a generalization of the result that appears in~\cite{abbe2015community}, and is proved by the same technique. The result (and the technique inspired by~\cite{abbe2015community}) is interesting since it provides a bound different than the classical Chernoff bound on the error exponent~\cite{chernoff1952measure}; see also~\cite{verdu1986asymptotic} and~\cite[Theorem 11.9.1]{cover2006elements}. This leads for example to a sharper control for the case of Poisson hypothesis testing. We start with the result for a general exponential family and then in Section~\ref{sec:Poi:error:exponent} specialize to the case of interest in this paper, the Poisson family.




\textbf{General exponential family.} Let $\pi(t;\gamma)$ denote the density of a 1-dimensional standard exponential family w.r.t. to some measure $\nu$ on $\reals$:
\begin{align}\label{eq:1d:expf:def}
\pi(t;\gamma) = h(t) \exp\big( \gamma t - A(\gamma) \big).
\end{align}
We consider distributions on $\reals^L$ that are products of these distributions, having density:
\begin{align}\label{eq:Lfold:expf:def}
p(x;\theta) = \prod_{\ell=1}^L \pi(x_\ell,\theta_\ell), \quad x= (x_\ell)\in \reals^L,\quad \theta = (\theta_\ell)\in \reals^L
\end{align}
with respect to $\mu = \nu^{\otimes L}$ ($L$-fold product measure whose coordinate measures are all $\nu$).
\begin{prop}\label{prop:Pe:bound}
	Let $p_{r}(x) := p(x;\theta_r)$, $r=0,1$ be two exponential family densities on $\reals^L$ (relative to $\mu = \nu^{\otimes L}$) as defined in~\eqref{eq:1d:expf:def} and~\eqref{eq:Lfold:expf:def}. Assume that $\nu$ is either the Lebesgue measure on $\reals$ or the counting measure on $\ints$, and that $\theta_0 \neq \theta_1$. For $s \in (0,1)$, let
	\begin{align}\label{eq:ths:Is:def}
	\theta_{s\ell} = (1-s) \theta_{0\ell} + s\theta_{1\ell}, 
	\quad\text{and}\quad 
	I_{s\ell} = \big[ (1-s) A(\theta_{0\ell}) + s A(\theta_{1\ell}) \big] - A(\theta_{s\ell}),
	\end{align}
	as well as $I_s:= \sum_{\ell=1}^L I_{s\ell}$, $T := \{\ell: \theta_{0\ell} \neq \theta_{1\ell}\}$ and 
	\begin{align}\label{eq:C:alpha:def}
	C(\alpha) := \int e^{-\alpha| t|} d\nu(t) = 
	\begin{cases}
	\frac2\alpha & \nu \; \text{is Lebesgue},\\
	\frac{1+e^{-\alpha}}{1-e^{-\alpha}} \le \f 2{1-e^{-\alpha}} & \nu \; \text{is counting}.
	\end{cases}
	\end{align}
	Consider testing $p_0$ against $p_1$ using the likelihood ratio test based on a single observation. Let $p_{r}$
	be the probability of error under $p_r$ for $r=0,1$. Then, the sum of the error probabilities is bounded as
	%
	\begin{align}\label{eq:Pe:final:bound}
	\pe{0} + \pe{1} \;\le\; \inf_{\ell \in T}   \inf_{s \in (0,1)} \;
	\Big[\, e^{-I_s}
	\infnorm{\pi(\,\cdot\,;\theta_{s\ell})} 	\,
	C\Big( \min(s,1-s) |\theta_{0\ell} - \theta_{1\ell}| \Big)
	\,\Big].
	\end{align} 
\end{prop}


\begin{rem}
	The proof goes through for any translation invariant measure $\nu$ (e.g., a Haar measure) with an appropriate constant $C(\alpha)$. It also goes through if we replace $t$ in~\eqref{eq:1d:expf:def} with a general sufficient statistic $\phi(t)$, as long as (1) $\phi$ is surjective from the support of the exponential family to $\reals$ and (2) $C(\alpha) = \int e^{-\alpha | \phi(t)|} d\nu(t) < \infty$ for all $\alpha > 0$ and (3) $\phi$ has a measurable inverse.
\end{rem}

\begin{rem}\label{rem:alphaopt:sopt}
	Let $\sopt$ be the maximizer of $s \mapsto I_s$. 
	Then, noting that $\alpha \mapsto C(\alpha)$ is decreasing, Proposition~\ref{prop:Pe:bound} implies
	\begin{align}
	\pe{0} +  \pe{1}  &\le  \exp \Big( {-} I_{\sopt} + \log  \infnorm{\pi(\,\cdot\,;\theta_{\sopt\ell})} + \log C(\alpha^*)\Big), \quad \text{where},\; \label{eq:Pe:alphaopt:bound}\\
	&\alpha^* = \min(\sopt,1-\sopt) \max_{\ell \in [L]}|\theta_{0\ell} - \theta_{1\ell}|. \quad \label{eq:alphaopt:def}
	\end{align}
	The bound is an improvement over the Chernoff bound if $\log  \infnorm{\pi(\,\cdot\,;\theta_{\sopt\ell})}$ is negative and $\log C(\alpha^*)$ is controlled. This is the case for the Poisson distribution as we show in the sequel.
\end{rem}


\subsubsection{Poisson case}\label{sec:Poi:error:exponent}
The Poisson case corresponds to~\eqref{eq:1d:expf:def} with $\gamma = \log \lambda$, $h(t) = (1/t!) 1\{t \ge 0\}$, $\nu =$ the counting measure and $A(\log \lambda) = \lambda$. Letting $\theta_{s\ell}  = \log \lambda_{s\ell}$ for all $s \in [0,1]$, we have from~\eqref{eq:ths:Is:def}
\begin{align*}
\lambda_{s\ell} = \lambda_{0\ell}^{1-s} \,\lambda_{1\ell}^{s}, \quad 
I_{s\ell} =  \big[ (1-s)  \lambda_{0\ell} + s\lambda_{1\ell} \big] - \lambda_{s\ell}.
\end{align*}
We also note that $ |\theta_{0\ell} - \theta_{1\ell}| = | \log (\lambda_{0\ell}/ \lambda_{1\ell})|$. Let us define
\begin{align*}
\sopt = \argmax_{s \in (0,1)} I_s, \quad\text{and},\quad \Iopt =\max_{s \in (0,1)} I_s, 
\quad \text{where} \quad I_s = \sum_{\ell=1}^L I_{s\ell}
\end{align*}
We will assume
\begin{align}\label{eq:01:omega:assump}
\lambda_{0\ell}/\lambda_{1\ell} \in [1/\omega,\omega], \; \forall \ell\in[L],\;\text{  for some $\omega > 1$}.
\end{align}
The following lemma shows that $\sopt$ stays away from the boundary:
\begin{lem}\label{lem:control:sopt}
	Assuming~\eqref{eq:01:omega:assump}, we have $\sopt \in [\frac1{2\omega},1-\frac1{2\omega}]$.
\end{lem}
Proof of this lemma and subsequent results appear in Section~\ref{sec:proof:Poi:error:exponent}.
From~\eqref{eq:alphaopt:def}, we have
$\alp^*= \min(\sopt, 1-\sopt) \max_{\ell}| \log (\lambda_{0\ell}/\lambda_{1\ell})|$ in the Poisson case.
Defining 
\begin{align}\label{eq:eps:01:def}
\eps_{01}:= \eps_{01}(\Lambda) := \max_{\ell\in[\Kc]} \lp \f{\lambda_{0 \ell}}{\lambda_{1\ell}} \vee \f{\lambda_{1 \ell}}{\lambda_{0\ell}}\rp - 1 ,
\quad 
\alpha_{01} := \frac1{2\omega} \log (1+ \eps_{01})
\end{align}
we note that $\alpha^*  = \min(\sopt, 1-\sopt) \log(1+\epsi_{01})$, hence Lemma~\ref{lem:control:sopt} implies $\alpha^*  \ge \alpha_{01}$, that is, $C(\alp^*) \le C(\alpha_{01})$ in~\eqref{eq:Pe:alphaopt:bound}, where $C(\cdot)$ has the form given in~\eqref{eq:C:alpha:def} for the counting measure, i.e., 
\begin{align}\label{eq:C:alpha:counting}
C(\alpha_{01})
=\f{1+e^{-\alpha_{01}}}{1-e^{-\alpha_{01}}}
\le \f 2{1-e^{-\alpha_{01}}}
\stackrel{(a)}{\le} \Big( \f 4 {\eps_{01}} + 3 \Big)\omega
\end{align}
where the last inequality is by the following lemma:

\begin{lem}\label{lem:approx:C:alpha}
	Inequality~(a) in~\eqref{eq:C:alpha:counting} holds.
\end{lem}

Next we bound the maximum of the density:



\begin{lem}[\cite{hodges1960poisson}]\label{lem:Poi:pmf:max}
	Let $\pi(t;\log \lambda) = e^{-t} (\lambda^t / t!) 1\{t \ge 0\}$ be the desnity of the Poisson family. Then, for all $\lambda > 0$,
	\begin{align*}
	\infnorm{\pi(\,\cdot\,;\log \lambda)} \le \Big(1 + \frac1{12\lambda}\Big)\frac{1}{\sqrt{2\pi\lambda}} .
	\end{align*}
	In particular $\infnorm{\pi(\,\cdot\,;\log \lambda)} \le \exp\big({-}\frac12 \log \lambda\big)$ for $\lambda \ge 0.056$.
\end{lem}



Combining Lemmas~\ref{lem:control:sopt}, \ref{lem:approx:C:alpha} and~\ref{lem:Poi:pmf:max}, we have the following corollary which gives the following overall bound on the error exponent:

\begin{cor}\label{cor:poi:err:exponent}
	Consider testing two Poisson vector models with mean vectors given by the rows of $\Lambda = \cvecc{\lambda_{0*}}{\lambda_{1*}} \in \reals_+^{2\times L}$, satisfying~\eqref{eq:01:omega:assump}.
	Let $\Lamin = \min_{r\ell} \lambda_{r\ell}$. Then, the sum of the error probabilities for the likelihood ratio test is bounded as
	\begin{align}\label{eq:sharp:err:expo}
	\pe{0} + \pe{1} &\;\le\; 	
	\omega \Big( \f {4}{\eps_{01}} + 3 \Big) \exp\lp 	{-}\Iopt -\frac12 \log \Lamin \rp. 
	\end{align}
\end{cor}

We also have the following general lower bound on $\eps_{01}$ in terms of the information $\Iopt$:
\begin{lem}\label{lem:control:theta:diff}
	Let $\Lambda = \cvecc{\lambda_{0*}}{\lambda_{1*}} \in \reals_+^{2\times L}$. There exists $\ell \in [L]$ such that
	\begin{align*}
	\Big|\log \f{\lambda_{0\ell}}{\lambda_{1\ell}}\Big|\ge \f 12 \log \lp 1+\f{8 \I^*}{\Kc\Laminf}\rp,
	\end{align*}
	which implies $\eps_{01} \ge \min\big( \f{2I^*}{\Kc\Laminf}, 2 \big)$.
\end{lem}
Although the bound in Lemma~\ref{lem:control:theta:diff} holds without any further assumption, it is not always tight. The difference in our two sets of results, namely~\eqref{eq:little:o:rate} and~\eqref{eq:Big:O:rate} is due to using the sharper bound~\eqref{eq:sharp:err:expo} versus replacing $\eps_{01}$ with its universal lower bound.

\begin{lem}[Perturbation of $\eps_{01}$]\label{lem:approx:C:alpha:p}
	Suppose $\Lambda'\in \Bcl(\delta)$, and let $\eps'_{01} = \eps_{01}(\Lambda')$ and $\eps_{01} = \eps_{01}(\Lambda)$ as in~\eqref{eq:eps:01:def}, and assume~\eqref{eq:01:omega:assump}.  
	Then $\epsi'_{01}\ge \epsi_{01}-2\omega(1+\epsi_{01})\delta$.
\end{lem}

%


\subsection{Approximation results for Lemma~\ref{lem:unif:param:b}} \label{sec:approx:lemmas}
Let us collect some approximation lemmas that will be used in the proof of Lemma~\ref{lem:unif:param}(b). The proofs can be found in Appendix~\ref{sec:proof:Poi:error:exponent}. We write pmf for the probability mass functions. We recall that a Poisson-binomial variable with parameter $(p_1,\dots,p_n)$ is one that can be written as $\sum_{i=1}^n X_i$ where $X_i \sim \text{Ber}(p_i)$, independent over $i=1,\dots,n$. We write pmf for the probability mass function.

\begin{lem}[Poisson-binomial approximation]\label{lem:poi:poibi}
	Let $\ppmf(x; \lambda)$ be the pmf of a Poisson variable with mean $\lambda$, and let $\pbpmf(x, p)$ be the pmf of a Poisson-binomial variable with parameters $p=(p_1,\dots,p_n)$ where $\summ j {n} p_j=\lambda$. Let $p^*:=\max_{j\in[n]}p_j$. Then,
	\begin{align*}
	\f{\pbpmf(x;p)}{\ppmf(x;\lambda)}\le e^{xp^*}, \quad \forall x \in \ints_+.
	\end{align*}
\end{lem}

This result immediately extends to the comparison between vector versions of the two distributions: 
\begin{cor}[Poisson-binomial approximation]\label{cor:poi:poibi:vec}
	Let $p^{(\ell)} = (p_{1}^{(\ell)}, \dots,p_{n_\ell}^{(\ell)}) \in [0,1]^{n_\ell}$ be a vector of probabilities for each $\ell \in [L]$ and let $\lambda^{(\ell)} = \sum_{i=1}^{n_\ell} p_i^{(\ell)} \in \reals_+$. Let
	\begin{align}\label{eq:vec:poi:bin:pmf}
	\Pbpmf(x, (p^{(1)},\dots,p^{(L)})) := \prod_{\ell=1}^L \pbpmf(x_\ell;p^{(\ell)}),  \quad \text{for each}\;
	x = (x_1,\dots,x_L)  \in \ints_+^L
	\end{align}
	be the pmf of a vector Poisson-binomial variable, and $	\Ppmf(x, (\lambda^{(1)},\dots,\lambda^{(L)})) = \prod_{\ell=1}^L \ppmf(x_\ell;\lambda^{(\ell)})$ be the corresponding vector Poisson pmf. Then, we have
	\begin{align*}
	\frac{\Pbpmf(x, (p^{(1)},\dots,p^{(L)}))}{\Ppmf(x, (\lambda^{(1)},\dots,\lambda^{(L)}))} 
	\le \exp \Big(p^* \sum_{\ell=1}^L x_\ell \Big), \quad \forall x \in \ints_+^L,
	\end{align*}
	where $p^* = \max\{ p_i^{(\ell)}:\; i\in[n_\ell], \ell\in [L]\}$.
\end{cor}

\begin{lem}[Poisson likelihood approximation]\label{lem:poi:likelihood:approx}
	Suppose $\max(|\lambda_1-\lambda|, |\lambda_2-\lambda|)\le \rho\le \f 13\lambda$, then for any $x\in\mathbb Z_+$, we have
	\begin{align*}
	\f{\phi(x;\lambda_1)}{\phi(x;\lambda_2)}\le \exp\lp \f{3\rho x}\lambda +2\rho \rp.
	\end{align*}
\end{lem}

\begin{lem}[Degree Truncation]\label{lem:degree:truncation}
	Let $b_{i+}=\sum_{\ell\in[\Kc]}b_{i\ell} = \sum_{j=1}^\nc A_{ij}$ be the degree of (row) node~$i$. Then,
	\begin{align*}
	\P(b_{i+} > 5\Kc\Laminf) \le \exp(-3\Kc\Laminf).
	\end{align*}
\end{lem}

\begin{proof}[ of Lemma~\ref{lem:degree:truncation}]
	Let row node $i$ belong to row cluster $k$, and let $b_{i+}=\sum_{\ell\in[\Kc]}b_{i\ell} = \sum_{j=1}^\nc A_{ij}$ be its degree, with expectation $\lambda_{k+}:=\sum_{\ell\in[\Kc]}\lambda_{k\ell}$. 
	By definition, we have $\lambda_{k+}\le \Kc\Laminf$. We would like to find an upper bound on the probability 
	\begin{align*}
	\P\big(\,b_{i+} > 5\Kc\Laminf\, \big)\le \P\big(\,b_{i+}-\lambda_{k+} > 4\Kc\Laminf\,\big)
	\end{align*}
	We let $v=\lambda_{k+}$, $vt=4\Kc\Laminf$, so $t\ge 4$. By Proposition~\ref{prop:prokh:concent}, we have
	\begin{align*}
	\P \big(b_{i+}-\lambda_{k+} > 4\Kc\Laminf \big) \le \exp\Big[ {-}\f 34 vt \log\lp 1+\f{2t}3\rp \Big]
	\le \exp\lp {-}\f 34 vt\rp\le \exp(-3\Kc\Laminf).
	\end{align*}
\end{proof}

\subsection{Proof of Lemma~\ref{lem:unif:param:b}}\label{sec:proof:unif:param:b}
Fix $i \in [\nr]$  such that $\y_i = k$, and $\zt \in [\Kr]^\nr$ and let $b_{i\ast} = b_{i \ast}(\zt)$.
Throughout, let $\Lambda' = (\lambda'_{k\ell}) := \Lambda(y,\zt)$ which belongs to $\Bcl(\delta)$ by assumption. Denoting the $k$th row of $\Lambda'$ as $\lambda'_{k\ast}$, we have $\ex[b_{i*}] = \lambda_{k\ast}$.
For $r \neq k \in [\Kr]$, $i$ such that $\y_i = k$ and $\Lamt\in\Bcl(\delta)$,
\begin{align*}
Y_{ikr}(b_{i*}, \Lamt)  
= \sum_{\ell=1}^\Kc b_{i\ell} \log \frac{\lamt_{r\ell}}{\lamt_{k\ell}} + \lamt_{k\ell} - \lamt_{r\ell}
\le 	\summ \ell \Kc \Big[ b_{i\ell} \log \frac{\lambda_{r\ell}+\rho}{\lambda_{k\ell}-\rho}+\lambda_{k\ell}-\lambda_{r\ell}+2\rho \Big] :=Y^*
\end{align*}
where $\rho:=\delta\infnorm{\Lambda}$ is the radius of $\Bcl(\delta)$. Hence, 
\begin{align*}
\P(\exists \Lamt\in\Bcl, Y_{ikr}(\Lamt) \ge 0) 
&\;\le\; \P(Y^*\ge 0) 
\;=\; \pr (b_{i*} \in F),
\end{align*}
where we have defined (recalling the definition of $\poillr$ from~\eqref{eq:joint:poi:llr}):
\begin{align*}
F := \Big \{x\in \Z_+^\Kc:\; \poillr(x;  \lambda_{r\ast}+\rho \mid \lambda_{k\ast}-\rho) \,\ge\, -2 L\rho \Big\}.
\end{align*}


\paragraph{Degree truncation.}
Let $b_{i+}=\sum_{\ell\in[\Kc]}b_{i\ell} = \sum_{j=1}^\nc A_{ij}$ be the degree of (row) node $i$, and 
\begin{align}
E=\Big\{x\in \Z_+^\Kc: \summ \ell \Kc x_\ell \le 5\Kc\Laminf \Big\}.
\end{align}
Using Lemma~\ref{lem:degree:truncation}, we have $\P(b_{i\ast}\notin E) \le \exp(-3\Kc\Laminf)$, 
which is faster than the rate we want to establish. Hence, for the rest of the proof it is enough to work on $\{b_{i\ast} \in E\}$. We have the following two approximations on this event:

\paragraph{Poisson-binomial approximation.}
Recall that $P$ is the connectivity matrix and we have,
\begin{align}\label{eq:Pinf:bound}
\infnorm{P} \le \frac{\infnorm{\Lambda}}{\min_i n_i(z)} \le \frac{\beta L \infnorm{\Lambda}}{\nc} 
\end{align}
where the first inequality follows from definition of $\Lambda$ in~\eqref{eq:tru:row:mean:def}, and the second from assumption~\eqref{assump:balance}.
We note that $b_{i\ell}=b_{i\ell}(\zt) = \sum_{j=1}^{\nc} A_{ij} 1\{\zt_j = \ell\}$ as defined in \eqref{eq:Bc:def}, follows a Poisson-binomial distribution. In order the describe the parameters of this distribution, let us introduce the following notation 
\begin{align*}
\lab_\ell(\zt) := (z_j:\; j\in [\nc] \; \text{such that} \;\zt_j = \ell \},
\end{align*}
that is, the vector of true labels associated with nodes in the $\ell$th cluster of $\zt$. Then, $P_{k,\lab_\ell(\zt)} = (P_{k,z_j}: \; j \in [m]\; \text{s.t.}\; \zt_j = \ell) \in \reals^{n_\ell(\zt)}$ is the probability vector associated with the Poisson-binomial distribution of $b_{i\ell}$. Also, let 
\begin{align*}
\lab(\zt) := (\lab_1(\zt),\dots,\lab_\Kc(\zt)), \quad \text{and}
\quad P_{k,\lab(\zt)} := (P_{k,\lab_1(\zt)},\dots,P_{k,\lab_\Kc(\zt)}).
\end{align*}
Then, we can say that $b_{i*} = b_{i*}(\zt)$ is a product Poisson-binomial distribution with parameter $P_{k,\lab(\zt)}$. In particular, $b_{i*}$ has pmf $\Pbpmf(x;P_{k,\lab(\zt)})$ as defined in~\eqref{eq:vec:poi:bin:pmf}. We also note that $\ex [b_{i*}(\zt)] = \lambda_{k*}(\y,\zt) =:\lambda'_{k*}$.
%
It follows from Corollary~\ref{cor:poi:poibi:vec}, noting that  $\infnorm{P_{k,\lab(\zt)}} \le \infnorm{P}$ combined with~\eqref{eq:Pinf:bound},
\begin{align*}
\frac{\Pbpmf(x;P_{k,\lab(\zt)})}{\Ppmf(x;\lambda'_{k*})} 
\le \exp \Big( \infnorm{P_{k,\lab(\zt)}} \sum_{\ell=1}^L x_\ell \Big)
\le  \exp \Big(  \frac{5\beta \Kc^2 \infnorm{\Lambda}^2}{\nc} \Big) =: \zeta_1, \quad \forall x \in E.
\end{align*}
\paragraph{Poisson likelihood approximation.} Recall that $\rho = \delta \infnorm{\Lambda}$. Since by assumption, $\omega\delta \le \f13$, we have $\rho\le \f \Laminf{3\omega}\le \f 13\Lammin$. Recall that by assumption $\Lambda' = (\lambda'_{k\ell}) \in \Bcl(\delta)$.
By Lemma \ref{lem:poi:likelihood:approx},
\begin{align*}
\f{\Ppmf(x;\lambda_{k*}')}{\Ppmf(x;\lambda_{k*}-\rho)}
&\le \prod_{\ell\in[\Kc]} \exp\lp \f{3\rho x_\ell}{\lambda_{k\ell}} +2\rho \rp
\le \exp\lp 2\Kc\rho+\f{15\rho}\Lammin \Kc\Laminf\rp \\
&\le \exp\lp 17\omega\Kc\rho \rp=: \zeta_2, \quad \forall x \in E.
\end{align*}

\medskip
With some abuse of notation, we treat $\Ppmf$ and $\Pbpmf$ are measures as well, thus, for example, $\Ppmf(E) = \sum_{x \in E} \Ppmf(x)$. Then, we have
\begin{align}\label{eq:prob:EF:1}
\P(b_{i*} \in E \cap F) 
\;=\; \Pbpmf\big(E \cap F; P_{k*}\big)
\;\le\; \zeta_1 \,\Ppmf\big(E \cap F;\lambda'_{k*}\big)
\;\le\; \zeta_1 \zeta_2\,  \Ppmf\big(E \cap F;\lambda_{k*}-\rho\big).
\end{align}
Thus, it is enough to bound  $\Ppmf(F;\lambda_{k*}-\rho)$ which gives a further upper bound. This quantity is closely related to testing Poisson vector distributions with mean $\lambda_{k*}-\rho$ and $\lambda_{r*}-\rho$ against each other.
Let us write $p_0(x) := \Ppmf(x;\lambda_{k*}-\rho)$ and $p_1(x) := \Ppmf (x;\lambda_{r\ast}+\rho )$ and note that $ \poillr( \, \cdot\,;  \lambda_{r\ast}+\rho \mid \lambda_{k\ast}-\rho) = \log(p_1/p_0)$. We have
\begin{align}
\sum_{x \in F}  \Ppmf(x;\lambda_{k*}-\rho) 
&= \sum_{x \,\in\, \Z_+^L} p_0(x) \;1 \Big\{ \log \frac{p_1(x)}{p_0(x)} \ge -2 L \rho \Big\} \notag\\
&=\sum_{x \,\in\, \Z_+^L}  p_0(x) \;1 \Big\{ \frac{ e^{2L\rho} p_1(x)}{p_0(x)} \ge 1 \Big \} \notag \\
&\le \sum_{x \,\in\, \Z_+^L}  \min \Big( e^{2L\rho} p_1(x),\; p_0(x) \Big) 
\le e^{2L\rho } \sum_{x \,\in\, \Z_+^L} \min \big( p_1(x),\; p_0(x) \big). \label{eq:prob:EF:2}
\end{align}
Let us define 
\begin{align}\label{eq:Is:def:2}
I_s(\lambda_0 \mid \lambda_1) =  \sum_{\ell=1}^L \big[ s \lambda_{0\ell} + (1-s) \lambda_{1\ell} \big] - \lambda_{0\ell}^s \lambda_{1\ell}^{1-s}, \quad \lambda_0,\lambda_1 \in \reals_+^L.
\end{align}
We can now apply Corollary~\ref{cor:poi:err:exponent}. Since $\f{\Laminf+\rho}{\Lammin-\rho}\le\f{\omega\Lammin+\f 13\Lammin}{\f 23\Lammin}\le 2\omega$, we need to substitute $\omega$ in Corollary~\ref{cor:poi:err:exponent} by $2\omega$. It follows that
\begin{align}
\sum_{x \,\in\, \Z_+^L}	\min \big( p_1(x),\; p_0(x) \big) 
&\;\le\; \zeta_3 \exp\Big( {-}I_s(\lambda_{r\ast} +\rho \mid \lambda_{k\ast} -\rho)  -\f 12 \log(\Lammin-\rho) \Big) \notag\\
&\;\le\;  \zeta_3 \exp\Big({-}I_s(\lambda_{k\ast} \mid \lambda_{r\ast}) +2\omega\Kc\rho-\f 12\lp\log\Lammin+\log \f 23\rp \Big) \notag\\
&\;\le\;  8\sqrt{\f 32}\zeta_3 \exp\Big({-}I_s(\lambda_{k\ast} \mid \lambda_{r\ast}) +2\omega\Kc\rho-\f 12\log\Lammin \Big) \label{eq:prob:EF:3}
\end{align}
where $\zeta_3 = {\omega}/(\eps_{kr}-2\omega(1+\epsi_{kr})\delta) + \omega$ from Lemma~\ref{lem:approx:C:alpha:p}, 
and the second line follows from the following elementary inequality:
\begin{align*}
(a -\rho)^{1-s} (b+\rho)^{s} \le a^{1-s} \Big(b^s + \frac{s\rho}{b^{1-s}}\Big) \le a^{1-s} b^{s} + \rho \omega
\end{align*}
assuming $a/b \le \omega$.  Note that $I_{kr} = \sup_{s \in (0,1)} I_s(\lambda_{r\ast} \mid \lambda_{k\ast})$. Putting the pieces~\eqref{eq:prob:EF:1},~\eqref{eq:prob:EF:2} and~\eqref{eq:prob:EF:3} together (and taking supremum over $s$) we have
\begin{align*}
\pr(b_{i*} \in E \cap F)
&\;\le\; 8\sqrt{\f 32} \zeta_2 \zeta_3 \,e^{2 L\rho + 2\omega\Kc\rho } \exp\Big({-} I_{kr} -\f 12\log\Lammin\Big).
\end{align*}
We note that 
\begin{align*}
\log(\zeta_1 \zeta_2 \,e^{2 L\rho + 2\omega\Kc\rho }) 
\;\le\; 17\omega\Kc\rho + \frac{5\beta \Kc^2 \infnorm{\Lambda}^2}{\nc}  + 4 \omega \Kc\rho 
\;\le\; 21\omega\Kc\rho + \frac{5\beta \Kc^2 \infnorm{\Lambda}^2}{\nc}
=: \log \zeta_4
\end{align*}
It follows that
\begin{align*}
\pr(b_{i*} \in E \cap F)
&\le  \zeta_3 \zeta_4 \exp\Big({-} I_{kr} -\f 12\log\Lammin\Big).
\end{align*}
Finally we have
\begin{align*}
\pr(b_{i*} \in F)  &\le \pr(b_{i*} \in E \cap F)  + \pr(b_{i*} \in E^c) \\
&\le 8\sqrt \f 32\zeta_3 \zeta_4 \exp\Big({-} I_{kr} -\f 12\log\Lammin\Big) +  \exp\big({-}3\Kc\Laminf\big) \\
&\le 11 \zeta_3\zeta_4 \exp\Big({-} I_{kr} -\f 12\log\Lammin\Big),
\end{align*}
assuming that $I_{kr}$ and $\Lamin$ are sufficiently large. 
Noting that by the definition of $\eta_{kr}$ in the statement of the theorem, 
$\eta_{kr} = \log(2\zeta_3 \zeta_4)$, the proof is complete.

\appendix

%
%
%
%


\section{Remaining proofs}\label{sec:appendix}

\subsection{Proofs of Sections~\ref{sec:subblk:analysis},~\ref{sec:pert:info} and~\ref{sec:matching:analysis}}\label{sec:proofs:subblk:analysis}

\begin{proof}[Proof of Lemma~\ref{lem:rand:part}]
	We have $n_{k}(\y^{(q)})\sim \text{Hypergeometric}(\nr/Q, n_k(y),\nr)$. For any fixed $k \in [\Kr]$ and $q' \in [\Q]$, the concentration of hypergeometric distribution~\cite{chvatal1979tail} gives $|\pi_{k}(\y^{(q')})-\pi_k(y)| \le \xi$ with probability at least $1-2\exp(-\nr \xi^2/\Q)$. The same probability bound holds for $|\pi_{\ell}(\z^{(q)})-\pi_\ell(\z)| \le \xi$, for any fixed $\ell \in [\Kc]$ and $q \in [\Q]$. Taking the union bound over $k,\ell,q,q'$ gives the desired result.
\end{proof}

\begin{proof}[Proof of Lemma~\ref{lem:subblk:tru:lam:dev}]
	Recall the definition of the true local mean parameters in~\eqref{eq:true:local:mean:def}, and the corresponding global parameters in Section~\ref{sec:local:global:param}. We have
	\begin{align*}
	\lambda_{k\ell}^{(q)} - \f{\lambda_{k\ell}}\Q
	&= P_{k\ell} \lp n_{\ell}(\z^{(q)})-\f{n_{\ell}(\z)}\Q\rp \\
	&= \f{P_{k\ell} \, n_\ell(\z)}\Q \lp \f{\pi_{\ell}(\z^{(q)})}{\pi_\ell(\z)} -1 \rp 
	=\f{\lambda_{k\ell}}\Q \lp \f{\pi_{\ell}(\z^{(q)})}{\pi_\ell(\z)} -1 \rp.
	\end{align*}
	Since $|\pi_{\ell}(\z^{(q)}) - \pi_\ell(\z)| \le \xi$ and $\pi_\ell(\z) \ge 1/(\beta \Kc)$ by assumptions~\eqref{assum:subblk:counts:a} and~\eqref{assump:balance}, the first inequality in~\eqref{eq:subblk:tru:lam:dev} follows, from which we have the second inequality by~\eqref{assum:subblk:counts:a}.
\end{proof}

\begin{proof}[Proof of Lemma~\ref{lem:loc:param:consist}]
	From Lemma~\ref{lem:subblk:tru:lam:dev}, we have $\infnorm{\Lambda^{(q)} - \Lambda/\Q} \le (\xi L \beta )\infnorm{\Lambda/\Q}$ and $\infnorm{\Lambda^{(q)}}\le \f32 \infnorm{\Lambda/\Q}$. Note that $\Lambda^{(q)}$ is the true (local) mean parameter matrix associated with subblock $A^{(q',q)}$, and this subblock has $\nr/(2\Q)$ rows. We will apply Lemma~\ref{lem:param:consist} to the submatrix $A^{(q',q)}$ and sublabels $\z^{(q')}$ and $\y^{(q)}$. In order to do so, we have to verify conditions~\eqref{assump:Lambda}, \eqref{assump:balance}) and~\eqref{assump:missclass} for the subblock. (Condition~\eqref{assump:missclass} is satisfied by assumption.) By Lemma~\ref{lem:subblk:tru:lam:dev}, we have
	\begin{align*}
	\frac{\infnorm{\Lambda^{(q)}}}{ \Lammin^{(q)}} \le 3 \frac{\infnorm{\Lambda}}{\Lamin} \le 3\omega.
	\end{align*}
	By~\eqref{eq:subblock:balance}, the condition~\eqref{assump:balance} holds with $\beta$ replaced with $2\beta$. We also need to replace $\Lamin$ in Lemma~\ref{lem:subblk:tru:lam:dev} with $\Lamin^{(q)} \ge  \Lamin/ (2\Q)$, and $C_\gamma$ with $4 C_\gamma$ (more precisely, we are replacing $C_{\gamma,\beta}$ with $C_{\gamma,2\beta}$). Thus, assuming $6(4C_\gamma)(3\omega) \le 1$, we obtain 
	\begin{align}\label{temp:879}
	\pr \Big( 
	\infnorm{\Lamh^{(q',q)} - \Lambda^{(q)}} \le 4( 4 C_\gamma+ \tau) \infnorm{\Lambda^{(q)}}
	\Big)
	\ge 1 - 2p_1 \Big(\tau;\,\frac{n}{2\Q}, \frac\Lamin{2\Q}, 2\beta \Big)
	\end{align}
	where $p_1(\cdot)$ is as in~\eqref{eq:ptail:Lambh}. Since $\infnorm{\Lambda^{(q)}}\le \f32 \infnorm{\Lambda/\Q}$, $4( 4 C_\gamma+ \tau) \infnorm{\Lambda^{(q)}}\le (24C_\gamma + 6\tau)\infnorm{\Lambda/\Q}$.  Thus, on the event in~\eqref{temp:879}, we have by triangle inequality
	\begin{align*}
	\infnorm{\Lamh^{(q',q)} - \Lambda/\Q} 
	&\le  4( 4 C_\gamma+ \tau) \infnorm{\Lambda^{(q)}} + (\xi L \beta )\infnorm{\Lambda/\Q} \\
	&\le  \Big[ 4( 4 C_\gamma+ \tau) \frac32 + \xi L \beta  \Big] \infnorm{\Lambda/\Q},
	\end{align*}
	which is the desired result.
\end{proof}

\begin{proof}[Proof of Lemma~\ref{lem:info:pert}]
	For the proof, it is enough to consider $\Lambda = \cvecc{\lambda_0}{\lambda_1} \in \reals_+^{2 \times \Kc}$, where $\lambda_0,\lambda_1 \in \reals_+^{\Kc}$ are the two rows of $\Lambda$. Similarly, let $\Lamt = \cvecc{\lamt_0}{\lamt_1} \in \reals_+^{2 \times \Kc}  \in \Bcl(\delta)$. Let us define
	\begin{align}\label{eq:Is:def:1}
	I_s(\lambda_0 \mid \lambda_1) =  \sum_{\ell=1}^\Kc \big[ (1-s) \lambda_{0\ell} + s \lambda_{1\ell} \big] - \lambda_{0\ell}^{1-s} \lambda_{1\ell}^{s}, \quad \lambda_0,\lambda_1 \in \reals_+^L
	\end{align}
	and $\alpha_\ell = \max \{|\lambda_{0\ell} - \lamt_{0\ell}|, \,|\lambda_{1\ell} - \lamt_{1\ell}|\}$. We have
	\begin{align*}
	\big|I_s(\lambda_0 \mid \lambda_1)  - I_s(\lamt_0 \mid \lamt_1) \big| \le  \sum_{\ell=1}^{\Kc}
	\Big[ \alpha_\ell  + 
	\big| \lambda_{0\ell}^{1-s} \lambda_{1\ell}^{s} - \lamt_{0\ell}^{1-s} \lamt_{1\ell}^{s}\big| \Big].
	\end{align*}
	Consider the function $f(a,b) = a^{1-s} b^s$ for $a,b > 0$. Assuming $\max\{a/b, b/a\} \le \omega$, we have
	$$\norm{\nabla f(a,b)}_1 \le (1-s) (b/a)^s + s(a/b)^{1-s} \le (1-s)\omega^s + s \omega^{1-s} \le \omega$$
	using $\omega \ge 1$. It follows that $|a^{1-s} b^{s} - u^{1-s} v^s| \le \omega \max\{|a-u|,|b-v|\}$ for $a,b,u,v > 0$. Thus, $\big|I_s(\lambda_0 \mid \lambda_1)  - I_s(\lamt_0 \mid \lamt_1) \big| \le (1+\omega) \sum_\ell\alpha_\ell \le 2 \omega L \delta \infnorm{\Lambda} $ since $\max_\ell \alpha_\ell \le \delta \infnorm{\Lambda}$ by assumption. Taking the supremum over $s$ gives part~(a).
	%
	%
\end{proof}

\begin{proof}[Proof of Lemma~\ref{lem:optim:id:perm}]
	Assume $\dmis(\yt,y) \le \alpha$ and let  $n_k = |i: \y_i = k|$ and $N_{k k'} = |i:\; \y_i = k,\, \yt_i=k'|$. Then,
	\begin{align*}
	n \dmis(\yt, \y) = \sum_k \sum_{k'\neq k}  N_{kk'} \le \alpha n \le \frac{\alpha}{\pi_k} n_k =: \eps n_k.
	\end{align*}
	It follows that $\sum_{k'\neq k}  N_{kk'}  \le \eps n_k$ for every $k$. We also obtain $N_{k k'} \le \eps n_k$ for all $k$ and $k'$ such that $k \neq k'$.
	Since $\sum_{k'} N_{kk'} = n_k$, we have $N_{kk} \ge (1-\eps) n_k$. Thus, as long as $\eps < 1/2$, we have $N_{kk} > N_{kk'}$ for all $k$ and $k'$ such that $k \neq k'$. That is, the diagonal of the confusion matrix is bigger than every element in the corresponding row.
	Now take $\sigma \neq \id$. Then, there exists $k$ such that $k' := \sigma^{-1}(k) \neq k$.
	\begin{align*}
	N^{\sigma}_{kk} := |i:\; \y_i = k,\, \sigma(\yt_i) =k|\; = \; |i:\; \y_i = k,\, \yt_i =k'| = N_{k k'} < N_{kk}.
	\end{align*}
	Then we have
	\begin{align*}
	n \dmis(\sigma(\yt),\y) = \sum_k (n_k - N^\sigma_{kk}) > \sum_k (n_k - N_{kk}) = n \dmis(\yt,\y).
	\end{align*}
	showing that $\id$ is the unique optimal permutation and proving part~(a). For part~(b), we note that $|\{i: \yt_i = k\}| \ge N_{kk} \ge (1-\eps) n_k > (1/2) n_k$ whenever $\eps < 1/2$.
\end{proof}

\begin{proof}[Proof of Lemma~\ref{lem:three:perms}]
	Assume that $\mis(\yt,\y) \le \alpha$ and $\mis(\yt',\y) \le \alpha$ where $\alpha < \frac14 \min_k \pi_k (\yt)$. By definition of the optimal permutation, $\dmis(\sigma(\yt),\y) \le \alpha$ and $\dmis(\sigma'(\yt'),\y) \le \alpha$. Since $\dmis$ is a metric (being the sum of discrete metrics over the coordinates), we have
	\begin{align*}
	\dmis(\sigma^{-1} \circ \sigma' (\yt'),\yt) = \dmis(\sigma' (\yt'),\sigma(\yt)) \le 2\alpha <  \frac12 \min_k \pi_k (\yt)
	\end{align*}
	where the first inequality is by the triangle inequality for $\dmis$ and the second by assumption. Applying Lemma~\ref{lem:optim:id:perm} gives the desired result. 
\end{proof}

\begin{proof}[Proof of Corollary~\ref{cor:subblock:perm:equality}]
	Take $q=2$ for simplicity. Assume that~\eqref{eq:mis:q-1:q:bound} holds with constant $8$ in place of $32$, which is all we need for this lemma. We have 
	\begin{align*}
	(n/2) \dmis(\sigma_{12}(\yt^{(1)}), \y^{(1)}) \le n \dmis(\sigma_{12}(\yt^{(1,2)}), \y^{(1,2)})
	\end{align*}
	by the definition of the $\dmis$. It then follows that
	\begin{align*}
	\dmis(\sigma_{12}(\yt^{(1)}), \y^{(1)}) < 2 \frac{1}{8\beta \Kr} \le \frac12 \min_k \pi_k(\y^{(1)})
	\end{align*}
	where the second inequality holds by the counterpart of~\eqref{eq:subblock:balance} for row labels. Applying Lemma~\eqref{lem:optim:id:perm} we conclude that $\sigma_{12} = \sigma^*(\yt^{(1)} \to \y^{(1)}) =: \sigma_1$ .
\end{proof}

\begin{proof}[of Corollary~\ref{cor:missing:perm}]
	Take $q=2$ for simplicity. Let $\eps = 1/(32 \beta \Kr)$. By assumption, we have $\mis(\yt^{(1,2)},\y^{(1,2)}) < \eps$ and $\mis(\yt^{(2,3)},\y^{(2,3)}) < \eps$. By Corollary~\ref{cor:subblock:perm:equality}, $\sigma_{12} = \sigma_2$ and $\sigma_{23} = \sigma_2'$. Then, the argument leading to~\eqref{eq:Mis:2:eps:bound} implies $\mis(\yt^{(2)}, \y^{(2)}) < 2 \eps$ and $\mis(\yt'^{(2)}, \y^{(2)}) < 2 \eps$. By assumption,
	\begin{align*}
	\mis(\yt'^{(2)}, \y^{(2)})  < 2 \eps = \f1 {16 \beta \Kr} \le \frac18 \min_k \pi_k(\y^{(2)}) \le \frac14 \min_k \pi_k(\yt^{(2)})
	\end{align*}
	where the second inequality holds by the counterpart of~\eqref{eq:subblock:balance} for row labels, and the third inequality follows from the second inequality and Lemma~\ref{lem:optim:id:perm}(b).
	It thus follows from Lemma~\ref{lem:three:perms} that $\sigma_2^{-1} \circ \sigma'_2 = \sigma^*(\yt'^{(2)} \to \yt^{(2)})$ which is the desired result.
\end{proof}

\subsection{Proofs of Section~\ref{sec:main:analysis}}\label{sec:proofs:main:analysis}

\begin{proof}[Proof of Lemma~\ref{lem:l2dist:I:inequality}]
	Let us define $I:=I_{kr}$ and
	\begin{align*}
	I_s=\sum_{\ell=1}^\Kc (1-s)\lambda_{k\ell}+s\lambda_{r\ell}-\lambda_{k\ell}^{1-s}\lambda_{r\ell}^s. 
	\end{align*}
	in this proof. For $s\in [0, 1]$, $s\mapsto I_s$ is a concave function and $I_0=I_1=0$. We have defined $I:=I_{s^*}=\sup_{s\in[0,1]}I_s$. Suppose $s^*\ge \f 12$, since $0\lp 1-\f 1{2s^*} \rp+\f{s^*}{2s^*}=\f 12$ and $\f 1{2s^*}\ge \f 12$, by concavity,
	\begin{align*}
	I_{1/2}\ge \lp 1-\f 1{2s^*} \rp I_0 + \f 1{2s^*} I_{s^*}\ge \f 12 I_{s^*}=\f I2
	\end{align*}
	Similarly, suppose $s^*\le \f 12$, $I_{1/2}\ge I/2$ still holds,
	from which it follows that
	\begin{align*}
	\sum_{\ell\in[\Kc]} (\lambda_{r\ell}-\lambda_{k\ell})^2 &= \sum_{\ell\in[\Kc]} (\sqrt{\lambda_{r\ell}}-\sqrt{\lambda_{k\ell}})^2(\sqrt{\lambda_{r\ell}}+\sqrt{\lambda_{k\ell}})^2\ge (I/2) (4\Lammin)
	= 2\Lammin I.
	\end{align*}
	Taking the minimum over $k$ and $r$ completes the proof.
\end{proof}

\begin{proof}[Proof of Lemma~\ref{lem:spec:clust:subblk}]
	Recall our choice of $\xi$ in~\eqref{eq:taurow:xi:def}---which will also be assumed in this proof---giving $\pr(\Pf^c) = o(1)$ as shown~\eqref{eq:P:Pf:final:bound}.  By Lemma~\ref{lem:subblk:tru:lam:dev}, we have
	$\Lambda^{(q)} \in \Blam{\Lambda/\Q}{\xi \Kc \beta}$ for all $q \in [Q]$, which combined with Lemma~\ref{lem:info:pert} (applied with $\delta = \xi \Kc \beta$) gives
	\begin{align*}
	|I_{kr}(\Lambda^{(q)}) - I_{kr}(\Lambda/\Q)| \le 2 \omega  (\xi \Kc \beta) \Kc \infnorm{\Lambda/\Q}
	\le  \frac{2 \omega  (\Kc \beta) \Kc \infnorm{\Lambda/\Q}}{\beta\omega(\Kr \vee \Kc)^2 (\Laminf \vee \Gaminf)} \le \frac{2}{\Q},
	\end{align*}
	using~\eqref{eq:taurow:xi:def}.  
	%
	Thus 
	\begin{align*}
	\Imin(\Lambda^{(q)})\ge \Imin(\Lambda/\Q)-\f 2\Q\ge \f{\Imin}{2\Q}
	\end{align*}
	as $\Imin\to\infty$. We are now ready to apply Corollary~\ref{cor:spectral:clustering:I} to the Algorithm~\ref{alg:scerr} operating on subblocks in $\Gcol_1$. It remains to verify that assumption~\eqref{assump:spec:clust:subblk} translates to  condition~\eqref{eq::asump:I} for the subblocks. Indeed,	 
	we have to replace $\Imin$ with $\Imin(\Lambda^{(q)})$, $\omega$ with $3\omega$ (by Lemma \ref{lem:subblk:tru:lam:dev}), $\beta$ with $2\beta$ (by \eqref{eq:subblock:balance}), and $\alpha$ with $\f{\nc/4}{\nr/2}=\f\alpha 2$ since the subblocks in $\Gcol_1$ are of size $\f \nr 2 \times \f \nc 4$. Therefore, by assumption~\eqref{assump:spec:clust:subblk},
	\begin{align}
	\f{(2\beta)^2 (3\omega)\Kr\Kc(\Kr\wedge\Kc) \,(\alpha/2)}{2\Imin(\Lambda^{(q)})} 
	\le \f{6 \Q\beta^2 \omega^2\Kr\Kc(\Kr\wedge\Kc) \,\alpha}{\Imin}
	\le C_1 (1+\kappa)^{-2},
	\end{align}
	verifying condition~\eqref{eq::asump:I} on the subblocks.
	Applying Corollary~\ref{cor:spectral:clustering:I}, we have the misclassification rate of $\yt^{(q)}$ satisfies
	\begin{align*}
	\mis(\yt^{(q)}, \y^{(q)}) \;\le \; 
	\f{(1+\kappa)^2(3\omega)(2\beta) \Kc(\Kr\wedge\Kc)(\alpha/2) }{2C_1(\Imin/2\Q)}
	\end{align*}
	which is the desired result.
\end{proof}

\subsection{Proofs of Section~\ref{sec:err:exponent}}

\begin{proof}[Proof of Proposition~\ref{prop:Pe:bound}]
	\textbf{Step~1: Interpolation.} Assume without loss of generality that $\theta_{01}\neq \theta_{11}$ and fix some $s \in(0,1)$. It is enough to establish the bound for $\ell=1$ and this particular $s$.
	Let $\pe{+} := \pe{0} + \pe{1}$ the be sum of the error probabilities  under the two hypothesis.
	Then, 
	\begin{align}\label{eq:Pe:first:expr}
	\pe{+} 
	&= 	\int p_0 1\{p_0 \le p_1\} d \mu  + \int p_1 1\{p_1 < p_0\} d \mu \\
	&= \int \min(p_0,p_1) d\mu  
	= \int p_0^{1-s} \,p_1^{s} \, \min(\lr^{s},\lr^{s-1}) d\mu.
	\end{align}
	where $\lr = p_0/p_1$ is the likelihood ratio.
	Let $p_{r\ell}:= \pi(\,\cdot\,;\theta_{r\ell})$ so that $p_r(x) = \prod_{\ell} p_{r\ell}(x_\ell)$. Similarly, let
	\begin{align}\label{eq:ps:def}
	p_s := 
	\frac{p_0^{1-s} \,p_1^{s}}{ \int p_0^{1-s} \, p_1^{s} d\mu},\quad \text{and}\quad
	p_{s\ell} := 
	\frac{p_{0\ell}^{1-s}\,p_{1\ell}^{s}}{ \int p_{0\ell}^{1-s}\,p_{1\ell}^{s} d\nu}
	\end{align}
	It is easy to see that $p_s(x) = \prod_{\ell=1}^L p_{s\ell}(x_\ell)$ and each $p_{s\ell}$ is a probability density (w.r.t. $\nu$). One can also verify that
	\begin{align*}
	\int p_{0\ell}^{1-s}\,p_{1\ell}^{s} d\nu = e^{-I_{s\ell}}, \quad \text{and}\quad p_{s\ell} = \pi(\,\cdot\,; \theta_{s\ell}),
	\end{align*}
	hence $p_s = p(\,\cdot\,;\theta_s)$ using definition~\eqref{eq:Lfold:expf:def}. That is, $p_s$ defined in~\eqref{eq:ps:def} belongs to the same exponential family, with parameter $\theta_s$ interpolating $\theta_0$ and $\theta_1$.
	We also note that $p_{0\ell}^{1-s} \, p_{1\ell}^{s} = e^{-I_{s\ell}} p_{s\ell}$, hence $p_0^{1-s} \,p_1^{s} = e^{-I_s} p_s$. 
	Substituting into~\eqref{eq:Pe:first:expr}, we obtain 
	\begin{align}\label{eq:Pe:interp}
	\pe{+}  = e^{-I_s} \int p_{s}  \, \min(\lr^{s},\lr^{s-1}) d\mu.
	\end{align}

	
	\textbf{Step 2: Reduction to the single component case ($L=1$).}
	Using $p_{r\ell}(t) = \pi(t;\theta_{r\ell})$, 
	we have
	$p_{r\ell}(t)/p_{r\ell}(t') = \exp(\theta_{r\ell} (t-t'))$, hence
	\begin{align*}
	\frac{p_{0\ell}(t)}{p_{0\ell}(t')} \frac{p_{1\ell}(t')}{p_{1\ell}(t)} = 
	\exp \big[ (\theta_{0\ell} - \theta_{1\ell}) (t-t')\big]
	\end{align*}
	Using $p_r(x) = \prod_{\ell} p_{r\ell}(x_\ell)$, the likelihood ratio can be written as
	\begin{align}\label{eq:llh:ratio:def}
	\lr(x) = \frac{p_0(x)}{p_1(x)} = \prod_\ell \lr_\ell(x_\ell), \quad \text{where}\quad 
	\lr_\ell(x_\ell) = \frac{p_{0\ell}(x_\ell)} {p_{1\ell}(x_\ell)}=\exp \big[ (\theta_{0\ell} - \theta_{1\ell})x-A(\theta_{0\ell})+A(\theta_{1\ell})\big].
	\end{align}
	As long as $\theta_{0\ell}\ne\theta_{1\ell}$, $l_\ell$ is well defined on $\reals$ and maps onto $\reals^{++}$. 
	For any $(x_2,\dots,x_L)$, let $x^*_1 = x^*_1(x_2,\dots,x_L)$ be the solution of the following equation:
	\begin{align*}
	\lr_1(x^*_1) \prod_{\ell=2}^L \lr_\ell(x_\ell) = 1
	\end{align*}
	which always exists in $\reals$ (and not necessarily on the support of the exponential family).
	Then, we have, setting $\delta = \theta_{01} - \theta_{11}$,
	\begin{align*}
	\lr(x) = \frac{\lr_1(x_1)}{\lr_1(x^*_1)} = 
	\frac{p_{01}(x_1)}{p_{01}(x^*_1)} \frac{p_{11}(x^*_1)}{p_{11}(x_1)}
	=	\exp \big[ (\theta_{01} - \theta_{11}) (x_1-x_1^*)\big] 
	=  \exp [\delta (x_1 - x_1^*)].
	\end{align*}
	It follows that
	\begin{align*}
	\min(\lr(x)^{s},\lr(x)^{s-1})\le e^{-\min(s,1-s)|\delta (x_1-x_1^*)|} = e^{-\alpha |x_1-x_1^*|}
	\end{align*}
	where we have defined $\alpha := |\delta| \min(s,1-s)$.
	Recall that  $p_{s}(x) = \prod_{\ell=1}^L p_{s\ell}(x_\ell)$  which we write compactly as $p_s = \prod_{\ell=1}^L p_{s\ell}$. Let us write $\mu = \mu^1 \times \mu^{2:L}$ as the product of underlying coordinate measures. By Fubini theorem, we first integrate over the first coordinate in~\eqref{eq:Pe:interp}:
	\begin{align}\label{eq:pe:Fubini}
	e^{I_s}\pe{+} = \int \prod_{\ell=2}^L p_{s\ell} \Big[ 
	\int p_{s1} \min(\lr^{s},l^{s-1}) d\mu^1
	\Big] d\mu^{2:L}
	\end{align}
	Let $J = J(x_2,\dots,x_L)$ denote the inner integral in~\eqref{eq:pe:Fubini} (in brackets). We have the bound
	\begin{align*}
	J \le \int p_{s1}(x_1) e^{-\alpha| x_1-x_1^*|} d\mu^1(x_1) 
	\le \infnorm{p_{s1}} \int e^{-\alpha| x_1-x_1^*|} d\mu^1(x_1).
	\end{align*}
	Note that $x_1^*$ is the only place where dependence on $(x_2,\dots,x_L)$ appears in the bound. Since $\mu^1$ is either the Lebesgue or  the counting measure, and both these measures are translation invariant, the bound is in fact independent of $x_1^*$. 
	That is, we have $J(x_2,\dots,x_L) \le C(\alpha)\infnorm{p_{s1}}$ for all $(x_2,\dots,x_L)$. It follows that the same bound holds for $\pe{+}$ by~\eqref{eq:pe:Fubini}, that is,
	\begin{align*}
	e^{I_s}\pe{+} \;\le\; C(\alpha)\infnorm{p_{s1}}
	\int \Big(\prod_{\ell=2}^L p_{s\ell}\Big) d\mu^{2:L} \;= \;C(\alpha)\infnorm{p_{s1}}
	\end{align*}
	since $\prod_{\ell=2}^L p_{s\ell}$ is a probability density w.r.t $\mu^{2:L}$. Since the choice of the coordinate $\ell=1$ and $s$ was arbitrary, the proof is complete.
	
\end{proof}

\subsection{Proofs of Section~\ref{sec:Poi:error:exponent}}\label{sec:proof:Poi:error:exponent}

\begin{proof}[Proof of Lemma~\ref{lem:control:sopt}]
	Let $f(s)=\summ \ell\Kc(s-1)\lambda_{k\ell}- s\lambda_{r\ell}+\lambda_{k\ell}^{1-s}\lambda_{r\ell}^s$, then $f(s)$ is a concave function of $s$ on $\reals_+$. Since $f(0)=f(1)=0$, $s^*\in (0,1)$. First, we show the statement is true when $\Kc=1$. In this case, $s^*$ satisfies
	\begin{align}\label{cri}
	\lambda_{k1}-\lambda_{r1}+\lambda_{k1}^{1-s^*}\lambda_{r1}^{s^*}\log\left(\f{\lambda_{r1}}{\lambda_{k1}}\right)=0
	\end{align}
	Let $x=\f{\lambda_{r1}}{\lambda_{k1}}$. Now \eqref{cri} is equivalent to
	$$1-x+x^{s^*}\log x=0.$$
	Hence
	$$s^*(x)=\f{\log((x-1)/\log x)}{\log x}.$$
	We extend the domain of $s^*(x)$ to 1 by defining $s^*(1)=\f 12$, then $s^*(x)$ is an continuous increasing function on $(0,\infty)$. Since $\f{\lambda_{r1}}{\lambda_{k1}}\in [1/\omega,\omega]$, we have $s^*\in [s^*(1/\omega), s^*(\omega)]\subset (0,1)$. One can observe that $s^*(x)=1-s^*(1/x)$, we have $s^*\in [s^*(1/\omega), 1-s^*(1/\omega)]$. One can also observe that $s^*(x)\ge \f x{2}$ for $x\in[0,1]$, so $s^*\in [\f 1{2\omega}, 1-\f 1{2\omega}]$.
	
	Now suppose $L>1$, let $s^*_\ell$ be the optimizer of $f_\ell(s)=(s-1)\lambda_{k\ell}- s\lambda_{r\ell}+\lambda_{k\ell}^{1-s}\lambda_{r\ell}^s$, we still have $s^*\in [\f 1{2\omega}, 1-\f 1{2\omega}]$. The optimizer $s^*$ of $f(s)=\summ \ell \Kc f_\ell(s)$ satisfies $s^*\in [\f 1{2\omega}, 1-\f 1{2\omega}]$ because $f_\ell(s)$ is concave for every $\ell\in[\Kc]$.
\end{proof}

\begin{proof}[Proof of Lemma~\ref{lem:approx:C:alpha}]
	We first note the following Laurent series:
	\begin{align*}
	\frac{1}{1 - (1+x)^{-r}} = \frac{1}{rx} + \frac{r+1}{2r} + \frac{r^2-1}{12r} x - O(x^2),\quad \text{as}\; x \to 0
	\end{align*}
	from which we get the inequality
	\begin{align*}
	\frac{1}{1-(1+x)^{-r}} \le \frac{r^{-1}}{x} + \frac{1+r^{-1}}{2}, \quad \text{for}\; x > 0,\; r < 1.
	\end{align*}
	Let $\eps=\eps_{01}$ and $\alpha=\alpha_{01}$. Applying this inequality with $r= 1/(2\omega)$ and $x = \eps$, and recalling $\alpha = \frac1{2\omega} \log (1+ \eps)$, we have
	\begin{align*}
	C(\alpha) \le \frac{2}{1-e^{-\alpha}} = \frac{2}{1-(1+\eps)^{-1/(2\omega)}} = 2 \frac{2\omega}{ \eps } + (1+ 2\omega).
	\end{align*}
	Using $1 \le \omega$ completes the proof.
\end{proof}

\begin{proof}[Proof of Lemma~\ref{lem:Poi:pmf:max}]
	We have
	\begin{align*}
	e^\lambda \infnorm{\pi(\,\cdot\,;\log \lambda)} = \sup_{t \, \in \,\ints_+} \frac{\lambda^t}{t!} 
	\le \sup_{t \, \in \,\reals_+}\frac{\lambda^t}{\sqrt{2\pi \lambda} (t/e)^t} 
	= \frac{e^\lambda}{\sqrt{2\pi \lambda}},
	\end{align*}
	where the first inequality is by Stirling's approximation and the last equality is by plugging in the maximizer $t=\lambda$.
\end{proof}

\begin{proof}[Proof of Lemma~\ref{lem:control:theta:diff}]
	There exists $\ell\in \Kc$ such that 
	$$(1-s^*)\lambda_{k\ell}+s^*\lambda_{r\ell}-\lambda_{k\ell}^{1-s^*}\lambda_{r\ell}^{s^*}\ge \f{I_{kr}}\Kc$$
	Without loss of generality, we assume $\lambda_{k\ell}<\lambda_{r\ell}$. Let $s^*$ be the optimizer of $I_{kr}$.
	Dividing $\lambda_{k1}$ both side, we have
	$$1-s^*+s^*\f{\lambda_{r\ell}}{\lambda_{k\ell}}-\lp\f{\lambda_{r\ell}}{\lambda_{k\ell}}\rp^{s^*}\ge \f{I_{kr}}{\Kc\lambda_{k\ell}}\ge \f{I_{kr}}{\Kc\Laminf}$$
	Let us define$f(x):=1-s^*+s^*x-x^{s^*}$,
	then for $x>1$, 
	$$f(x)\le \f 12 (1-s^*)s^*(x-1)^2\le \f 18(x-1)^2$$
	Thus $f(x)\ge  \f{I_{kr}}{\Kc\Laminf}$ implies $x\ge \sqrt{1+\f{8I_{kr}}{\Kc\Laminf}}$, or equivalently, $\log x\ge \f 12 \log \lp1+\f{8I_{kr}}{\Kc\Laminf}\rp$.
\end{proof}

\begin{proof}[Proof of Lemma~\ref{lem:approx:C:alpha:p}]
	Without loss of generality, we assume $\lambda_{01}/\lambda_{11}= 1+\eps_{01}$. Letting  $\rho=\delta\Laminf$, we have $\infnorm{\Lambda' - \Lambda} \le \rho$ by definition.
	Let $f(x)=(\lambda_{01}-x)/(\lambda_{11}+x)$. Then $f(x)$ is convex on $(0,\infty)$ with derivative $f'(x)=-(\lambda_{01}+\lambda_{11})/(\lambda_{11}+x)^2$, hence
	\begin{align*}
	\f{\lambda'_{01}}{\lambda'_{11}} \ge 
	\f{\lambda_{01}-\rho}{\lambda_{11}+\rho}
	=f(\rho)
	&\ge f(0)+\rho f'(0)\\
	&=\f{\lambda_{01}}{\lambda_{11}}-\f{\lambda_{01}+\lambda_{11}}{\lambda_{11}^2}\rho
	=1+\epsi_{01}-\f{\lambda_{01}+\lambda_{11}}{\lambda_{11}^2}\rho.
	\end{align*}
	Combined with
	\begin{align}
	\f{\rho}{\lambda_{11}}=\f{\delta\Laminf}{\lambda_{11}}\le\omega\delta \quad \text{and}\quad
	\f{\lambda_{01}+\lambda_{11}}{\lambda_{11}}=2+\epsi_{01}\le 2(1+\epsi_{01})
	\end{align}
	we have $\lambda'_{01}/\lambda'_{11} 
	\ge 1+\eps_{01}-2\omega(1+\epsi_{01})\delta$ which gives the desired result.
\end{proof}

\begin{proof}[Proof of Lemma~\ref{lem:poi:poibi}]
	Let $X_j\sim \poi(p_j)$ independent over $j=1,\dots,n$, so that $\summ j n X_j\sim \poi(\lambda)$. Fix $x \in \ints_+$ and let $\Sc(x) = \{S \subset [n]: |S| = x \}$. For any subset $S$ of $[n]$ and vectors $\alpha,\beta \in \reals_+^n$, let $\psi(\alpha,\beta,S) = \prod_{j \in S} \alpha_j \prod_{j \notin S} \beta_j$. We have
	\begin{align*}
	\ppmf(x;\lambda) &= \pr\Big(\sum_j X_j = x\Big) \\
	&\ge \pr\Big(\sum_j X_j = x,\; X_j \in \{0,1\},\; \forall j \in [n]\Big) \\
	&= \sum_{S\, \in \, \Sc(x)} \Big[ \prod_{j \in S} \pr(X_j = 1) \prod_{j \notin S} \pr(X_j = 0) \Big]
	= \sum_{S\, \in \, \Sc(x)} \psi\big( (p_j e^{-p_j}), (e^{-p_j}),S\big).
	\end{align*}
	On the other hand $\pbpmf(x;p) = \sum_{S\, \in \, \Sc(x)} \psi( (p_j), (1-p_j), S)$. Thus,
	\begin{align*}
	\frac{\pbpmf(x;p)}{\ppmf(x;\lambda)}
	\le
	\frac{\sum_{S\, \in \, \Sc(x)}  \psi\big( (p_j), (1-p_j), S\big)}
	{\sum_{S\, \in \, \Sc(x)}  \psi\big( (p_j e^{-p_j}), (e^{-p_j}),S\big)}
	&\le 
	\max_{S \in \Sc(x)} 
	\frac{\psi\big( (p_j), (1-p_j), S\big)}
	{\psi\big( (p_j e^{-p_j}), (e^{-p_j}),S\big)} \\
	&=  \max_{S \in \Sc(x)} \psi\big( (e^{p_j}), ((1-p_j) e^{p_j}), S\big)
	\end{align*}
	using $(\sum a_i)/( \sum_i b_i) \le \max (a_i/b_i)$ which holds assuming the sums have equal number of terms, all of which positive. Using $(1-x) e^{x} \le 1$, It follows that
	\begin{align*}
	\frac{\pbpmf(x;p)}{\ppmf(x;\lambda)}
	\le \max_{S \in \Sc(x)} \psi\big( (e^{p_j}), (1), S\big) = \max_{S \in \Sc(x)} \prod_{j\in S} e^{p_j} \le e^{x p^*}.
	\end{align*}	
\end{proof}

\begin{proof}[ Proof of Lemma~\ref{lem:poi:likelihood:approx}]
	We have
	\begin{align*}
	\f{\phi(x;\lambda_1)}{\phi(x;\lambda_2)}
	=\lp \f{\lambda_1}{\lambda_2}\rp^x e^{\lambda_2-\lambda_1}
	\le \lp \f{\lambda+\rho}{\lambda-\rho}\rp^x e^{2\rho}
	= \lp 1+\f{2\rho}{\lambda-\rho}\rp^x e^{2\rho}
	\le \exp \lp \f{2\rho x}{\lambda-\rho} + 2\rho\rp.
	\end{align*}
	Since $\lambda-\rho\ge \f 23\lambda$ by assumption, the result follows.
\end{proof}

\begin{proof}[ Proof of Lemma~\ref{lem:degree:truncation}]
	Let row node $i$ belong to row cluster $k$, and let $b_{i+}=\sum_{\ell\in[\Kc]}b_{i\ell} = \sum_{j=1}^\nc A_{ij}$ be its degree, with expectation $\lambda_{k+}:=\sum_{\ell\in[\Kc]}\lambda_{k\ell}$. 
	By definition, we have $\lambda_{k+}\le \Kc\Laminf$. We would like to find an upper bound on the probability 
	\begin{align*}
	\P\big(\,b_{i+} > 5\Kc\Laminf\, \big)\le \P\big(\,b_{i+}-\lambda_{k+} > 4\Kc\Laminf\,\big)
	\end{align*}
	We let $v=\lambda_{k+}$, $vt=4\Kc\Laminf$, so $t\ge 4$. By Proposition~\ref{prop:prokh:concent}, we have
	\begin{align*}
	\P \big(b_{i+}-\lambda_{k+} > 4\Kc\Laminf \big) \le \exp\Big[ {-}\f 34 vt \log\lp 1+\f{2t}3\rp \Big]
	\le \exp\lp {-}\f 34 vt\rp\le \exp(-3\Kc\Laminf).
	\end{align*}
\end{proof}

\subsection{Proof of Lemma~\ref{lem:unif:param:a}}\label{sec:proof:unif:param:a}

\begin{proof}[Proof of Lemma~\ref{lem:unif:param}(a)]
	Fix $\zt$ and let $b_{i*} = b_{i*}(\zt)$. For $r \neq k \in [\Kr]$ and $i$ such that $\y_i = k$, and $\Lamt\in\Bcl(\delta)$,
	\begin{align*}
	Y_{ikr}(b_{i*},\Lamt)  
	= \sum_{\ell=1}^\Kc \Big[ b_{i\ell} \log \frac{\lamt_{r\ell}}{\lamt_{k\ell}} + \lamt_{k\ell} - \lamt_{r\ell} \Big]
	\;\le\; \summ \ell \Kc  \Big[ b_{i\ell} \log \frac{\lambda_{r\ell}+\rho}{\lambda_{k\ell}-\rho}+\lambda_{k\ell}-\lambda_{r\ell}+2\rho \Big]
	\;:=\; Y^*
	\end{align*}
	where $\rho:=\delta\infnorm{\Lambda}$ is the radius of $\Bcl(\delta)$. Hence $\P(\exists \Lamt\in\Bcl, Y_{ikr} \ge 0)\le \P(Y^*\ge 0)$. By Markov inequality, we have $\pr(Y^* \ge 0) \le \ex [e^{s Y^*}]$ for any $s \ge 0$. To simplify the notation, let us write $v_{\ell} = s \log [(\lambda_{r\ell}+\rho)/(\lambda_{k\ell}-\rho)]$ and $w_\ell = s( \lambda_{k\ell} - \lambda_{r\ell}+2\rho)$, so that $s Y^* = \summ \ell \Kc b_{i\ell}\, v_\ell + w_\ell$. By independence, we have
	\begin{align*}
	\log \ex [e^{s Y_*}] 
	= \log \ex \Big[\prod_{\ell=1}^\Kc  e^{ b_{i\ell}\, v_\ell + w_\ell}\Big]
	= \summ\ell\Kc \log \ex [e^{ b_{i\ell}\, v_\ell + w_\ell}]
	= \summ \ell\Kc \log \big[ e^{w_\ell} \ex e^{b_{i\ell} v_\ell} \big].
	\end{align*}
	Since the mgf of a Poisson-binomial variable is bounded above by that of a Poisson variable with the same mean,
	\begin{align*}
	\log \ex e^{s Y_{ikr}} \le
	\summ \ell\Kc w_\ell + \psi\big( v_\ell, \lambda_{k\ell}(y,\zt) \big)  
	\end{align*}
	where $\psi(t,\mu) = \mu(e^t - 1)$ is the log-mgf of a $\poi(\mu)$ random variable. Recalling the assumption $\Lambda(y,\zt) \in \Bcl$, we have 
	\begin{align*}
	\summ\ell\Kc w_\ell + \psi\big( v_\ell, \lambda_{k\ell}(\y,\zt) \big) 
	&= \summ\ell\Kc
	\Big[ \lambda_{k\ell}(\y,\zt) 
	\left( \frac{\lambda_{r\ell}+\rho}{\lambda_{k\ell}-\rho} \right)^s-\lambda_{k\ell}(\y,\zt)+s( \lambda_{k\ell} - \lambda_{r\ell}+2\rho) \Big].
	\end{align*}
	Since $\Lambda(\y,\zt)\in \Bcl(\delta)$, $\lambda_{k\ell}(\y,\zt)\le \lambda_{k\ell}+\delta\Laminf=\lambda_{k\ell}+\rho$.  Since $\lambda_{k\ell}-\rho = \lambda_{k\ell}-\delta\Laminf \ge \lambda_{k\ell}- \f \Laminf{3\omega} \ge \lambda_{k\ell} - \f 13\Lammin = \f 23\lambda_{k\ell},$
	\begin{align*}
	\lambda_{k\ell} \left( \frac{\lambda_{r\ell}+\rho}{\lambda_{k\ell}-\rho} \right)^s 
	&=\lambda_{k\ell} \left( \frac{\lambda_{r\ell}-\rho\f{\lambda_{r\ell}}{\lambda_{k\ell}}+(1+\f{\lambda_{r\ell}}{\lambda_{k\ell}})\rho}{\lambda_{k\ell}-\rho} \right)^s \\
	&\le \lambda_{k\ell}\left[ \lp  \f{\lambda_{r\ell}}{\lambda_{k\ell}}  \rp^s + s \lp  \f{\lambda_{r\ell}}{\lambda_{k\ell}}  \rp^{s-1} \lp \frac{(1+\f{\lambda_{r\ell}}{\lambda_{k\ell}})\rho}{\lambda_{k\ell}-\rho} \rp \right]\\
	&\le \lambda_{k\ell} \lp  \f{\lambda_{r\ell}}{\lambda_{k\ell}}  \rp^s + \f 32 s\cdot 2\omega\rho\\
	&\le \lambda_{k\ell} \lp  \f{\lambda_{r\ell}}{\lambda_{k\ell}}  \rp^s + 3\omega\rho.
	\end{align*}
	Moreover, $\lambda_{r\ell}+\rho=\lambda_{r\ell}+\delta\Laminf=\lambda_{r\ell}+\omega\delta\Lammin\le \lambda_{r\ell}+\f 13\Lammin\le \f 43\lambda_{r\ell}$. Thus, we have
	\begin{align*}
	\rho \left( \frac{\lambda_{r\ell}+\rho}{\lambda_{k\ell}-\rho} \right)^s 
	\le \rho \lp \f {\f 43 \lambda_{r\ell}}{\f 23 \lambda_{k\ell}}\rp^s
	\le 2\omega\rho.
	\end{align*}
	Taking infimum over $s>0$, by Lemma~\ref{lem:control:sopt}, the maximizer $s^*$ of $I_{kr}$ is always bounded between 0 and 1, hence we have
	\begin{align*}
	\P\big(\exists \Lamt\in\Bcl,\; Y_{ikr}(b_{i*},\Lamt) \ge 0 \big)
	&\le P(Y_*\ge 0)\\
	&\le \exp\big({-}I_{kr}+8\Kc\omega\delta\Laminf \big) 
	=\exp\big({-}(1-\eta')I_{kr}\big).
	\end{align*}
\end{proof}

\subsection{Proofs of Section~\ref{main:res}}\label{sec:proof:sec:main:res}
\begin{proof}[Proof of Proposition~\ref{prop:HT:err:rates}]
	The upper bound has been provided by Corollary~\ref{cor:poi:err:exponent}. Here we will show the lower bound, using the notation established in the proof of Proposition~\ref{prop:Pe:bound} and Section~\ref{sec:Poi:error:exponent}. We rename $\lambda_{k*}$ and $\lambda_{r*}$, and work with $\lambda_{0*}$ and $\lambda_{1*}$ instead, and we assume throughout that,
	$\lambda_{0\ell}, \lambda_{1\ell} \ge 1$ for all $\ell \in [\Kc]$.
	We recall from~\eqref{eq:Pe:interp} that
	\begin{align}
	\pe{+}  = \int \min(p_0, p_1) d\mu = e^{-I_s} \int p_{s}  \, \min(\lr^{s},\lr^{s-1}) d\mu,
	\end{align}
	where $p_s$ is defined in~\eqref{eq:ps:def} and $l$  in~\eqref{eq:llh:ratio:def}. Since $\mu$ is the counting measure, we have
	\begin{align}
	\pe{+}  \ge \max_{x\in \Z_+^\Kc}\, \min(p_0(x),p_1(x))
	= \max_{x\in \Z_+^\Kc} \, e^{-I_s} p_{s}(x)  \, \min(\lr^{s}(x),\lr^{s-1}(x)).
	\end{align}
	Finding the maximizer $x$ over $\Z_+$ gives the lower bound. First, let us extend the Poisson density  
	as $\phi(t;\lambda)=\lambda^t e^{-\lambda}/\Gamma(t+1)$ to any $t \in \reals_+$, so that $l$ is well-defined on $\reals_+^\Kc$, given by
	\begin{align*}
	l(x) =  \exp \lp \sum_{\ell \in [\Kc]} x \log \frac{\lambda_{0\ell}}{\lambda_{1\ell}} - \lambda_{0\ell} + \lambda_{1\ell}\rp,
	\; x \in\reals_+^\Kc.
	\end{align*}
	Recall that $\lambda_{s\ell} = \lambda_{0\ell}^{1-s} \lambda_{1\ell}^{s}$, $\lambda_{s} = (\lambda_{s\ell})$  and $I_{s} = \sum_{\ell} [(1-s) \lambda_{0\ell} + s \lambda_{1\ell} -\lambda_{s\ell}]$ (cf. Section~\ref{sec:Poi:error:exponent}). We note that 
	\begin{align*}
	\frac{d I_s}{ds} 
	= \sum_{\ell\in[\Kc]}-\lambda_{0\ell}+\lambda_{1\ell}+ \lambda_{s\ell} \log\lp\f{\lambda_{0\ell}}{\lambda_{1\ell}}\rp.
	= \log(l(\lambda_s))
	\end{align*}
	The function $s \mapsto I_s$ is concave, smooth, nonconstant (by assumption) and we have $I_0 = I_1 = 0$. Hence, the unique maximizer $s^*$ of $s \mapsto I_s$ belongs to $(0,1)$ and satisfies $dI_s/d s\big|_{s^*}=0$, that is, $\log(l(\lambda_{s^*})) = 0$, or equivalently
	$p_0(\lambda_{s^*})/ p_1(\lambda_{s^*}) = l(\lambda_{s^*})= 1$. By the definition of $p_s$, we have
	\begin{align*}
	e^{-I_{s^*}}p_{s^*}(\lambda_{s^*})
	=p_{0}^{1-s}(\lambda_{s^*})p_{1}^{s}(\lambda_{s^*})
	=p_0(\lambda_{s^*})=p_1(\lambda_{s^*}).
	\end{align*}
	We recall that $p_s$ is the product of Poisson densities with parameters $\lambda_{s\ell}$.
	By a version of  the Stirling's inequality for the Gamma functions~\cite{jameson2015simple}:
	\begin{align*}
	\Gamma(x+1) = x \Gamma(x) \le (2\pi)^{1/2} x^{x+1/2} e^{-x} e^{1/(12x)}, \forall x > 0
	\end{align*}
	hence $\Gamma(x+1) \le C_0\, x^{x+1/2} e^{-x}$ for all $x \ge 1$, where $C_0 = (2\pi)^{1/2} e^{1/12}$. Then,
	\begin{align*}
	\phi(\lambda;\lambda) = \frac{\lambda^\lambda e^{-\lambda}}{\Gamma(\lambda+1)} \ge C_0^{-1} \lambda^{-1/2}, 
	\end{align*}
	from which it follows that 
	\begin{align*}
	p_{s^*}(\lambda_{s^*})
	= \prod_{\ell \in \Kc} \phi(\lambda_{s^*\ell}; \lambda_{s^*\ell}) 
	\ge C_0^{-\Kc}\prod_{\ell\in[\Kc]} \lambda_{s^* \ell}^{-1/2}.
	\end{align*}
	Thus, $e^{-I_{s^*}}C_0^{-\Kc} \prod_{\ell} \lambda_{s^*}^{-1/2}$ is a lower bound on $P_{e,+}$ whenever $\lambda_{s^*}\in\Z_+^\Kc$. In general, $\lambda_{s^*}$ does not have integer coordinates. Instead, pick any $x\in \Z_+^{\Kc}$ satisfying $\|x-\lambda_{s^*}\|_{\ell_\infty}\le 1$. 
	
	Since $t\mapsto \phi(t;\lambda)$ is a quasi-concave function (i.e., upper-level sets are convex), we have $\phi(t;\lambda) \ge \min\{\phi(a;\lambda),\phi(b;\lambda)\}$ for every $t \in [a,b]$, hence, for every $t \in [a-1,a+1]$, we obtain using $\Gamma(x+1) = x \Gamma(x)$,
	\begin{align*}
	\phi(t;\lambda) \ge e^{-\lambda} \min\Big\{  \frac{\lambda^{a-1}}{\Gamma(a)},  \frac{\lambda^{a+1}}{\Gamma(a+2)} \Big\} 
	= \frac{e^{-\lambda} \lambda^{a}}{\Gamma(a+1)}\min\Big\{ \frac{a}{\lambda}, \frac\lambda{a+1} \Big\},
	\end{align*}
	that is, 
	\begin{align*}
	\frac{\phi(t;\lambda)}{\phi(a;\lambda)} 
	\ge \min\Big\{ \frac{a}{\lambda}, \frac\lambda{a+1} \Big\}, \quad t \in [a-1,a+1].
	\end{align*}
	Since $|x_\ell-\lambda_{s^* \ell}| \le 1$,
	\begin{align*}
	p_{0\ell}(x_\ell) 
	\;\ge\; p_{0\ell}(\lambda_{s^*\ell})\,
	\min \Big\{ \f{\lambda_{s^*\ell}}{\lambda_{0\ell}}, \f{\lambda_{0\ell}}{\lambda_{s^*\ell}+1} \Big\} 
	\;\ge\; (2\omega)^{-1} p_{0\ell}(\lambda_{s^*\ell})
	\end{align*}
	where we have used, for any $s \in [0,1]$,
	\begin{align*}
	\min\Big\{\frac{\lambda_{s\ell}}{\lambda_{0\ell}}, \frac{\lambda_{0\ell}}{\lambda_{s\ell}}\Big\} =  \Big( \min\Big\{ \frac{\lambda_{1\ell}}{\lambda_{0\ell}}, \frac{\lambda_{0\ell}}{\lambda_{1\ell}}\Big\}\Big)^{s} \ge \Big(\frac1{\omega}\Big)^s \ge \frac{1}{\omega}
	\end{align*}
	and $\lambda_{s^*\ell}/(\lambda_{s^*\ell} + 1) \ge 1/2$ since $\lambda_{s^*\ell} \ge 1$.
	Similarly $p_{1\ell}(x_\ell)\ge  p_{1\ell}(\lambda_{s^*\ell})/(2\omega)$.
	Hence,
	\begin{align*}
	\sum_{x\in\mathbb Z_+^{\Kc}}\min \{p_0(x),p_1(x) \}
	\ge \f{\min \{p_{0}(\lambda_{s^*}), p_{1}(\lambda_{s^*})\}}{(2\omega)^\Kc}
	=\f{e^{-I_{s^*}}p_{s^*}(\lambda_{s^*})}{(2\omega)^\Kc}
	\ge \f{e^{-I_{s^*}}}{(2C_0\omega)^\Kc}\prod_{\ell\in[\Kc]} \lambda_{s^*}^{-1/2}
	\end{align*}
	where we have used $\min\{p_0(x),p_1(x)\} = e^{-I_s} p_s(x) \min\{l(x)^s, l(x)^{s-1}\}$ and $l(\lambda_{s^*}) = 1$.
	Thus,
	\begin{align}
	P_{e,+}
	&\ge \exp\lp {-}I_{s^*} - \Kc \log (2 C_0\omega) - \f \Kc 2 \log\Laminf\rp \notag \\
	&\ge \exp\lp {-}I_{s^*}-\Kc \log (2 C_0\omega^{3/2}) - \f \Kc 2 \log\Lammin \rp \label{eq:Pe:poi:lower:bnd}
	\end{align} 
	using the assumption $\Laminf \le \omega \Lamin$. The proof is complete.
\end{proof}

\subsection{Auxiliary lemmas for Theorem~\ref{thm:minimax}}\label{sec:minimax:lemmas}
The following lemmas are used in the proof of the minimax Theorem~\ref{thm:minimax}.
\newcommand\pt{\widetilde{p}}

\begin{lem}\label{lem:poi:lower:bounds:bin}
	For a discrete probability distribution $\Lc$ on $\Z_+$, let us write $\pmf(x; \Lc), x \in \Z_+$ for the probability mass function of $\Lc$.  There is a universal constant $c > 0$, such that for $\omega>1$, 
	\begin{align}\label{eq:pmf:bin:poi:lower}
	\pmf\Big( x; \Bin\big(n, \frac \lambda n\big)\Big) \ge c \, \pmf\Big( x; \poi(\lambda )\Big), \quad \forall x \le  2\omega \lambda \le \sqrt n/3.
	\end{align}
\end{lem}

\begin{proof}
	Let $\pt$ be the pmf of the Binomial distribution in~\eqref{eq:pmf:bin:poi:lower}. By the Stirling's approximation,
	\begin{align*}
	\frac{n!}{(n-x)!} &\ge \frac{\sqrt{2\pi n} (n/e)^n }{ \sqrt{2\pi (n-x)} [(n-x)/e]^{n-x}} \cdot \frac{e^{1/(12 n +1)}}{ e^{1/(12(n-x))}}  \\
	&\ge c_1 \sqrt{\frac{n}{n-x}} n^n (n-x)^{-(n-x)} e^{-x} 
	\end{align*}
	where we have used $x  \le 2n/3$ to bound the second factor by $c_1$ from below. Hence,
	\begin{align*}
	\pt(x) = \frac{n!}{x! (n-x)!} \Big(\frac\lambda{n}\Big)^x \Big(1- \frac\lambda{n}\Big)^{n-x} 	\Big(\frac{n-\lambda}{n-x}\Big)^{n-x}
	&= 	\frac{\lambda^x}{x!} \frac{n!}{(n-x)!}  n^{-n} (n-\lambda)^{n-x} \\
	&\ge c_1 \frac{\lambda^x e^{-x}}{x!} \sqrt{\frac{n}{n-x}} \Big(\frac{n-\lambda}{n-x}\Big)^{n-x}.
	\end{align*}
	We have $\sqrt{n/(n-x)} \ge 1$. Using the inequality $1+t \ge (1- t^2) e^{t}$ for $|t| \le 1$,
	\begin{align*}
	\Big(\frac{n-\lambda}{n-x}\Big)^{n-x} = \Big(1 - \frac{\lambda-x}{n-x}\Big)^{n-x} \ge \Big[ 1-  \Big(\frac{\lambda-x}{n-x}\Big)^2\Big]^{n-x} e^{x-\lambda}.
	\end{align*}
	Again by $0 \le x \le 2\omega \lambda \le \sqrt n/3\le n/3$ and using $( 1-x)^n\ge 1-nx$,  we have
	\begin{align*}
	\Big[ 1- \Big(\frac{\lambda-x}{n-x}\Big)^2\Big]^{n-x}
	\ge 1-\f{(\lambda - x)^2}{n-x}
	\ge 1-\f{(2\omega\lambda)^2}{2n/3}
	\ge 1-\f{n/9}{2n/3}
	= \f 56.
	\end{align*}
	It follows that $\pt(x) \ge c_2\, \lambda^x e^{-\lambda}/x!$ which is the desired result.
\end{proof}


\begin{lem}\label{lem:bin:err:lower}
	Let $\pt_{k\ell}$ be the probability mass function of $\poi(\lambda_{k \ell})$ and $p_{k\ell}$ that of $\Bin(n_\ell, \lambda_{k \ell}/n_\ell)$. Let $\pt_k=\bigotimes_{\ell=1}^\Kc \pt_{k\ell}$ and $p_k=\bigotimes_{\ell=1}^\Kc p_{k\ell}$ and similarly define $\pt_r$ and $p_r$. 
	Assume that \[
	\max (\lambda_{k \ell}, \lambda_{r \ell})
	\;\le\; \min\big(\omega\lambda_{k \ell}, \,\omega\lambda_{r \ell}, \,\sqrt {n_\ell}/3 \big), \quad \forall \ell\in[\Kc]
	\] 
	for some $\omega > 1$. Then, there is a universal constant $C>0$
	such that the sum of the type~I and type~II errors of the likelihood ratio test for $p_k$ against $p_r$ satisfies
	\begin{align*} 
	P_{e,+}\ge \exp\lp {-}I_{s^*}-\Kc \log (C\omega^{3/2}) - \f \Kc 2 \log\Lammin \rp,
	\end{align*}
	where $I_{s^*}$ is the information between $\lambda_{k*}$ and $\lambda_{\ell*}$ defined in \eqref{eq:Info:def} and $\Lammin = \min_{\ell}(\lambda_{k \ell},\lambda_{r \ell})$.
\end{lem}

\begin{proof}
	For $x\in \mathbb Z_+^\Kc$ satisfying $\|x-\lambda_{s^*}\|_{\infty}\le 1$, we have $x_\ell\le \lambda_{s^* \ell}+1\le 2\omega\min(\lambda_{k \ell}, \lambda_{r \ell})$. Then by Lemma \ref{lem:poi:lower:bounds:bin}, 
	$p_k(x)\ge c^\Kc \pt_k(x)$ and $p_r(x)\ge c^\Kc \pt_r(x)$. Therefore,
	\begin{align*}
	P_{e,+}&\;\ge\; \max_{x\in \Z_+^\Kc}\, \min\big(p_k(x),p_r(x)\big) \;\ge\; 
	c^\Kc \max_{x:\,\|x-\lambda_{s^*}\|_{\infty}\le 1}\;\min \big(\pt_k(x),\pt_r(x)\big)\\
	&\ge c^\Kc\exp\lp {-}I_{s^*}-\Kc \log (2 C_0\omega^{3/2}) - \f \Kc 2 \log\Lammin \rp\\
	&\ge \exp\lp {-}I_{s^*}-\Kc \log (C\omega^{3/2}) - \f \Kc 2 \log\Lammin \rp,
	\end{align*}
	where the third inequality is by~\eqref{eq:Pe:poi:lower:bnd} and $C$ and $C_0$ are positive universal constants. The proof is complete.
\end{proof}

\section{Extra Simulation Results}\label{sec:extra:simulations}
Here we present  extra simulation results under the setup of Section~\ref{sec:sims}.
The following figure shows the overall NMI and log. error rate for different values of $C$ and $\alpha$: 
\begin{figure}[h]
	\centering
	\begin{tabular}{cc}
		\includegraphics[width=2.5in]{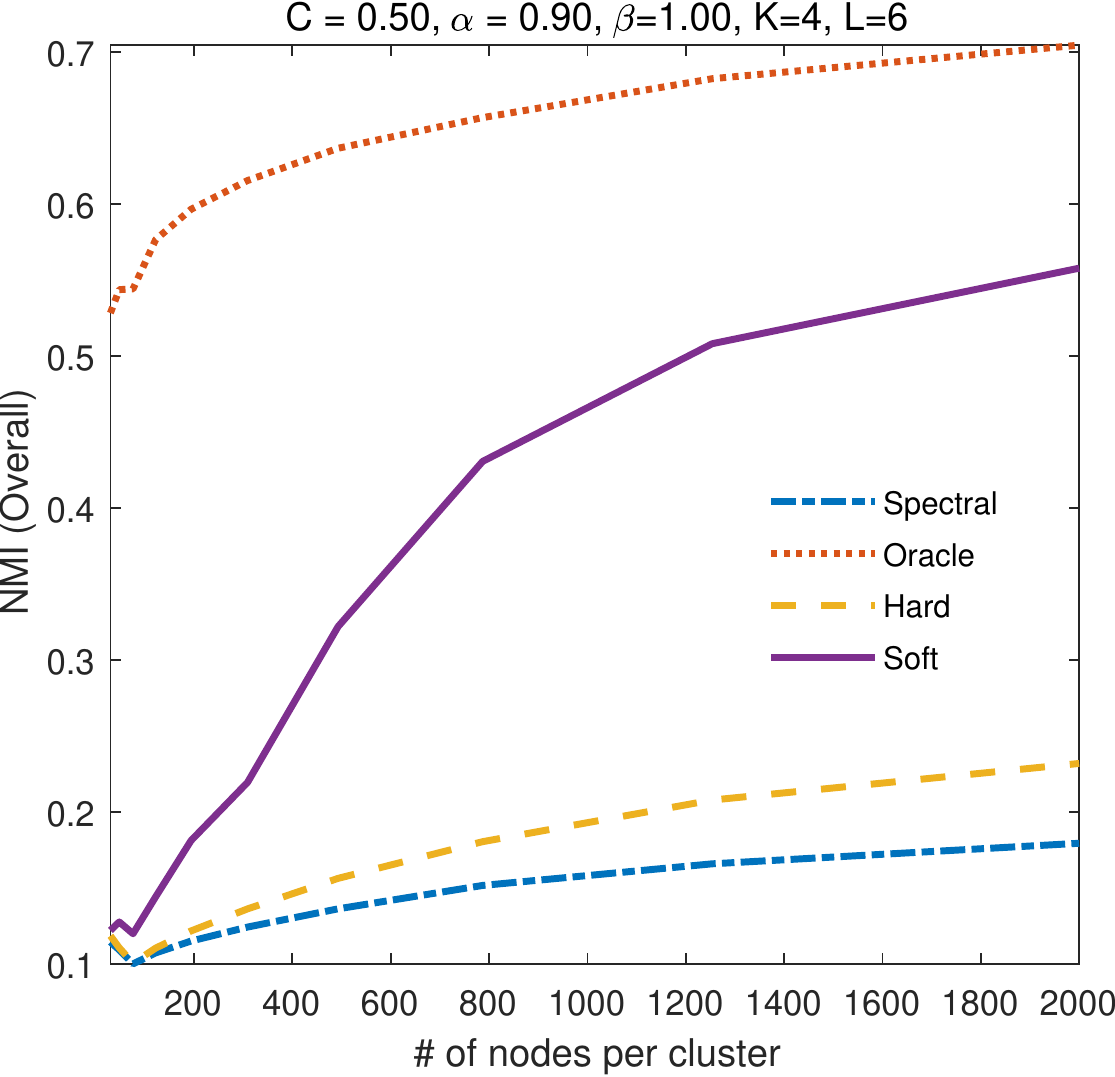} &
		\includegraphics[width=2.5in]{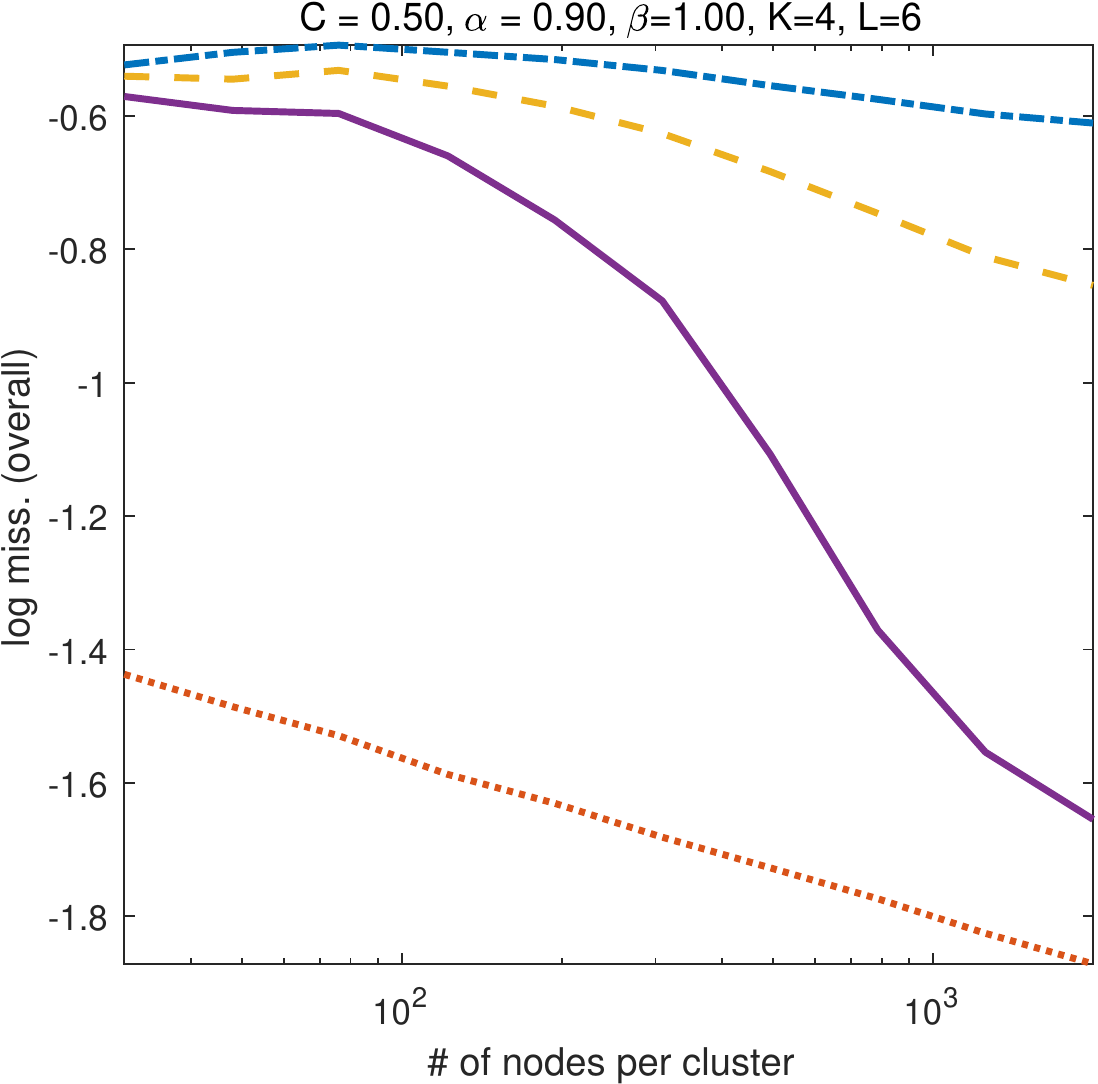} 
	\end{tabular}
\end{figure}
The next figure illustrates the results for unbalanced cluster sizes. To be specific, 
\begin{align*}
\pi(\y)=\f{(1,4,6,9)}{20} \quad \text{and} \quad \pi(\z)=\f{(1,3,4,6,7,9)}{30},
\end{align*}
which implies $\beta\ge 3$ according to \eqref{assump:balance}:
\begin{figure}[h]
	\centering
	\begin{tabular}{cc}
		\includegraphics[width=2.5in]{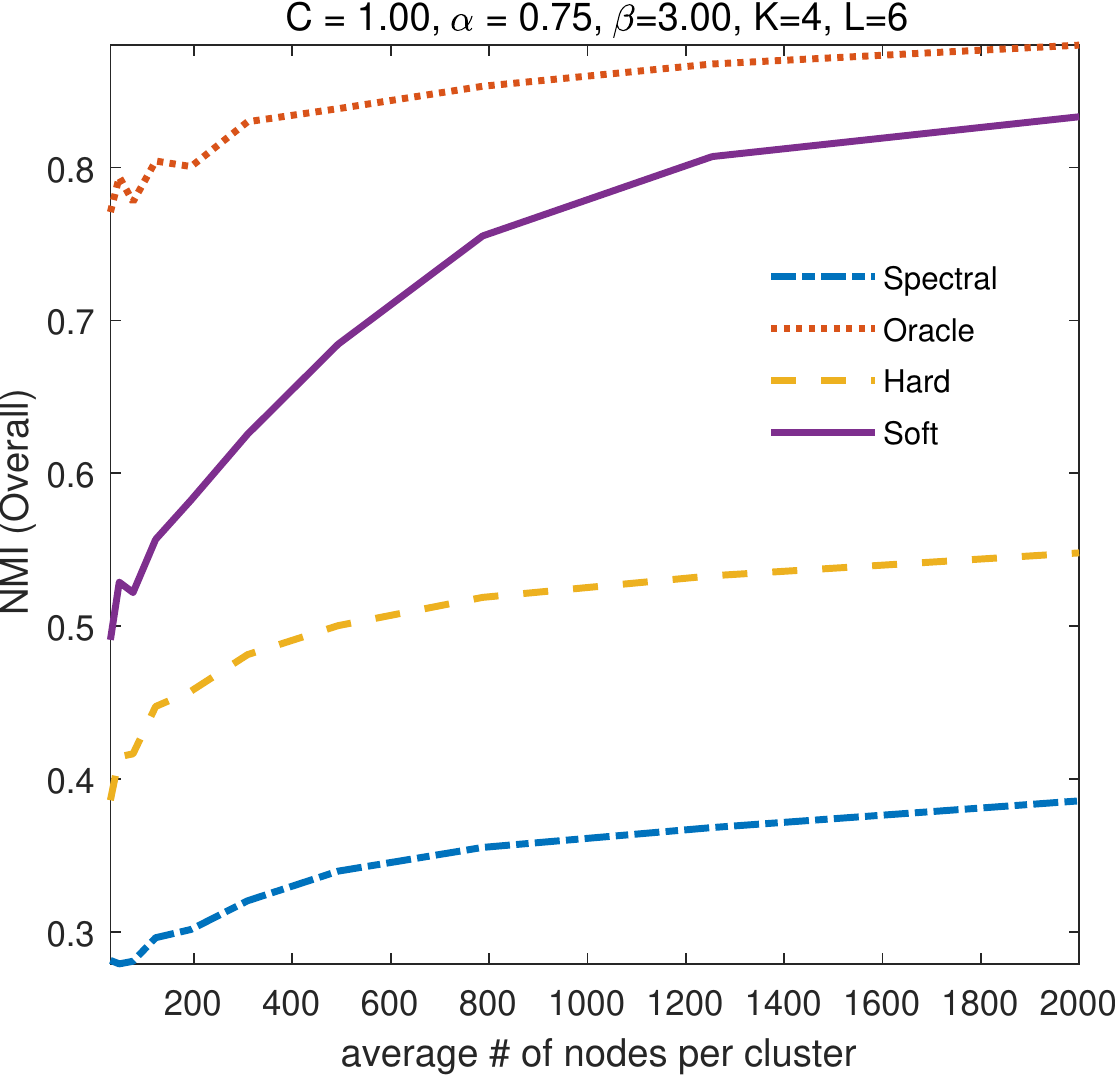} &
		\includegraphics[width=2.5in]{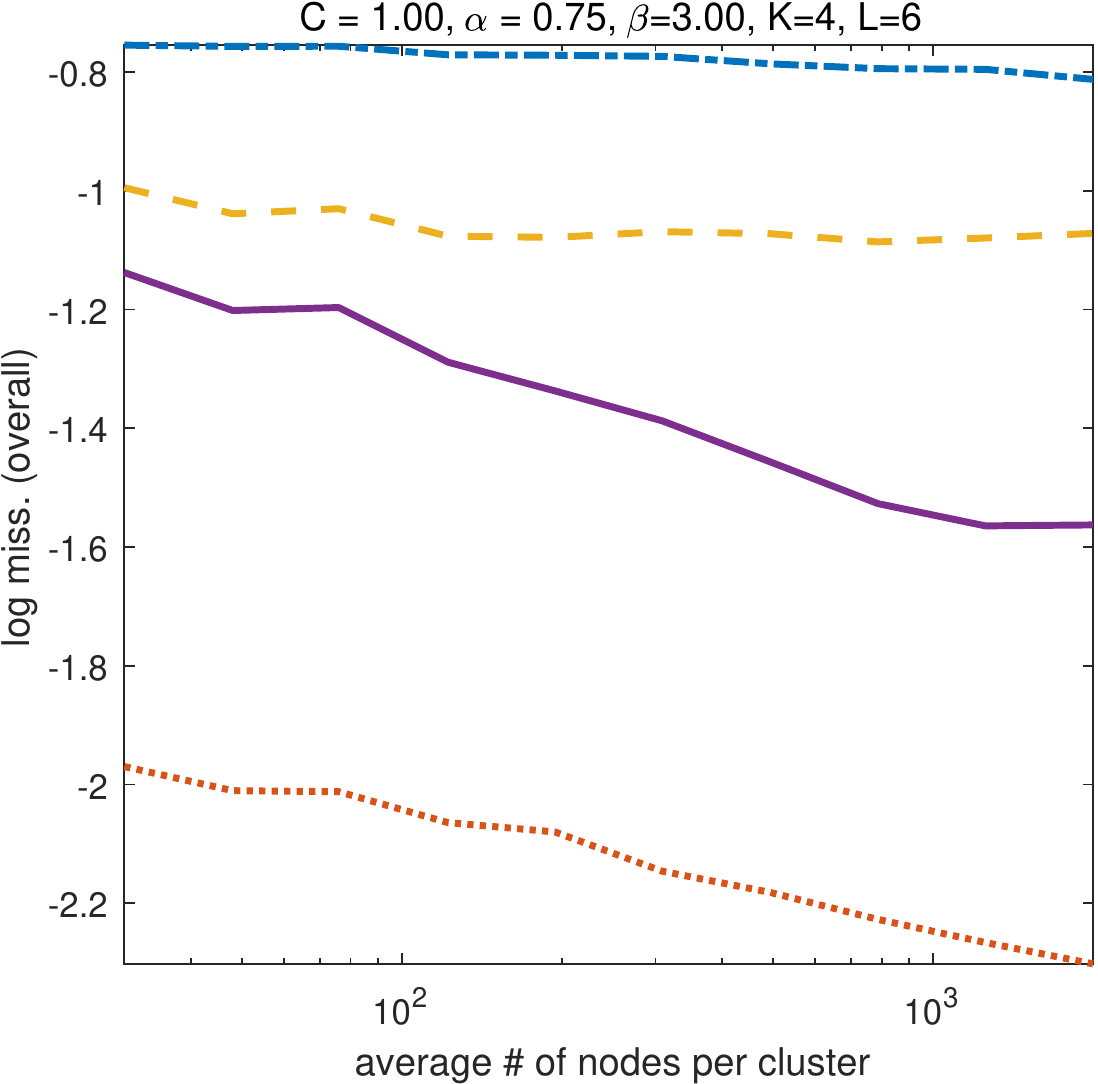}
	\end{tabular}
\end{figure}
We also consider the setting where the number of the clusters of one side is significantly greater that of the other. We let $\Kr = 4$, $\Kc = 12$ and
\begin{align*}
B =
\left[ \begin{array}{@{}*{12}{c}@{}}
1  &2  &3  &4  &5  &6  &7  &8  &9  &10 &11 &12\\
4  &5  &6  &7  &8  &9  &10 &11 &12 &1  &2  &3\\
7  &8  &9  &10 &11 &12 &1  &2  &3  &4  &5  &6\\
10 &11 &12 &1  &2  &3  &4  &5  &6  &7  &8  &9
\end{array} \right].
\end{align*}
The simulation results are as follows:
\begin{figure}[h]
	\centering
	\begin{tabular}{cc}
		\includegraphics[width=2.5in]{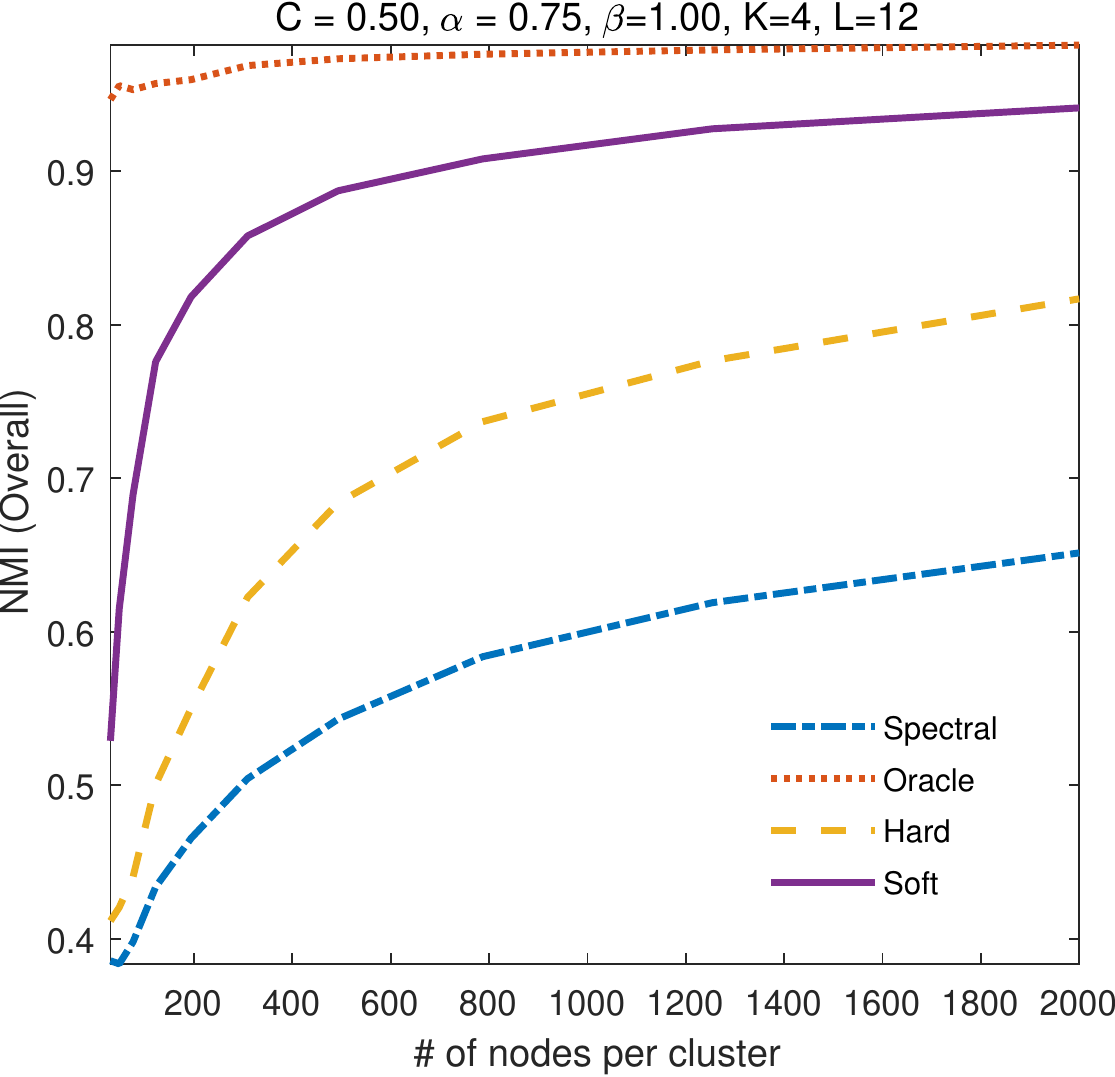} &
		\includegraphics[width=2.5in]{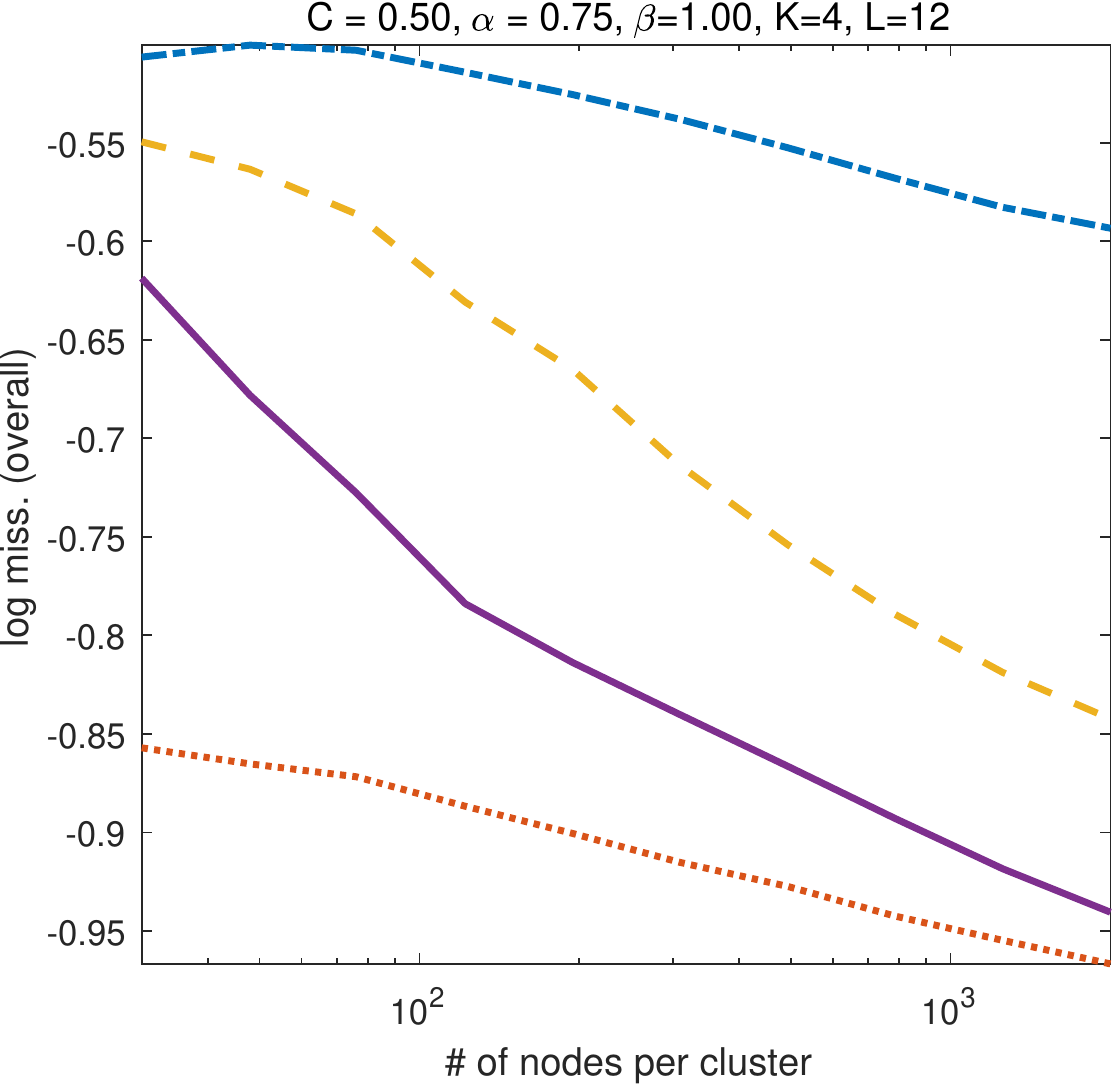} 
	\end{tabular}
\end{figure}



\end{document}